%% file: fpc35.tex
\documentclass[12pt,a4paper,final]{amsart}

\usepackage{color}
\usepackage[OT1]{fontenc}    %
\usepackage{type1cm}         %
\usepackage{amsthm}          %
\usepackage{amsmath}
\usepackage{amssymb}
\usepackage{pstricks-add}
\usepackage{hyperref}
\usepackage[dvips]{graphicx} 
\usepackage[bf]{caption}     
\usepackage{psfrag}          
\usepackage[top=1.5cm, bottom=1.5cm]{geometry}        
\usepackage{chngcntr}
\counterwithin{figure}{section}

\usepackage{bm}  
\usepackage{buczofpc}
\usepackage{version} 
\usepackage[notcite,notref]{showkeys} 
\usepackage{datetime}

\definecolor{aquam}{rgb}{0.5,1.0,1.0}
\definecolor{bbrown}{rgb}{0.75,0.38,0.15}
\definecolor{Cyan}{rgb}{0,0.6,0.6}
\definecolor{Darkblue}{rgb}{0,0,1}
\definecolor{Dodgerblue2}{rgb}{0,0.5,1}
\definecolor{Green}{rgb}{0,0.6,0.06}
\definecolor{Kahki}{rgb}{1,1,0.5}
\definecolor{Magenta}{rgb}{1,0,1}
\definecolor{bMagenta}{rgb}{1,.6,1}
\definecolor{Orange}{rgb}{0.8,0.3,0}
\definecolor{dOrchid}{rgb}{0.7,0.2,0.4}
\definecolor{Orchid}{rgb}{1,0.5,1}
\definecolor{Purple}{rgb}{0.65,0.07,0.85}
\definecolor{Royalblue}{rgb}{0.6,0.85,0.87}
\definecolor{Tan}{rgb}{0.54,0.42,0.23}
\definecolor{bTan}{rgb}{0.94,0.82,0.63}
\definecolor{zoltan}{rgb}{0,0.1,0.3}
\definecolor{Turquoise}{rgb}{0,0.85,0.87}
\definecolor{Yellow}{rgb}{1,1,0}
\definecolor{darkamber}{rgb}{0.4,0.19,0.28}
\definecolor{bYellow}{rgb}{1,1,0.6}
\definecolor{bRed}{rgb}{1,0.7,0.7}
\definecolor{boxcolb}{rgb}{0.87,0.77,0.75}
\definecolor{boxcol}{rgb}{0.6,0.85,0.87}
\definecolor{boxcolgreen}{rgb}{0.64,0.93,0.79}
\definecolor{boxcolaa}{rgb}{.75,.99,.70}
\definecolor{boxcolbb}{rgb}{0.39,0.50,0.56}
\definecolor{boxcolcc}{rgb}{1,0.81,0.65}
\definecolor{yy}{rgb}{0.43,0.21,.18}
\definecolor{gA}{gray}{0.5}
\definecolor{gB}{gray}{0.8}
\definecolor{gC}{gray}{0.9}

\numberwithin{equation}{section}

\theoremstyle{plain}
\newtheorem{theorem}{Theorem}[section]
\newtheorem{corollary}[theorem]{Corollary}
\newtheorem{proposition}[theorem]{Proposition}
\newtheorem{lemma}[theorem]{Lemma}

\newtheorem{observation}[theorem]{Observation}

\theoremstyle{remark}
\newtheorem{remark}[theorem]{Remark}

\newtheorem{example}[theorem]{Example}

\newtheorem{question}[theorem]{Question}

\theoremstyle{definition}
\newtheorem{definition}[theorem]{Definition}
\newtheorem{construction}[theorem]{Construction}
\newtheorem{algorithm}[theorem]{Algorithm}

\def\B{{\mathcal B}}
\def\H{{\mathcal H}}
\def\L{{\mathcal L}}

\def\F{{\mathcal F}}
\def\E{{\mathbb E}}
\def\P{{\mathbb P}_p}
\def\Q{{\mathcal Q}}
\def\C{{\mathcal C}}
\def\S{{\mathcal S}}
\def\K{{\mathcal K}}
\def\R{{\mathbb R}}

\def\N{{\mathbb N}}
\def\T{{\mathcal T}}
\def\e{{\varepsilon}}
\def\ii{{\bf i}}

\def\bL{{\bf L}}

\newcommand{\LL}[2]{L^{#1}(m_1,k_0)_{#2}}
\newcommand{\DD}[2]{\Delta^{#1}_{#2}}

\newcommand{\mfnn}{\scalebox{2}[2]{${\mathfrak n}$}}
\newcommand{\mfn}{\mathfrak n}
\newcommand{\mfc}{\mathfrak c}

\DeclareMathOperator{\dimh}{dim_H}

\DeclareMathOperator{\dist}{dist}

\DeclareMathOperator{\diam}{diam}

\DeclareMathOperator{\inter}{Int}

\DeclareMathOperator{\dmax}{d_{max}}
\DeclareMathOperator{\dH}{d_H}

\DeclareMathOperator{\lpb}{LP_b}
\DeclareMathOperator{\lpi}{LP_{i-r}}
\DeclareMathOperator{\lp1}{LP_{1-r}}

\DeclareMathOperator{\interm0}{Int_{m_0}}

\begin{document}

\title[Fractal percolation is unrectifiable]
{Fractal percolation is unrectifiable}

\author[Z. Buczolich]{Zolt\'an Buczolich$^1$}
\address{Department of Analysis, ELTE E\"otv\"os Lor\'and University, 
         P\'azm\'any P\'eter s\'e\-t\'any 1/c, H-1117 Budapest, Hungary$^{1,4}$}
\email{zoltan.buczolich@ttk.elte.hu$^1$}

\author[E. J\"arvenp\"a\"a]{Esa J\"arvenp\"a\"a$^2$}
\address{Department of Mathematical Sciences, P.O. Box 3000,
         90014 University of Oulu, Finland$^{2,3,5}$}
\email{esa.jarvenpaa@oulu.fi$^2$}

\author[M. J\"arvenp\"a\"a]{Maarit J\"arvenp\"a\"a$^3$}
\email{maarit.jarvenpaa@oulu.fi$^3$}

\author[T. Keleti]{Tam\'as Keleti$^4$}
\email{tamas.keleti@gmail.com$^4$}

\author[T. P\"oyht\"ari]{Tuomas P\"oyht\"ari$^5$}
\email{tuomaspo@gmail.com$^5$}


\thanks{ZB received support from the Hungarian National Research, Development
and Innovation Office--NKFIH, Grant K124003, and
KT received support from the Hungarian
National Foundation for Scientific Research, grant K124749.
ZB and KT also thank the R\'enyi Institute for the visiting researcher
positions during the fall semester of 2017.
EJ, MJ and TP acknowledge the support of the Centre of Excellence in
Analysis and Dynamics Research funded by the Academy of Finland. EJ and MJ 
thank the ICERM semester program on ``Dimension and Dynamics''. This material is
partially based upon work supported by the Swedish Research Council under grant
no. 2016-06596 while the authors were in residence at Institut Mittag-Leffler in
Djursholm, Sweden during the Fall semester of 2017.}

\maketitle

\begin{abstract} 
We show that there exists $0<\alpha_0<1$ (depending on the parameters) such 
that the fractal percolation is 
almost surely purely $\alpha$-unrectifiable for all $\alpha>\alpha_0$.
\end{abstract}

\setcounter{tocdepth}{3}
\tableofcontents


\section{Introduction}\label{intro}

Fractal percolation, also known as Mandelbrot percolation, is a 
classical random process introduced by Mandelbrot in 1974
for the purpose of modelling turbulence  
\cite{Man}. Mandelbrot called the model canonical curdling whereas the name
fractal percolation was established for the process later.
We begin by describing the model briefly 
and refer to Section~\ref{model} for more precise definitions.   

Fix $0\le p\le 1$, and let $N\in\N:=\{0,1,2,\dots\}$ with $N\ge 2$.
Letting $d\in\N\setminus\{0\}$, construct a random compact subset $E$ of the
unit cube $Q_0:=[0,1]^d\subset\R^d$ in the following manner:
Divide $Q_0$ into $N^d$ subcubes of equal size. Independently of 
each other, each of them is chosen with probability $p$ and deleted with
probability $1-p$, and the collection of the chosen subcubes is denoted by
$\C_1$. Continue by repeating the same process for each $Q\in\C_1$. The set of
all chosen cubes at the second level is denoted by $\C_2$. Iterating
this process inductively gives the fractal percolation set $E$, defined as
\[
E:=\bigcap_{n=1}^\infty \bigcup_{Q\in\C_n}Q.
\]
The probability space $\Omega$ is the space of all constructions and
the natural probability measure on $\Omega$ induced by this procedure is
denoted by $\P$.

We shortly describe some basic properties of fractal percolation set relevant
to our purposes and refer to \cite{C} or \cite{G} for further information.
It is clear that $E=\emptyset$ with positive probability if $p<1$, since 
$\C_1=\emptyset$ with probability $(1-p)^{N^d}$. It follows from 
the theory of branching processes that $E=\emptyset$ almost surely if the
expected number of chosen cubes is at most one, that is, $p\le N^{-d}$. 
Kahane and Peyri\`ere \cite{KP} proved that, in the opposite case
$p>N^{-d}$, the Hausdorff dimension, $\dimh$, of the limiting set is almost
surely
a constant conditioned on non-extinction, that is,
\begin{equation}\label{dimH}
\dimh E=\frac{\log(pN^d)}{\log N}
\end{equation}
almost surely conditioned on the event $E\ne\emptyset$. 

In \cite{CCD}, J.T. Chayes, L. Chayes and R. Durrett verified in the case 
$d=2$ that there is
a critical probability $0<p_c<1$ such that if $p<p_c$, then $E$ is totally 
disconnected with probability one, whereas the opposing sides of $Q_0$ are 
connected by a connected component of $E$ with positive probability provided 
that $p\ge p_c$. The latter phenomenon is commonly referred to as fractal
percolation. The exact value of $p_c$ is not known. From \eqref{dimH}
it trivially follows that $p_c\ge N^{-1}$ if $d=2$, and in 
\cite{CCD} it is proved that $N^{-1}<p_c<1$. Corresponding results are
apparently valid also for $d>2$ (see \cite{FG}).

Even though $p_c>N^{-1}$, the set $E$ looks connected from outside as soon
as its dimension is larger than one. Indeed, in \cite{FG}, Falconer and 
Grimmett proved that in this case the coordinate projections of $E$ contain an 
interval almost surely conditioned on non-extinction. Further, Rams and Simon
\cite{RS1,RS2} showed that almost surely all projections of $E$ contain an
interval simultaneously if $\dimh E>1$. This result also follows from \cite{PR}.
Finally, almost surely all visible parts
of $E$ are 1-dimensional (see \cite{AJJRS}).

By the above mentioned result of \cite{CCD},  conditioned on non-extinction,
$E$ contains almost surely a non-trivial 
connected component as soon as $p\ge p_c$. It 
is natural to ask whether $E$ contains a non-trivial path connected component
in this case. This was answered positively by Meester in \cite{Me}. As far as
the regularity of paths contained in $E$ is concerned, Chayes showed in
\cite{C2} that the lower box counting dimension of any path contained in $E$ is
strictly larger than 1 with a bound depending on 
the parameters $p$ and $N$. Thus, $E$ does not contain uniform $\alpha$-H\"older
curves for $\alpha$ close to 1. In particular, $E$ does not contain any
rectifiable curves. An explicit lower bound for the lower box counting 
dimension of the non-trivial curves contained in $E$ was given by Orzechowski
in \cite{O2}. In \cite{O1}, he proved that $E$ contains non-trivial curves
whose upper box counting dimension is strictly less than 2. Again, there is 
an explicit expression for the upper bound.

Broman et al.~\cite{BCJM} showed that, in the case $d=2$ and $p\ge p_c$,
the set $E$ can be decomposed as $E=E^c\cup E^d$, where $E^d$ is the totally
disconnected part of $E$ and $E^c$ consists of non-trivial connected components
of $E$. Moreover, $\dimh E^c<\dimh E^d=\dimh E$ and there exists $0<\beta<1$, 
depending on the parameters, such that $E^c$ is an uncountable union of 
non-trivial $\beta$-H\"older curves. 

In this paper, we supplement the result of Broman et al. We first
define a concept of $\alpha$-unrectifiability: given $0<\alpha\le 1$, a set 
$A\subset\R^d$ is purely $\alpha$-unrectifiable if 
$\H^{\frac 1\alpha}(A\cap\gamma([0,1]))=0$ for all $\alpha$-H\"older curves 
$\gamma\colon [0,1]\to\R^d$, where $\H^s$ is the $s$-dimensional Hausdorff
measure.  We note that related fractional rectifiability questions have
been studied for example by Mart\'{\i}n and Mattila for deterministic sets in 
\cite{MM1,MM2,MM3} and, quite recently, by Badger, Naples and
Vellis for measures in \cite{BNV,BV}.  In our main theorem
(Theorem~\ref{maingeneral}),  we verify that, for every $0\le p<1$ and
$N,d\in\N\setminus\{0,1\}$, there exists $\alpha_0<1$ such that almost surely
the fractal percolation set $E$ is purely $\alpha$-unrectifiable for all
$\alpha_0<\alpha\le 1$. Since the case $\alpha=1$ corresponds to standard
$1$-unrectifiability, our result implies that 
$E$ is almost surely purely $1$-unrectifiable and, thus,
purely $k$-unrectifiable for all $k\in\N$. In Section~\ref{standard}
(see Theorem~\ref{mainstandard}), we give a simpler proof than that of our
main theorem for $1$-unrectifiability. The
general case, requiring new tools, is considered in Section~\ref{general}.
We believe that these new tools turn out to be useful in many other problems
related to the fractal percolation and other random geometric constructions.

The paper is organised as follows.  In Section~\ref{idea}, we explain the
heuristic idea of the proof of our main result.   
In Section~\ref{model}, we define the fractal
percolation model by introducing a slightly different viewpoint than the
standard one described earlier in this section but leading to same
probabilities. In Section~\ref{standard}, we introduce basic concepts, prove
preliminary results and give a short proof
for the 1-unrectifiability of the fractal percolation in
Theorem~\ref{mainstandard} which, in turn, implies the $k$-unrectifiability,
see Corollary~\ref{kunrectifiable}. 
Section~\ref{hereditarilyexists} is concerned with probability estimates
guaranteeing the existence and abundance of hereditarily good cubes, 
 one  of the key concepts in our paper.  In these cubes 
holes generated during the fractal percolation process are ``sufficiently
uniformly'' located.  In  Section~\ref{general},  we study
$\alpha$-unrectifiability and prove our main result: Theorem~\ref{maingeneral}.
The proof is based on Proposition \ref{notalphaholder}, which is our main tool
guaranteeing the  length increase in hereditarily good cubes for broken
line approximations of curves staying close to the fractal percolation set.
 The proof of Proposition~\ref{notalphaholder} requires several  quite
technical definitions, algorithms and results. These are collected in three
Appendices.

In Section \ref{appendixa}  (Appendix A),  we construct the special
sequences used in Section~\ref{brokenlinesection}  to define correct zoom
levels.  Appendix B, that is, Section~\ref{brokenlinesection}, is dedicated
to our new tool consisting of several algorithms utilised to construct
special broken line approximations of curves staying close to the fractal
percolation set. Finally, in Appendix C, that is, in
Section~\ref{lengthincrease}, we  verify  growth estimates
for the length of the broken line approximations defined in
Section~\ref{brokenlinesection} and prove Proposition~\ref{notalphaholder}.

\section{Idea of the proof}\label{idea}

In this section, we describe the heuristic idea behind the proof of our main
result concerning the $\alpha$-unrectifiability of fractal percolation. We
also  describe  some subtleties encountered while making this idea rigorous.

The length of a smooth curve can be calculated by constructing finer and finer
broken line approximations of the curve and  by  taking the limit of their lengths.
In particular, the lengths of the broken line approximations are uniformly
bounded. If the curve is  genuinely  $\alpha$-H\"older continuous, the lengths of the
broken line approximations typically tend to infinity  and  the exponent $\alpha$
controls how fast this may happen. Indeed, let $\gamma\colon [a,b]\to\R^d$
be $\alpha$-H\"older continuous. Suppose that, for some integer $L$ and some
number $r>0$, one can find points $a=:a_1<a_2<\dots<a_L<a_{L+1}:=b$ such that
\begin{equation}\label{expgrowth}
(1+r)|\gamma(a)-\gamma(b)|\le\sum_{i=1}^L|\gamma(a_i)-\gamma(a_{i+1})|.
\end{equation}
Assuming that this process can be iterated $q$ times, there are points
$a=:a_1<a_2<\dots<a_{L^q}<a_{L^q+1}:=b$ such that
\[
(1+r)^q|\gamma(a)-\gamma(b)|\le\sum_{i=1}^{L^q}|\gamma(a_i)-\gamma(a_{i+1})|.
\]
Using $\alpha$-H\"older continuity and Jensen's inequality on the right hand
side, one obtains
\begin{equation}\label{idealeq}
(1+r)^q|\gamma(a)-\gamma(b)|\le C\sum_{i=1}^{L^q}|a_i-a_{i+1}|^\alpha
  \le CL^q(L^{-q}|a-b|)^\alpha.
\end{equation}
If this is true for large $q$,  we have  that $1+r\le L^{1-\alpha}$, that is,
the exponential growth rate of the broken line approximation is controlled by
the H\"older exponent. In particular, if the above inequality is valid for some
fixed $r$ and $L$,  the exponent  $\alpha$ cannot be very close to 1.

The idea for the existence of $r$ is as follows: If the intersection
of a curve with the fractal percolation set has positive measure, the density
point theorem implies that, at small scales, there cannot be big gaps, that
is, parts of the curve whose intersection with the set is empty. On the other
hand, there are a lot of holes at many scales in the percolation set.
 Hence,  in
order to avoid gaps, the curve has to go around the holes,  increasing 
the length compared to a straight line.

 For the purpose of iterating  Inequality~\eqref{expgrowth}, one needs to find
the fixed relative increase of length at many successive scales and at various
different places, that is, for all subcurves.  There  are two competing
phenomena: the size of a hole and the probability of its existence. In order
to have many successive scales with holes,  high probability for the
existence of a hole is needed.  This is obtained by decreasing the size of
a hole,  which , in turn, makes $r$ smaller.

There is also a trade-off between  $L$  and probabilities.
The larger $L$, the easier  to  obtain a length increase of $1+r$, but
the harder  to  iterate, since  there are more  subcurves which
should satisfy suitable conditions.

The choice of $q$ is also delicate. Taking large $q$ makes the scales very
small, which implies that the holes and, therefore, gaps become very small.
 Therefore, extremely small scales are required when using the density point
argument. It turns out that we need to choose  the points $a$ and $b$ in
\eqref{idealeq} such that they depend on $q$. To overcome this difficulty, we
introduce a concept called tightness (see Definition~\ref{tight}).

Finally, we note that the order of quantifiers is tricky. Due to the definition
of unrectifiability, one has to show that, almost surely, the intersection of
the fractal percolation set with all $\alpha$-H\"older curves has zero measure.

\section{Fractal percolation model}\label{model}

Letting $d\in\N\setminus\{0,1\}$ (the case $d=1$ is trivial for our purposes),
we begin by describing the underlying probability space related to the fractal
percolation set in $\R^d$. Fix $N\in\N\setminus\{0,1\}$. Let 
$\T$ be the rooted $N^d$-branching tree and set
$\Omega:=\{0,1\}^{v(\T)}=\{\omega\mid\omega\colon v(\T)\to\{0,1\}\}$, where
$v(\T)$ is the set of vertices of $\T$. Let 
$J:=\{1,\dots,N^d\}$. The vertices of $\T$ may be naturally encoded by
finite words with letters in $J$, that is, by elements of 
$\bigcup_{n=0}^\infty J^n$, where the root corresponds to the empty word
$\emptyset$ and
the vertices whose distance to the root is $n$ are coded by the words
$\ii:=i_1\cdots i_n$ of length $n$, where $i_j\in J$ for all $j=1,\dots,n$. 
We denote the length of a word $\ii$ by $|\ii|$ and
define a metric $\rho$ on $\Omega$ by setting 
$\rho(\omega,\omega'):=N^{-|\omega\wedge\omega'|}$, where 
\[
|\omega\wedge\omega'|:=\min\{n\in\N\mid\text{ there exists }\ii\in J^n
  \text{ with }\omega(\ii)\ne\omega'(\ii)\}.
\]
For $0\le p\le 1$, define a Borel probability measure $\P$ on $\Omega$ 
by 
\[
\P:=((1-p)\delta_0+p\delta_1)^{v(\T)},
\]
where $\delta_k$ is the Dirac measure at $k$.

We consider the probability space
$(\Omega,\mathcal B,\P)$, where $\mathcal B$ is the completion of the Borel 
$\sigma$-algebra. Every $\omega\in\Omega$ defines a
fractal percolation set $E(\omega)\subset\R^d$ as follows: For $n\in\N$, let 
\[
\Q_n:=\Big\{\prod_{i=1}^d[(l_i-1)N^{-n},l_iN^{-n}]\mid l_i=1,\dots,N^n \text{ and } 
        i=1,\dots,d\Big\}
\]
be the collection of grid cubes of $Q_0:=[0,1]^d$ with side length $N^{-n}$.
{\it The level of a cube} $Q\in\Q_n$ is $n$. Enumerating the elements of $\Q_1$
by $J$ and using the same enumeration for 
the subcubes of $Q$ belonging to $\Q_{n+1}$ for all $Q\in\Q_n$,
we define a natural
bijection between $J^n$ and $\Q_n$. The image of $\ii\in J^n$ under this 
bijection is denoted by $Q_\ii$. Given $\omega\in\Omega$, a cube $Q_\ii\in\Q_n$
is {\it chosen} if
$\omega(\ii)=1$ and {\it deleted} if $\omega(\ii)=0$. The set of chosen cubes
in $\Q_n$ is denoted by $\C_n(\omega):=\{Q_\ii\in\Q_n\mid\omega(\ii)=1\}$. For 
every $\omega\in\Omega$, we define the fractal percolation set $E(\omega)$ by
\[
E(\omega):=\bigcap_{n=0}^\infty\bigcup_{Q\in\C_n(\omega)}Q.
\]
Note that $E(\omega)\ne\emptyset$ if and only if there exists an infinite 
subtree $T\subset\T$ rooted at  $\emptyset$  such that $\omega(\ii)=1$
for all vertices $\ii$ of $T$. In particular, $E(\omega)$ may be identified
with the infinite component (determined by the condition $\omega(\ii)=1$) of
$\T$ containing the root.

\begin{remark}\label{unconventional}
In this section, we have chosen a slightly different viewpoint while defining
the probability space $\Omega$ than the standard one described in the
introduction. Indeed, the set $\C_n(\omega)$ depends only on $\omega(\ii)$
with $|\ii|=n$ and, therefore, the sequence $(\C_n)_{n\in\N}$ is not nested.
This is merely a notational convention, which does not change the probabilities
related to the fractal percolation sets $E(\omega)$, but turns out to be useful
for our purposes.
\end{remark}

\section{Pure 1-unrectifiability}\label{standard}


We begin  with  some notation used throughout the paper. For all
$i\in\{1,\dots,d\}$, the orthogonal projection onto the $i$-th coordinate axis
is denoted by $\Pi_i$. For all $x,y\in\R^d$, $L(x,y)$ is the line segment
connecting $x$ and $y$. The complement of a set $A\subset\R^d$ is denoted by
$A^c$ and $\inter A$ refers to the interior of $A$. Finally, set
\begin{equation}\label{restrictedQn}
\Q_n(A):=\{Q\in\Q_n\mid Q\subset A\}
\end{equation}
for all $A\subset\R^d$ and $n\in\N$.

\begin{definition}\label{properly}
We say that $i\in\{1,\dots,d\}$ is a {\it principle direction}  for  a line
$\ell\subset\R^d$ 
if
$j=i$ maximises the length of $\Pi_j(s)$ for subsegments $s$ of $\ell$.
 Observe that for certain lines $\ell $ the principle direction is not unique.
 For all
$n\in\N$, we say that a line $\ell\subset\R^d$ {\it intersects a cube}
$Q\in\Q_n$ {\it properly} if $\H^1(\Pi_i(\ell\cap Q))\ge d^{-1}N^{-n}$, where $i$
is a principle direction  for  $\ell$.
For any $n\in\N$ and $i\in\{1,\dots,d\}$,   
we say that a line $\ell\subset\R^d$ {\it intersects a cube}
$Q\in\Q_n$ {\it very properly in direction $i$} if 
$\Pi_i(\ell\cap Q)$ contains an interval
of the form $[k d^{-1}N^{-n},(k+1) d^{-1}N^{-n}]$
for some integer $k$.

A set $\L\subset\R^d$ is an {\it $i$-layer} if $\L=\Pi_i^{-1}([a,b])$ for some
$a,b\in\R$ with $a\le b$, and $\L$ is a {\it layer} if it is an $i$-layer for
some $i\in\{1,\dots,d\}$. A set $\L\subset\mathbb R^d$ is an
{\it $(n,i)$-layer} if $\L=\Pi_i^{-1}([kN^{-n},(k+1)N^{-n}])$ for some  
$k\in\{0,\dots,N^n-1\}$ and it is an {\it $(n,i)$-double-layer} if
$\L=\Pi_i^{-1}([kN^{-n},(k+2)N^{-n}])$.

Let $a,b\in\R$ with $a<b$. We say that a curve $\gamma\colon [a,b]\to\R^d$
{\it passes through a layer} $\L$, if $\gamma(a)$ and $\gamma(b)$ belong to
different connected components of $(\inter\L)^c$.  For $x,y\in\R^d$, the
line segment 
$L(x,y)$ passes  through $\L$, if $\gamma\colon [0,1]\to\R^d$,
$\gamma(t):=x+t(y-x)$,  passes through $\L$.

We say that cubes $Q,Q'\in\Q_n$ are {\it neighbours} if
$Q\cap Q'\ne\emptyset$. In this case,  we use the notation $Q\sim Q'$.
Cubes $Q,Q'\in\Q_n$ are {\it $i$-neighbours}, denoted by $Q\sim_i Q'$, if
$Q\sim Q'$ and $Q,Q'\subset\L$ for some $(n,i)$-layer $\L$.
\end{definition}


The above definition of intersecting properly is motivated by the
following geometric observation.

\begin{observation}\label{properintersection}
Fix $i\in\{1,\dots,d\}$, $n\in\N$ and an $(n,i)$-layer $\L$. Let $x,y\in\R^d$
and assume that
\begin{itemize}
\item[(a)] the line segment $L(x,y)$ passes through $\L$ and
\item[(b)] for all $j\ne i$ there are $(n,j)$-layers $\L_j^1$ and $\L_j^2$ such
  that $L(x,y)\subset\L_j^1\cup\L_j^2$.
\end{itemize}
Let $\ell$ be the line containing $L(x,y)$.
Then there exists $Q\in\Q_n(\L)$ with the 
following properties:
\begin{itemize}
\item[(i)] 
The line $\ell$ intersects $Q$ very properly in 
direction $i$.
 \item[(ii)] The line $\ell$ 
 intersects $Q$ properly and
  $\ell\cap Q=L(x,y)\cap Q$.
\item[(iii)] 
  The cube $Q$ is a neighbour to every cube belonging to $\Q_n(\L)$ that is
  intersected by $L(x,y)$ in a set of positive length.
\end{itemize}
\end{observation}

\begin{proof}
 Define $\gamma\colon [0,1]\to\R^d$ by $\gamma(t):=x'+t(y'-x')$, where
$\{x',y'\}=L(x,y)\cap\partial\L$ and the boundary of a set $A$ is denoted by
$\partial A$.  By (b), the line segment $L(x,y)$ intersects at most two
different cubes in a set of positive length in each direction $j\ne i$. Note
that along $L(x,y)$ one moves from one cube to another one at most once in each
coordinate direction $j\ne i$.  Therefore, there are points
$0<t_1<\dots<t_m<1$ with $m\le d-1$ satisfying the property that, for all
$l=0,\dots,m$, there is $Q\in\Q_n(\L)$ such that
$\gamma([t_l,t_{l+1}])\subset Q$.  Hence if we subdivide the $(n,i)$-layer
$\L$ into $d$ parallel $i$-layers  $\widetilde\L_1,\dots,\widetilde\L_d$
 of equal width then  there exist $l,k\in\{0,\dots,d-1\}$ such that
$\gamma|_{[t_l,t_{l+1}]}$ passes through $\widetilde\L_k$,  which implies (i).
The first statement in (ii)
is true since a principle direction maximises the length of the projection of a
line segment, and the second claim follows from assumption (a).
Finally, (iii) follows from (b).
\end{proof}

The next definition deals with random concepts, that is, concepts which
depend on $\omega\in\Omega$.

\begin{definition}\label{good}  
Let $m_0\in\mathbb N\setminus\{0\}$, and fix $\omega\in\Omega$.
For all $n\in\N$, a cube $Q'\in\Q_{n+m_0}$ is {\it strongly $i$-deleted} if
$\widetilde Q\in\Q_{n+m_0}\setminus\C_{n+m_0}(\omega)$ for all
$\widetilde Q\sim_i Q'$, see Figure~\ref{fig33}. A cube $Q'\in\Q_{n+m_0}$
is {\it weakly $i$-chosen} 
if it is not strongly $i$-deleted, that is, either $Q'$ or at least one of its 
neighbours in the same $(n+m_0,i)$-layer is chosen.

Let $\ell\subset\R^d$ be a line with a principle direction $i\in\{1,\dots,d\}$
and let $Q\in\Q_n$ be such that $\ell\cap Q$ has positive length. Let
$\L_1,\dots,\L_k$ be the $(n+m_0,i)$-layers (in the natural order) for which
$\ell\cap Q\cap\L_j$ has positive length. We say that  $Q$  is
{\it $m_0$-good for the line $\ell$} if there is a strongly $i$-deleted cube
$Q'\in\Q_{n+m_0}(Q)$ so that
\begin{itemize}
\item[(a)] $Q'\not\subset\L_j$ for $j\in\{1,2,k-1,k\}$ and
\item[(b)] $\ell$ intersects $Q'$ properly.
\end{itemize}
A cube $Q\in\Q_n$ is {\it $m_0$-bad for a line $\ell$} if it is not
$m_0$-good for $\ell$.

For a $(d-1)$-dimensional face $F$ of $Q\in\Q_n$, we denote by $G_{m_0}(F)$
the grid points of the natural grid of side length $N^{-n-m_0}$ in $F$. For all
$Q\in\Q_n$, define a collection of lines 
\[
\Gamma(Q,m_0):=\bigcup_{F}\bigcup_{v}\Gamma_{F,v}(m_0),
\]
where the first union is over all $(d-1)$-dimensional faces $F$ of $Q$, the 
second one is over all vertices of $Q$ not contained in $F$ and
$\Gamma_{F,v}(m_0)$ is the collection of all lines $\ell$ such that $v\in\ell$
and $\ell\cap G_{m_0}(F)\ne\emptyset$, see Figure~\ref{fig33}. 
Note that if $F$ is perpendicular to the $i$-th
coordinate axis, then $i$ is a principle direction for every
$\ell\in\Gamma_{F,v}(m_0)$.
We call
a cube $Q\in\Q_n$ {\it $m_0$-good}
if it is $m_0$-good for all lines which intersect $Q$ properly
and are parallel to some line in $\Gamma(Q,m_0)$. Finally, a cube $Q\in\Q_n$ is
{\it $m_0$-bad} if it is not $m_0$-good. 
\end{definition}

\begin{figure}[h]
\centering{
\resizebox{1\textwidth}{!}{\input{fig33.tex}}}
\caption{On the left hand side, the black cubes are deleted, white ones are
  chosen and the black cube with a grey spot is strongly $i$-deleted, where
  $i=1$. A collection $\Gamma_{F,v}(m_0)$ is illustrated on the right hand side.}
\label{fig33}
\end{figure}
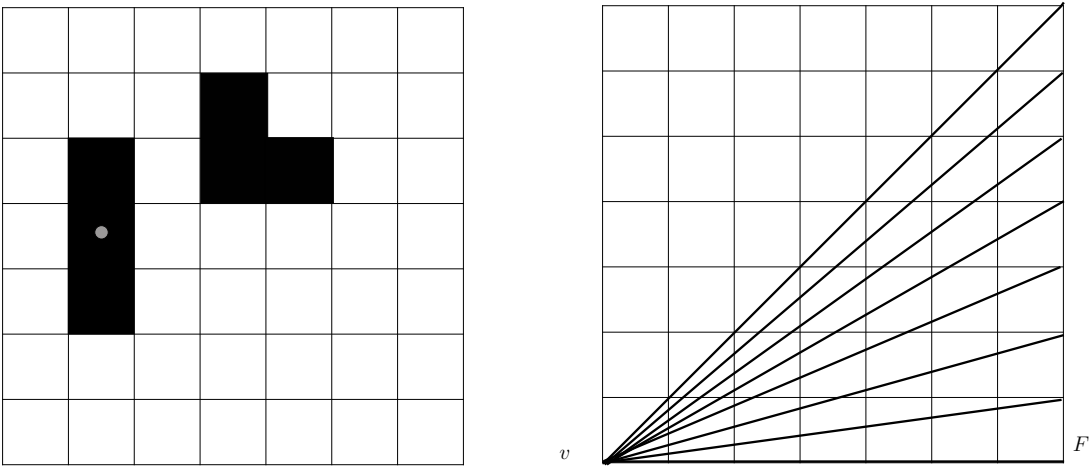

\begin{remark}\label{notchosen}
(a) Observe that all the concepts in Definition~\ref{good} depend only on
$\C_{n+m_0}(\omega)$. If $Q,Q'\in\Q_n$ with $Q\not\sim Q'$, then the events
``$Q$ is $m_0$-good'' and ``$Q'$ is $m_0$-good'' are independent 
since only those cubes of 
${\mathcal C}_{n+m_0}(\omega)$ play any role that intersect $Q$ or $Q'$  and since $m_0$ is at least $1$, no cube intersects both $Q$ and $Q'$.  
If $Q\sim Q'$ then 
these events are not independent
due to the lines passing through $Q$ or $Q'$ close
to their boundaries.  

(b) If $Q\in\Q_n$ is $m_0$-good for a line $\ell$, then there is a line segment 
$I_\ell\subset Q\cap\ell\cap E(\omega)^c$ with 
$\mathcal H^1(I_\ell)\ge d^{-1}N^{-n-m_0}$.
\end{remark}

We continue by introducing the notation utilised in
Lemmas~\ref{goodforline}--\ref{numberoflines}. Set
\begin{equation}\label{Nm0}
N_{m_0}:=\lfloor d^{-1}N^{m_0}\rfloor-4,
\end{equation}
where $\lfloor x\rfloor$ is the integer part of $x\in\R$. If $Q\in\Q_n$ and
$\ell$ is a line such that $\ell\cap Q$ contains more than one
point, let $\ell^a(Q)$ and $\ell^b(Q)$ be the end points of the line segment
$\ell\cap Q$. Given $n\in\N$, $Q\in\Q_n$ and $i\in\{1,\dots,d\}$, let $\ell$
be a line with a principle direction $i$ and intersecting $Q$ properly. 
The number of $(n+m_0,i)$-layers 
which the line segment
$L(\ell^a(Q),\ell^b(Q))$  intersects in a set of positive length is at least
$\lfloor d^{-1}N^{m_0}\rfloor$. Ignoring the first two and the last two of them
(in the natural order, recall property (a) in Definition~\ref{good}), we are
left with at least $N_{m_0}$ layers which $L(\ell^a(Q),\ell^b(Q))$ passes
through.
For each such $(n+m_0,i)$-layer $\L'$,
by choosing
$x$ and $y$ to be the end points of the line segment $\L'\cap\ell$, the
assumptions of Observation~\ref{properintersection} are valid for $i$, $x$, $y$
and $\L'$. Therefore, in each of these layers $\L'$, there is at
least one cube $Q'\in\Q_{n+m_0} (\L')$ which $\ell$ intersects 
very
properly in direction
$i$.  
Select the first $N_{m_0}$ of these layers.
If there is more than one
very
properly intersected cube inside some layer, order the cubes inside the layer
in some systematic way and select the smallest 
very
properly intersected cube with
respect to this order. Let
\begin{equation}\label{selectedcubes}
\K_{m_0}^Q(\ell):=\{Q_1',\dots,Q_{N_{m_0}}'\}
\end{equation}
be the collection of cubes selected in this manner, see
Figure~\ref{fig:percol3}.
If $\ell$ does not intersect $Q$ properly,
we set
$\K_{m_0}^Q(\ell):=\emptyset$.

\begin{lemma}\label{goodforline}
Let $n\in\N$, $Q\in\Q_n$ and $i\in\{1,\dots,d\}$. Fix $m_0\in\N$ such that
$N_{m_0}>1$. Assume that a line $\ell$ 
with  a  principle direction $i$
intersects $Q$ properly. 
For all
$0\le p<1$, there exists $0\le q=q(p,d)<1$ such that
\[
\P(\{\omega\in\Omega\mid\text{ every }Q'\in\K_{m_0}^Q(\ell)\text{ is weakly }
 i\text{-chosen}\})\le q^{N_{m_0}}.
\]
In particular,
\[
\P(\{\omega\in\Omega\mid Q\text{ is }m_0\text{-good for }\ell\})
  \ge 1-q^{N_{m_0}}.
\]
\end{lemma}

\begin{proof}
By definition, for every $Q'\in\Q_{n+m_0}$, 
\[
\P(\{\omega\in\Omega\mid Q'\text{ is strongly }i\text{-deleted}\})
  \ge (1-p)^{3^{d-1}},
\]
giving
\[
\P(\{\omega\in\Omega\mid Q'\text{ is weakly }i\text{-chosen}\})
  \le 1-(1-p)^{3^{d-1}}=:q.
  \]
Since the cubes in $\K_{m_0}^Q(\ell)$ belong to different $(n+m_0,i)$-layers,
the events ``$Q_j'$ is weakly $i$-chosen'' and
``$Q_k'$ is weakly $i$-chosen'' are independent provided that $j\ne k$.
Therefore,
\[
\P(\{\omega\in\Omega\mid\text{ every }Q'\in\K_{m_0}^Q(\ell)\text{ is weakly }
 i\text{-chosen}\})\le q^{N_{m_0}}.
\]
Note that $Q$ is $m_0$-good for $\ell$, if at least one of the cubes
$Q'\in\K_{m_0}^Q(\ell)$ is strongly $i$-deleted. Thus,
\[
\P(\{\omega\in\Omega\mid Q\text{ is }m_0\text{-good for }\ell\})
  \ge 1-q^{N_{m_0}}.
\]
\end{proof}

\begin{definition}\label{face}
 Let $n\in\N$.  If $F$ is a face of a cube $Q\in\Q_n$, 
we  denote by $-F$ the face of $Q$ which is parallel to $F$ and not equal
to $F$.  If  $\ell$ is a line and $a$ is a point in $\R^d$, we 
denote by $\ell_a$ the line parallel to  $\ell$ and  containing $a$.
\end{definition}

\begin{lemma}\label{gammaparallel}
Let $n,m_0\in\N$ with $m_0>0$ and let $Q\in\Q_n$. Assume that $\ell$ intersects
$Q$ properly and is parallel to some $\ell'\in\Gamma(Q,m_0)$. Then there exist
a face $F$ and a vertex $v$ of $Q$ and a line $\hat\ell\in\Gamma_{F,v}(m_0)$
such that $\ell=\hat\ell_a$ for some $a\in -F$. 
\end{lemma}

\begin{proof}
Since  $\ell'\in\Gamma(Q,m_0)$, we have $\ell'\in\Gamma_{F',v'}(m_0)$ for
some face $F'$ and vertex $v'$ of $Q$. If $\ell\cap(-F')\ne\emptyset$,
we may choose
$\hat\ell=\ell'$ and $F=F'$. If $\ell\cap(-F')=\emptyset$, then
$\ell$ intersects $F'$ and there is $\hat\ell\in\Gamma_{-F',-v'}(m_0)$ parallel
to $\ell$ such that $\ell=\hat\ell_a$ for some $a\in F'$, where $-v'$ is
the opposite vertex to $v'$.
\end{proof}

Next we estimate the number of essentially different translates of a line
$\ell\in\Gamma_{F,v}(m_0)$.  Denote  by $\dmax$ the maximum metric
on $\R^d$.

\begin{lemma}\label{numberoflines}
Let $n\in\N$, $i\in\{1,\dots,d\}$ and $Q\in\Q_n$. Assume that
$m_0\in\N$ is such that $N_{m_0}\ge 1$.
Fix a face $F$ and a vertex $v$ of $Q$ such that $v\not\in F$ and $F$ is
perpendicular to the $i$-th coordinate axis. Assume that
$\ell\in\Gamma_{F,v}(m_0)$.
\begin{itemize}
\item[(a)]  If 
$
a,
b\in -F$ with $\dmax(a,b)\ge N^{-n-m_0}$, then
  $\K_{m_0}^Q(\ell_a)\cap\K_{m_0}^Q(\ell_b)=\emptyset$.
\item[(b)] 
We have
\[
\#\{\K_{m_0}^Q(\ell_b)\mid b\in 
-F
\}\le 
N^{2m_0(d-1)}, 
\]
where the number of elements in a set $B$ is denoted by $\# B$.
\end{itemize}
\end{lemma}

\begin{proof} For an illustration of collections $\K_{m_0}^Q(\ell_a)$ and
$\K_{m_0}^Q(\ell_b)$, see Figure~\ref{fig:percol3}. 
Claim (a) follows directly from the definition of $\K_{m_0}^Q(\ell)$ and is
obvious from Figure~\ref{fig:percol3}. 
%
%

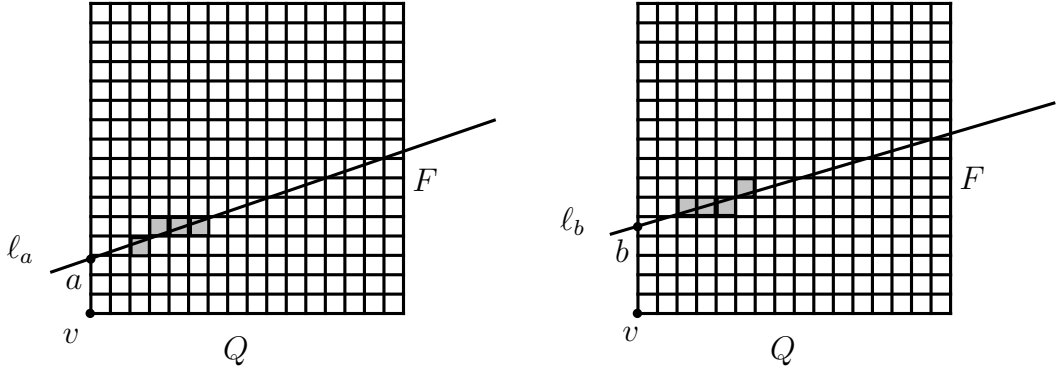
\begin{figure}[h]
\centering{
\resizebox{1\textwidth}{!}{\input{figpercol3.tex}}}
\caption{Illustration of collections $\K_{m_0}^Q(\ell_a)$ and
$\K_{m_0}^Q(\ell_b)$ in the case $d=2$, $N_{m_0}= 4$ and $i=1$. The cubes belonging
to $\K_{m_0}^Q(\ell_a)$ and $\K_{m_0}^Q(\ell_b)$ are shaded.}
\label{fig:percol3}
\end{figure}
%
%
To prove (b), by symmetry we can suppose that $i=1$ and $-F$ is  included  in
the hyperplane $(x_1=0):=\{x\in\R^d\mid x_1=0\}$. 
Let $H$ be the union of the hyperplanes 
$(x_1=j N^{- n  -m_0}d^{-1})$ for $j=2d,2d+1,\ldots,(N_{m_0}+2)d$, and
denote by $D$  the union of  the hyperplanes spanned by  the 
faces of the cubes of $\Q_{n+m_0}(Q)$ that are not parallel to $F$.
Note that when $b$ varies in $-F$, the set $\K_{m_0}^Q(\ell_b)$
can be changed only when $\ell_b$ intersects $H\cap D$.
Thus,  if we project $H\cap D$ in the direction of $\ell$
to the hyperplane  $(x_1=0)$,  we get a grid such that 
$\K_{m_0}^Q(\ell_b)$  is  constant   if $b$ varies inside  any 
cell of the grid. This grid  has  at most 
$(N_{m_0}d +1) (N^{m_0}+1)$  many  $(d-2)$-dimensional  walls in each direction
 $j\ne 1$,  so it has at most
$((N_{m_0}d+1)(N^{m_0}+1)+1)^{d-1}$ cells.
Since by definition $N_{m_0}d\le N^{m_0}- 4 d\le N^{m_0}-2$,
we have 
$((N_{m_0}d+1) (N^{m_0}+1)+1)^{d-1}\le N^{2m_0(d-1)}$,
which completes the proof of (b).
\end{proof}

According to the next proposition, which is one of our key preliminary results,
the probability for being $m_0$-good is large when $m_0$ is large.

\begin{proposition}\label{probgood}
For all $0\le p<1$ and $m_0\in\N\setminus\{0\}$, there exists a number
$0\le p_g=p_g(m_0,p,d,N)\le 1$ with
$\lim_{m_0\to\infty}p_g(m_0,p,d,N)=1$ such that 
\[
\P(\{\omega\in\Omega\mid Q\text{ is }m_0\text{-good}\})\ge p_g
\]
for all $n\in\mathbb N$ and $Q\in\Q_n$.
\end{proposition}

\begin{proof}
If $Q$ is not $m_0$-good, there are $i\in\{1,\dots,d\}$, a face $F$ of $Q$
perpendicular to the $i$-th coordinate axis, a vertex $v$ of $Q$ not contained
in $F$ and a line $\ell$ 
with  a  principle direction $i$
intersecting $Q$ properly 
such that
$\ell$ is parallel to some $\ell'\in\Gamma_{F,v}(m_0)$, $\ell$ intersects $-F$
and $Q$ is $m_0$-bad for $\ell$. In particular, every $Q'\in\K_{m_0}^Q(\ell)$
is weakly $i$-chosen. By Lemma~\ref{numberoflines}, the number of different
collections $\K_{m_0}^Q(\ell_b)$ is at most 
$N^{2m_0(d-1)}$.
Using the fact that  $F$ and $v$ may be chosen in $2d$ and $2^{d-1}$ 
different ways, respectively, leads to
\[
\#\Gamma(Q,m_0)\le d2^d\#G_{m_0}(F)=d2^d(N^{m_0}+1)^{d-1}\le CN^{m_0(d-1)},
\]
where $C$ depends only on $d$. Using Lemma~\ref{goodforline} and
Lemma~\ref{numberoflines}, we conclude that
\begin{align*}
\P(\{\omega\in\Omega\mid Q \text{ is }m_0\text{-bad}\})
&\le CN^{m_0(d-1)}
N^{2m_0(d-1)}
q^{N_{m_0}}
= CN^{3m_0(d-1)}
q^{N_{m_0}}=:s.
\end{align*}
Since $N^{m_0}\le 2dN_{m_0}$ when
$N_{m_0}>
5
$, the claim follows with $p_g:=\max\{1-s,0\}$.
\end{proof}

We need a few geometrical lemmas.
We denote by
\begin{equation}\label{dist}
\dist(x,A):=\inf\{|x-y|\mid y\in A\}
\end{equation}
the Euclidean distance between a point $x\in\R^d$ and a set $A\subset\R^d$.

\begin{lemma}\label{minimumdistance}
Let $x=(1,x_2,\dots,x_d)$ and $w=(0,1,w_3,\dots,w_d)$ be points in $\R^d$ such
that $0\le x_j\le 1$ for all $j=2,\dots,d$ and $0\le w_j\le 1$ for all
$j=3,\dots,d$. Let $\ell$ be the line spanned by
$x$. Then
\[
\dist(w,\ell)\ge\frac 1{\sqrt 2}.
\]
\end{lemma}

\begin{proof}
A standard calculation shows that $\dist(w,\ell)$ is minimised when $x_2=1$ and
$|(x_3,\dots,x_d)|$ obtains an extreme value, that is,
$(x_3,\dots,x_d)\in\{(0,\dots,0),(1,\dots,1)\}$. If
$(x_3,\dots,x_d)=(0,\dots,0)$, the claim easily follows. In the other case, a
standard calculation shows that the minimum is obtained when
$w_j=\frac 12$ for all $j=3,\dots,d$ and then $\dist(w,\ell)=\frac 1{\sqrt 2}$.
\end{proof}

\begin{lemma}\label{mindistancegenform}
If  $j\in\{1,\dots,d\}$ is a principle direction  for  a line
$\ell\subset\R^d$ then,  for  all  $w\in\R^d$,  we have  that
\begin{equation}\label{e:sqrt2}
\dmax(w,\ell\cap \Pi_j^{-1}(\Pi_j(w)))\le 
\sqrt 2\dist(w,\ell).
\end{equation}
\end{lemma}

\begin{proof}
By choosing the coordinate system properly,  this
follows directly from Lemma~\ref{minimumdistance}.
\end{proof}

\begin{lemma}\label{lines}
 Let $n,m_0\in\N$ with $m_0>0$ and $Q\in\Q_n$. Suppose that
$j\in\{1,\dots,d\}$ is a principle direction  for  a line $\ell\subset\R^d$ and
$\ell$ intersects $Q$ properly.  
Then there exist $\ell'\in\Gamma(Q,m_0)$ and a line $\ell''\subset\R^d$
such that
\begin{enumerate}
\item[(i)] the lines $\ell'$ and $\ell''$ are parallel  and $j$ is their
  principle direction,
\item[(ii)] $\Pi_j(\ell\cap Q)=\Pi_j(\ell''\cap Q)$ and
\item[(iii)] $\dmax\left(\ell\cap\Pi_j^{-1}(z),
  \ell''\cap\Pi_j^{-1}(z)\right)\le \frac{1}{2}N^{-n-m_0}$
  for every $z\in\Pi_j(Q)$.
\end{enumerate} 
\end{lemma}

\begin{proof}
Figure~\ref{fig39} is serving as an illustration for the proof of this lemma
and also for its application in the proof of Lemma \ref{gap}. Hence, apart from
objects used in this proof,  some other ones are also shown. The dashed line
in the figure is $\Pi_j^{-1}(z)$. 

Without loss of generality,  we can suppose that $n=0$, $Q=[0,1]^d$ and
$j=1$. Let the direction vector of $\ell$ be $(1,a_2,\ldots,a_d)$. 
Since  $1$ is a principle direction  for  $\ell$,  we have  that
 $|a_i|\le 1$  for all $i=2,\ldots,d$. 
By symmetry,  we can suppose that $a_i\ge 0$ for every $i$.
Let $v=(0,\ldots,0)$,  $u=(1,a_2,\ldots,a_d)$ and
$F$ be the face of $Q$ in the  hyperplane $(x_1=1)$.  
Then there is a point 
$u'=(1,a'_2,\ldots,a'_d)\in G_{m_0}(F)$ such that
$|a_i-a'_i|\le \frac{1}{2}N^{-n-m_0} $ for  all  $i=2,\ldots,d$. 
 The points  $v$ and $u'$ define an element $\ell'\in\Gamma(Q,m_0)$.
Let $ \tilde x = (0,b_2,\ldots,b_d)$ and $ \tilde y = ( 1  ,c_2,\ldots,c_d)$ be
the intersections of $\ell$ with the  hyperplanes  $(x_1=0)$ and 
$(x_1=1)$, respectively.  Clearly,  $c_i-b_i=a_i$ for  all 
$i=2,\ldots,d$.

Note that if we want to define $\ell''$ as a line 
through  $(0,b''_2,\ldots,b''_d)$ and $(1,c''_2,\ldots,c''_d)$, 
the requirements (i), (ii) and (iii) about $\ell''$
are satisfied if and only if, for all $i=2,\ldots,d$, 
we have 
\begin{enumerate}
\item[(1)] $c''_i-b''_i=a'_i$, 
\item[(2)] $(\forall t\in [0,1])\  
  a_i t + b_i\in[0,1]\Leftrightarrow a'_i t + b''_i\in[0,1]$ and
\item[(3)] $|b''_i-b_i|\le\frac{1}{2}N^{-n-m_0}\text{ and }
  |c''_i-c_i|\le\frac{1}{2}N^{-n-m_0}.$
\end{enumerate}  
For different  values of  $i$ these conditions are independent, 
and for any fixed $i$ they are equivalent to the
two-dimensional versions of the original (i), (ii) and (iii). 
Therefore,  we can suppose that $d=2$.
 Then  either one or both endpoints of the segment
$\ell\cap Q$ are on the vertical sides of the square $Q$.
 In the first case,  it is easy to see that we can choose $\ell''$ as
the line parallel to $\ell'$  and going  through the other endpoint
of $\ell\cap Q$. In the latter case, there are two possibilities: either
$\ell\cap Q$ intersects $\ell'\cap Q$ or not. If the intersection is non-empty,
we may choose $\ell''=\ell'$. In the opposite case, we can choose $\ell''$ as
the line parallel to $\ell'$ and going through 
the endpoint of $\ell\cap Q$ closer to $\ell'$.
This situation is shown in Figure~\ref{fig39}, where $\ell ''$
is going through the endpoint $\tilde x$. 
\end{proof}

We denote by $\dH$ the Hausdorff distance between compact subsets of $\R^d$.

\begin{lemma}\label{gap}
Let $n,m_0\in\N$ with $N_{m_0}\ge 1$, 
$\omega\in\Omega$ and $i\in\{1,\dots,d\}$.
Suppose that $x,y\in [0,1]^d$ are such that $L(x,y)$ passes through an
$(n,i)$-layer $\L$. Let $\gamma\colon [a,b]\to\R^d$ be a curve passing through
$\L$ such that 
\begin{equation}\label{gammacloseL}
\dH(\gamma([a,b]),L(x,y))<\frac 1{2\sqrt 2}N^{-n-m_0}.
\end{equation}
 Finally,
assume that all the cubes $Q\in\Q_n(\L)$ intersecting $L(x,y)$ are $m_0$-good.  
Then there are $\tilde a,\tilde b\in [a,b]$ such that 
$\gamma(\mathopen]\tilde a,\tilde b\mathclose[)\cap E(\omega)=\emptyset$ and
$|\gamma(\tilde a)-\gamma(\tilde b)|\ge d^{-1}N^{-n-m_0}$.
\end{lemma}

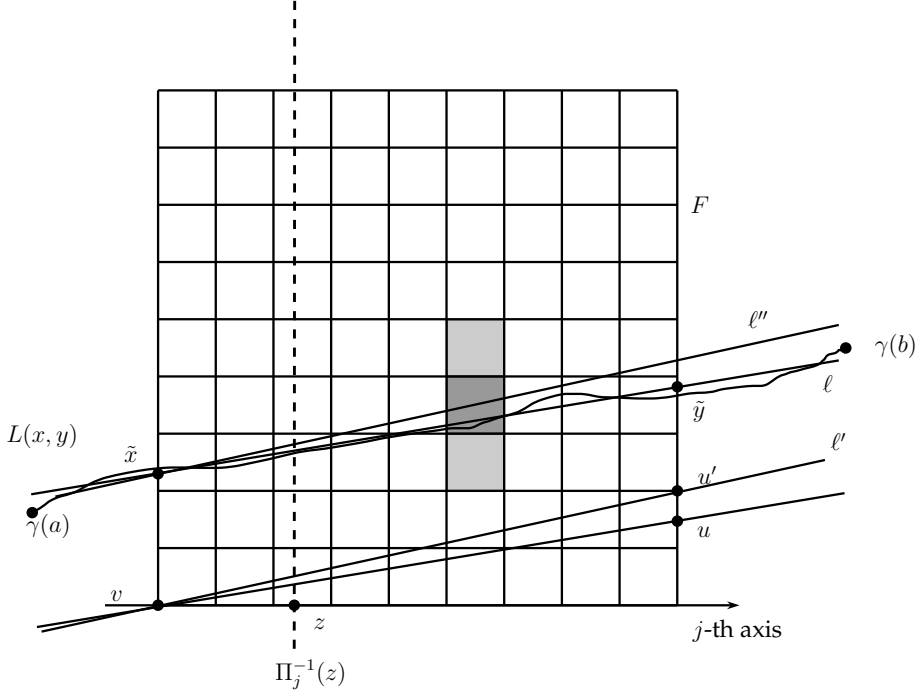
\begin{figure}[h]
\centering{
\resizebox{0.9\textwidth}{!}{\input{fpcfig39.tex}}
\caption{  An illustration for the proofs of Lemmas \ref{lines}  and \ref{gap}.}
\label{fig39}}
\end{figure}

\begin{proof}
If $x$, $y$ and $\L$ do not satisfy the assumptions of
Observation~\ref{properintersection}, there is
$j\in\{1,\dots,d\}\setminus\{i\}$ such that $L(x,y)$ intersects three
successive $(n,j)$-layers, and one may pick points $x',y'\in L(x,y)$ and an
$(n,j)$-layer $\L'$ such that the line segment $L(x',y')\subset L(x,y)$ passes
through $\L'$ and $x'$ and $y'$ belong to the boundary of $\L'$. Iterating this
argument, we end up with points $\tilde x$ and $\tilde y$ and an
$(n,\tilde j)$-layer $\widetilde\L$ satisfying assumptions of
Observation~\ref{properintersection} with $i$ replaced by $\tilde j$. 

Let $\ell$ be the line containing $L(x,y)$ and let $j\in\{1,\dots,d\}$ be its
principal direction. Clearly, $\ell$ contains $\tilde x$ and $\tilde y$.
By Observation~\ref{properintersection}, there is $Q\in\Q_n(\widetilde\L)$ such
that $\ell$ intersects $Q$ properly and $\ell\cap Q=L(x,y)\cap Q$. 
By Lemma~\ref{lines}, there exist $\ell'\in\Gamma(Q,m_0)$ and a line
$\ell''\subset\R^d$  such that (i), (ii) and (iii) of Lemma~\ref{lines} hold. 

Let $\L_1,\dots,\L_k$ be the $(n+m_0,j)$-layers such that
$\ell''\cap Q\cap\L_h$ has positive length (recall this type of notation in 
Definition~\ref{good}). Since $N_{m_0}\ge 1$ and $\ell''$ intersects $Q$
properly, we have that $k\ge  N_{m_0}+4\ge 5$. Since $Q$ is $m_0$-good,
there is $h\in\{1,\dots,k\}\setminus\{1,2,k-1,k\}$ such that
$\ell''$ intersects properly a cube $Q'\in\Q_{n+m_0}(Q\cap\L_h)$ that is strongly
$j$-deleted.  By \eqref{gammacloseL}, there are points $c,d\in[a,b]$ such
that $|\gamma(c)-x|<N^{-n-m_0}$ and $|\gamma(d)-y|<N^{-n-m_0}$. Applying property
(ii) of Lemma~\ref{lines} and the fact that $h\not\in\{1,2,k-1,k\}$, we conclude
that the curve $\gamma|_{[c,d]}$ passes through $\L_h$. 
Suppose that $\gamma(t)\in\L_h$. Applying \eqref{e:sqrt2} of
Lemma~\ref{mindistancegenform} with $w=\gamma(t)$
and using \eqref{gammacloseL} and property (iii) of Lemma~\ref{lines}, 
we conclude that $\gamma|_{[ c,d]}$ passes through
$\Pi_j^{-1}(\Pi_j(Q'\cap\ell''))$ inside $\bigcup_{Q''\sim_j Q'}Q''$.
 Write $\partial\bigl(\Pi_j^{-1}(\Pi_j(Q'\cap\ell''))\bigr)=:B_c\cup B_d$,
where $B_c$ and $B_d$ are the hyperplanes closer to $\gamma(c)$ and $\gamma(d)$,
respectively. Now the points $\tilde a:=\max\{t\in [c,d]\mid\gamma(t)\in B_c\}$
and $\tilde b:=\min\{t\in [c,d]\mid\gamma(t)\in B_d\}$ satisfy the claim. 
\end{proof}

The next proposition will be used for studying the size of the set of points
that belong to a good cube with good neighbour cubes at infinitely many scales. 
 For  our purposes,  the critical value of dimension will be 1
 (see the proof of
Theorem~\ref{mainstandard} and Proposition~\ref{Salphaissmall})  but, with
the same effort, we prove a more general statement, since it might be useful in
other connections related to dimensions of random constructions. 

\begin{proposition}\label{smallsets}
Let $(\F_n)_{n\in\N}$ be a sequence of independent sub-$\sigma$-algebras of $\B$ 
on $\Omega$. For all $n\in\N$, let $A_n\colon\Omega\times\Q_n\to\{0,1\}$ be a 
function such that $A_n(\cdot,Q)$ is $\F_n$-measurable for all $Q\in\Q_n$. Let 
$0\le\varrho\le 1$. Assume that, for all $n\in\N$ and $Q\in\Q_n$,
\begin{equation}\label{probA}
\P(\{\omega\in\Omega\mid A_n(\omega,Q)=1\})\le\varrho.
\end{equation}
Let $M_1\in\N$ and assume that $\simeq$ is a reflexive relation on
$\bigcup_{n\in\N}\Q_n$ such that, for all $Q\in\Q_n$, there are at most $M_1$
cubes $Q'\in\Q_n$ with $Q'\simeq Q$, and $Q'\not\simeq Q$, if $Q'\in\Q_m$ with
$m\ne n$. Given $\omega\in\Omega$, we say that a cube 
$Q\in\Q_n$ is {\it selected}, if $A_n(\omega,Q')=1$ for some $Q'\simeq Q$. 
For all $n,k\in\N$, define
\[
\widetilde E_n(\omega):=\bigcup_{\substack{Q\in\Q_n\\ Q\text{ is selected}}}Q\quad
\text{and}\quad\widetilde E^k(\omega):=\bigcap_{n=k}^\infty\widetilde E_n(\omega).
\]
Then,  for all $k\in\N$, we have
\[
\dimh\widetilde E^k(\omega)\le\frac{\log(\max\{M_1\varrho N^d,1\})}{\log N}
\]
for $\P$-almost all $\omega\in\Omega$. In particular,
$\dimh\widetilde E^k(\omega)<1$ for all small enough $\varrho$.    
\end{proposition}

\begin{proof}
Note that $\widetilde E^k(\omega)$ is a finite union of sets satisfying similar
assumptions as $\widetilde E^1(\omega)$. Hence, it suffices to prove the claim
for
$\widetilde E^1(\omega)$. For all $n\in\N$, let 
\[
\S_n(\omega):=\Big\{Q\in\Q_n\mid Q\subset\bigcap_{k=1}^n\widetilde E_k(\omega)
\Big\},
\]
that is, $\S_n(\omega)$ is the collection of selected cubes in $\Q_n$ which
are subsets of selected cubes in $\Q_k$ for all $k=1,\dots,n-1$.
Set $\tilde\varrho:=\max\{M_1\varrho N^d,1\}$ and 
$s:=\tfrac{\log\tilde\varrho}{\log N}$. Then
\begin{equation}\label{sbound}
\tilde\varrho N^{-s}=1.
\end{equation}
For all $n\in\N$, define 
\begin{equation}\label{Zndef}
Z_n(\omega):=(\sqrt d)^s\sum_{Q\in\S_n(\omega)}N^{-ns}.
\end{equation}
Note that the diameter of $Q\in\Q_n$ is $\diam Q=\sqrt dN^{-n}$ and 
\begin{equation}\label{Znprop}
Z_n(\omega)=(\sqrt d)^s\sum_{Q\in\S_{n-1}(\omega)}N^{-(n-1)s}
  \sum_{\substack{Q'\in\S_n(\omega)\\ Q'\subset Q}}N^{-s}.
\end{equation}

Let $\widetilde\F_n:=\bigvee_{i=1}^n\F_n$ be the $\sigma$-algebra generated by 
$\F_1,\dots,\F_n$. Let $Q\in\S_{n-1}(\omega)$. Note that every $Q'\in\Q_n(Q)$
has at most $M_1$ cubes $Q''\in\Q_n(Q)$ with $Q''\simeq Q'$ (including $Q'$
itself), and there are $N^d$ elements in $\Q_n(Q)$. Therefore, combining
\eqref{Znprop}, \eqref{Zndef}, \eqref{probA} and \eqref{sbound}, we conclude
that
\begin{align*}
\E(Z_n\mid\widetilde\F_{n-1})&=(\sqrt d)^s\sum_{Q\in\S_{n-1}(\omega)}N^{-(n-1)s}
    \E\Bigl(\sum_{\substack{Q'\in\S_n(\omega)\\ Q'\subset Q}}N^{-s}\Bigr)\\
  &\le Z_{n-1}M_1\varrho N^dN^{-s}\le Z_{n-1},
\end{align*}
implying that $(Z_n)_{n\in\N}$ is a supermartingale with respect to the 
filtration $(\widetilde\F_n)_{n\in\N}$. By Doob's supermartingale 
convergence theorem (see, for example, \cite[Section 11.5]{W}), the limit 
$\lim_{n\to\infty}Z_n(\omega)=Z(\omega)$ exists
and is finite for $\P$-almost all $\omega\in\Omega$. For all such 
$\omega\in\Omega$, we have for all $\delta>0$ and for large enough $n\in\N$ 
that
\[
\H_\delta^s(\widetilde E^1(\omega))\le Z_n(\omega)\le Z(\omega)+1.
\]
Therefore, $\H^s(\widetilde E^1(\omega))<\infty$, giving 
$\dimh\widetilde E^1(\omega)\le s$  and completing the proof. 
\end{proof}

Next we prove that the fractal percolation set is purely
1-unrectifiable almost surely. Even though this result is a special case of
our main theorem 
(Theorem~\ref{maingeneral}), we give here a simple alternative proof.
Recall that a set $F\subset\R^d$ is purely $k$-unrectifiable if
$\H^k(f(\R^k)\cap F)=0$ for all Lipschitz maps $f\colon\R^k\to\R^d$.
The following characterisation of pure $k$-unrectifiability 
will be utilised in the proof of Theorem~\ref{mainstandard}.

\begin{theorem}\label{FM}
Let $k\in\N\setminus\{0\}$. A set $F\subset\R^d$ is purely $k$-unrectifiable if
 and only if  $\H^k(M\cap F)=0$ for all $k$-dimensional
$C^1$-submanifolds $M\subset\R^d$.
\end{theorem}

\begin{proof}
See \cite[Theorem 15.21]{M} or \cite[Theorem 3.2.29]{F}.
\end{proof}  

\begin{corollary}\label{1impliesk}
If $F\subset\R^d$ is a purely $1$-unrectifiable Borel set,  then it is
also purely $k$-unrectifiable for any $k\in \N\setminus\{0\}$.
\end{corollary}

\begin{proof}
Assume that $F$ is not purely $k$-unrectifiable. Then
there exists a $k$-dimen\-sional $C^1$-submanifold $M$ such that 
$\H^k(M\cap F)>0$. Fubini's theorem, in turn, implies that
there exists a $1$-dimensional $C^1$-submanifold $L$ on $M$ such that
$\H^1(L\cap F)>0$ and, therefore,
$E(\omega)$ is not purely $1$-unrectifiable.
\end{proof}

\begin{theorem}\label{mainstandard}
For all $0\le p<1$, the set $E(\omega)$ is purely 1-unrectifiable for 
$\P$-almost all $\omega\in\Omega$.
\end{theorem}

\begin{proof}
Suppose that the claim is not true. Then
\[
A:=\{\omega\in\Omega\mid E(\omega)\text{ is not purely }1
\text{-unrectifiable}\}
\]
has positive $\P$-measure. Consider $\omega\in A$. By Theorem~\ref{FM},
there exists a $1$-dimensional $C^1$-submanifold $M\subset\R^d$ such that  
$\H^1(M\cap E(\omega))>0$. Let $\mu:=\H^1|_M$ be the restriction of $\H^1$ to
$M$ and denote by $D(\omega)$ the set of $\mu$-density points of
$E(\omega)$, that is,
\[
D(\omega)=\Bigl\{x\in E(\omega)\mid\lim_{ r\searrow 0}
\frac{\mu(E(\omega)\cap B(x,r))}{\mu(B(x,r))}=1\Bigr\}
\]
where $B(x,r)\subset\R^d$ is the closed ball with radius $r>0$ centred at 
$x\in\R^d$. By \cite[Corollary 2.14]{M}, we have that $D(\omega)\subset M$ and
$\mu(D(\omega))=\mu(E(\omega))$.

For all $n\in\N$, define a function $A_n\colon\Omega\times\Q_n\to\{0,1\}$
by setting $A_n(\eta,Q):=1$, if and only if $Q$ is $m_0$-bad. Let
$\simeq$ be the relation $\sim$ from Definition~\ref{properly} defining
neighbouring cubes. Then $\simeq$ satisfies the assumptions of the relation in
Proposition~\ref{smallsets} with $M_1=3^d$.
By means of Proposition~\ref{smallsets}, we will show that
\begin{equation}\label{subsetgoal}
D(\omega)\subset\bigcup_{k}\widetilde E^k(\omega),
\end{equation}
where $\widetilde E^k(\omega)$ is as in Proposition~\ref{smallsets}.
Note that, by Remark~\ref{notchosen}, $A_n$ satisfies the measurability
assumption of Proposition~\ref{smallsets} with $\F_n$ being the
$\sigma$-algebra generated by $\C_{n+m_0}$ and, moreover,   
Proposition~\ref{probgood} implies that assumption~(\ref{probA}) is valid
with $\varrho=1-p_g$.
For verifying~\eqref{subsetgoal}, let $x\in D(\omega)$ and let $\ell$ be
the tangent space of $M$ at $x$. Since $M$ is a $C^1$-submanifold, 
for all $m_0\in\N$, there exist $r_0>0$ such that
\[
\dH(M\cap B(x,r),\ell\cap B(x,r))<\tfrac 1{2\sqrt{2 d}} N^{-m_0}r
\]
for all $0<r\le r_0$. Further, there is a constant $c_1\ge 1$ 
such that $r\le\mu(B(x,r))\le c_1r$ for all $0<r\le r_0$. Combining this with
the fact that $x$ is a $\mu$-density point of $E(\omega)$ implies the
existence of $r_1>0$ such that 
$\mu(E(\omega)^c\cap B(x,r))<d^{-1}N^{-m_0}r$ for all $0<r\le r_1$.
 Applying Lemma~\ref{gap} with the line segment
$\ell\cap B(x,\sqrt d N^{-n})$ and the curve $M\cap B(x,\sqrt d N^{-n})$, we
conclude that,  for all large $n\in\N$, either $Q_n(x)$ or one of its 
neighbour cubes is $m_0$-bad, where $Q_n(x)$ is the cube in $\Q_n$ whose
half open counterpart contains $x$. Hence,
$x\in\widetilde E^k(\omega)$ for some $k\in\N$, completing the proof of
\eqref{subsetgoal}.

Since $\lim_{m_0\to\infty}p_g=1$, we deduce from Proposition~\ref{smallsets} that,
for large enough $m_0\in\N$, we have $\dimh(\bigcup_k\widetilde E^k(\omega))<1$
for $\P$-almost all $\omega\in A$. This leads to a contradiction with
\eqref{subsetgoal}, since $\dimh(D(\omega))=1$.
\end{proof}

Combining Corollary~\ref{1impliesk} and 
Theorem~\ref{mainstandard},  we obtain pure
$k$-unrectifiability of typical fractal percolation sets for all
$k\in\N\setminus\{0\}$.
\medskip

\begin{corollary}\label{kunrectifiable}
Let $k\in\N\setminus\{0\}$. For all $0\le p<1$, the set $E(\omega)$ is purely 
$k$-unrectifiable for $\P$-almost all $\omega\in\Omega$.
\end{corollary}


\section{Existence of hereditarily good cubes}\label{hereditarilyexists}

In the previous section, dealing with 1-unrectifiability, one of our key
concepts  is  the $m_0$-good cube.  While we zoom in, differentiable curves
 start  to look almost like lines and, therefore,  the little cubes in
$m_0$-good cubes removed during the fractal percolation process  guarantee
that  differentiable curves cannot intersect the fractal percolation set in
a set of positive linear measure. If $\alpha<1$ then  $\alpha$-H\"older 
curves are more flexible and can ``go around'' holes. The price one needs to pay
for it is a tiny bit of length increase compared to a straight line segment.
In order to prove that curves staying close to the fractal percolation set
cannot be  $\alpha$-H\"older  for $\alpha$ close to $1$, we need that
$m_0$-good cubes at different levels are occurring sufficiently uniformly. To
be more precise, as we zoom out at certain zoom levels, with high probability,
some  $m_0$-bad  cubes show up. However, with high probability, in one
``large'' cube we do not have too many  $m_0$-bad  cubes. 
The ``large'' cubes are considered "good" if this is the case. If we zoom out
further, we might find some ``much larger'' cubes which, with high probability,
contain some ``large'' cubes which are  ``bad''.  However, with high
probability, in one ``much larger'' cube we do not have too many ``large'' cubes
which are  ``bad''.  The ``much larger'' cubes are considered ``good''
if this is the case.  Clearly, this procedure can be iterated.  To have
 an  idea about this procedure, one  may  look at Figure~\ref{fig41}.
One can think of the small blue squares (two dimensional cubes) as the
 $m_0$-bad  cubes. Then the larger blue cubes are  bad 
``large'' cubes (they contain too many  $m_0$-bad  cubes) and 
one  ``good'' much larger level cube is shown  in  the figure.

To determine the different zoom levels, at which we  look  at good and bad
cubes, is very delicate. If we zoom out too fast,  we  might end up with too
many bad cubes  and, on the other hand, if we zoom out too slowly,  the
probability estimates about
the number of good cubes do not work. The complicated and quite technical task
of finding the proper zoom levels is postponed to Section~\ref{appendixa}.

In this section, our main aim is to define hereditarily good cubes in
Definition~\ref{hereditary}, that is, cubes where the above procedure can
be repeated for several different sequences of scales, and to prove
Proposition~\ref{Salphaissmall}  stating  that, with probability one,
the Hausdorff dimension  of the set of points, which are not included in
infinitely many ``highly hereditarily good'' cubes, is less than one.

\begin{figure}[h]
\centering{
\resizebox{0.6\textwidth}{!}{\input{fig41.tex}}}
\caption{A $(k,\bL,\mfc,m_0)$-good cube containing $(k-1,\bL,\mfc,m_0)$-bad and
  $(k-2,\bL,\mfc,m_0)$-bad cubes. Here $\mfc=2$.} \label{fig41}
\end{figure}

\begin{definition}\label{levels}
Fix $\omega\in\Omega$ and $n\in\N$. Let $\mfc,m_0\in\N\setminus\{0\}$. Assume
that $\bL:=(L_j)_{j=0}^k$ is a strictly decreasing finite sequence of 
nonnegative
integers such that $L_k=0$. Let $Q\in\Q_n$. We say that $Q$ is
{\it $(\bL,\mfc,m_0)$-good} if it is $(k,\bL,\mfc,m_0)$-good, where the concept
of being $(k,\bL,\mfc,m_0)$-good is defined inductively below. If $Q$ is not
$(\bL,\mfc,m_0)$-good, it is {\it $(\bL,\mfc,m_0)$-bad}.

$\bullet$
If $k=0$, we say that $Q$ is {\it $(0,\bL,\mfc,m_0)$-good} if it is $m_0$-good.
Otherwise, $Q$ is {\it $(0,\bL,\mfc,m_0)$-bad}.

$\bullet$
Assume that we have defined the concepts of being
$(\tilde k,\bL,\mfc,m_0)$-good and $(\tilde k,\bL,\mfc,m_0)$-bad for all
strictly decreasing sequences $\bL=(L_j)_{j=0}^{k-1}$ of natural numbers with
$L_{k-1}=0$ and $\tilde k=0,\dots,k-1$.

$\bullet$
Let $\bL=(L_j)_{j=0}^k$ with $L_k=0$. For $\tilde k\le k-1$, define a sequence
$\widetilde\bL=(\widetilde L_j)_{j=0}^{\tilde k}$ by setting
$\widetilde L_j:=L_j-L_{\tilde k}$ for all $j=0,\dots,\tilde k$. We say that $Q$
is {\it $(\tilde k,\bL,\mfc,m_0)$-good} if it is
$(\tilde k,\widetilde\bL,\mfc,m_0)$-good. Otherwise, $Q$ is
{\it $(\tilde k,\bL,\mfc,m_0)$-bad}. Finally, we say that $Q$ is
{\it $(k,\bL,\mfc,m_0)$-good} if there are at most $\mfc$ cubes
$Q'\in\Q_{n+L_{k-1}}(Q)$  (recall \eqref{restrictedQn})   which are
$(k-1,\bL,\mfc,m_0)$-bad and, otherwise, $Q$ is {\it $(k,\bL,\mfc,m_0)$-bad}. 
\end{definition}

\begin{remark}\label{levelsremark}
If $Q\in\Q_n$ is $(k,\bL,\mfc,m_0)$-good, then all but $\mfc$ cubes
$Q'\in\Q_{n+L_{k-1}}(Q)$ are $(k-1,\bL,\mfc,m_0)$-good. This, in turn, means that
all except $\mfc$ cubes $Q''\in\Q_{n+L_{k-2}}(Q')$ are $(k-2,\bL,\mfc,m_0)$-good.
Iterating this, we conclude that, 
for $Q\in\Q_n$, the property of being $(\bL,\mfc,m_0)$-good is determined by 
$m_0$-good cubes $Q^{(k)}\in\Q_{n+L_0}(Q)$. On the other hand, whether a cube 
$Q^{(k)}\in\Q_{n+L_0}(Q)$ is $m_0$-good or not depends only on
$\C_{n+L_0+m_0}(\omega)$.
\end{remark}

 As observed in Remark~\ref{notchosen}.(a), the events
``$Q$ is $m_0$-good'' and ``$Q'$ is $m_0$-good'' are not independent if
$Q\sim Q'$. To overcome this difficulty, we introduce a notion of
independently good cubes in Definition~\ref{iwgood}. We begin by 
defining a class of shifted grids. Write
\[
S:=\{-1,1\}^d\cup\{\mathbf{0}\},
\]
where $\mathbf{0}$ is the origin of $\R^d$. 

\begin{definition}\label{shiftedgrid}
Let $n',m_0\in\N$ and $s\in S$. For all $n\in\N$, define 
\begin{align*}
\Q_n^{s,n',m_0}:=\Bigl\{\prod_{i=1}^d[(l_i-1)N^{-n},l_iN^{-n}]
    +\sum_{i=1}^d s_iN^{-n-n'-m_0}e_i\mid\,\, &l_i=1,\dots,N^n,\\
  &i=1,\dots,d\Bigr\},
\end{align*}
where $\{e_1,\dots,e_d\}$ is the standard basis of $\R^d$.
\end{definition}

When $s=\mathbf{0}$, we have that $\Q_n^{\mathbf{0},n',m_0}=\Q_n$ for all
$n,n',m_0\in\N$.
For all $n,n',m_0\in\N$ and for all $Q\in\bigcup_{s\in S}\Q_n^{s,n',m_0}$, set
\[
\interm0 Q:=(1-2N^{-m_0})Q,
\]
 where, for a cube $\widetilde Q$ and a positive constant $C$, we denote by
$C\widetilde Q$ the cube having the same centre as $\widetilde Q$ and the side
length $C$ times that of $\widetilde Q$.

\begin{definition}\label{iwgood}  
Let $n,m_0\in\mathbb N$ with $m_0>0$, and fix $\omega\in\Omega$.
Let $\ell\subset\R^d$ be a line with a principle direction
$i\in\{1,\dots,d\}$, and let $Q\in\bigcup_{s\in S}\Q_n^{s,0,m_0}$ be such that
$\ell\cap\interm0 Q$ has positive length. Let
$\L_1,\dots,\L_k$ be the $(n+m_0,i)$-layers (in the natural order) for which
$\ell\cap\interm0 Q\cap\L_j$ has positive length. We say that $Q$ is
{\it independently $m_0$-good for the line $\ell$} if there is a strongly
$i$-deleted cube $Q'\in\Q_{n+m_0}(\interm0 Q)$ so that
\begin{itemize}
\item[(a)] $Q'\not\subset\L_j$ for $j\in\{1,2,k-1,k\}$ and
\item[(b)] $\ell$ intersects $Q'$ properly.
\end{itemize}

A cube $Q\in\bigcup_{s\in S}\Q_n^{s,0,m_0}$  {\it is independently $m_0$-good}
if it is simultaneously
independently $m_0$-good for all lines which intersect $\interm0 Q$ properly
and are parallel to some line in $\Gamma(Q,m_0)$, where (with a slight abuse of
definition since the side length of $\interm0 Q$ is not equal to
$N^{-n}$) a line $\ell$ intersects the cube $CQ$ properly in direction $i$ if
$\H^1(\Pi_i(\ell\cap CQ))\ge d^{-1}N^{-n}$.
Finally, a cube $Q\in\Q_n$ is {\it independently $m_0$-bad} if it is not
independently $m_0$-good. 
\end{definition}

\begin{remark}\label{independent}
Clearly, if $Q,\widetilde Q\in\Q_n^{s,0,m_0}$ with $Q\ne\widetilde Q$, the events
``$Q$ is independently $m_0$-good'' and ``$\widetilde Q$ is independently
$m_0$-good'' are independent. However, if $Q$ is independently $m_0$-good, we
have no information about those lines which intersect $Q$ properly but do not
intersect $\interm0 Q$ properly.
This is the reason why shifted grids are needed.
\end{remark}

Next lemma describes the usefulness of shifted grids.

\begin{lemma}\label{iwgoodgood}
Let $n,m_0\in\N$ such that $m_0>0$ and let $Q\in\Q_n$. Suppose that all cubes
$\widetilde Q\in\bigcup_{s\in S}\Q_n^{s,0,m_0}$ intersecting the interior of $Q$
are independently $m_0$-good. Then $Q$ is $m_0$-good.
\end{lemma}


\begin{proof}
If a line $\ell$ intersects $Q$ properly, there exists
$\widetilde Q\in\bigcup_{s\in S}\Q_n^{s,0,m_0}$ such that $\ell$ intersects
$\interm0\widetilde Q$ properly, which gives the claim since
$\interm0\widetilde Q\subset Q$.
\end{proof}

Now we state a modified version of Definition~\ref{levels} valid for
shifted and non-shifted cubes.

\begin{definition}\label{iwlevels}
Fix $\omega\in\Omega$ and $m_0,\mfc\in\N\setminus\{0\}$. Let $k\in\N$ and let
$\bL=(L_j)_{j=0}^k$ be a strictly decreasing sequence of integers such that
$L_k=0$. For all $n\in\N$, the concept of $Q\in\bigcup_{s\in S}\Q_n^{s,L_0,m_0}$
being {\it independently $(\bL,\mfc,m_0)$-good or bad} is defined
by replacing goodness with independently goodness in Definition~\ref{levels}.
\end{definition}

\begin{remark}\label{iwlevelsremark}
Whether a cube $Q\in\bigcup_{s\in S}\Q_n^{s,L_0,m_0}$ is independently
$(\bL,\mfc,m_0)$-good or not is determined by the independently $m_0$-good
cubes $Q'\in\Q_{n+L_0}^{s,0,m_0}(Q)$. Furthermore, whether a cube
$Q'\in\Q_{n+L_0}^{s,0,m_0}(Q)$ is independently $m_0$-good or not depends only on
$\C_{n+L_0+m_0}(\omega)|_{Q'}$. Therefore, the events ``$Q$ is independently
$(\bL,\mfc,m_0)$-good'' and ``$\widetilde Q$ is independently
$(\bL,\mfc,m_0)$-good'' are independent provided
$Q,\widetilde Q\in\Q_n^{s,L_0,m_0}$ with $Q\ne\widetilde Q$.
\end{remark}   

Next lemma may be regarded as an extension of Lemma~\ref{iwgoodgood}. It 
determines the value of the constant $\mfc$ we will utilise later.

\begin{lemma}\label{*iwg}
Let $k,m_0\in\N$ with $m_0>0$ and let $\bL=(L_j)_{j=0}^k$ be a strictly
decreasing sequence of integers with $L_k=0$. Let $n\in\N$ and $Q\in\Q_n$.
Suppose that all cubes $\widetilde Q\in\bigcup_{s\in S}\Q_n^{s,L_0,m_0}$ which
intersect the interior of $Q$ are independently $(\bL,1,m_0)$-good. Then $Q$
is $(\bL,\mfc,m_0)$-good, where $\mfc:=4^d+1$.
\end{lemma}

\begin{proof}
We prove the claim by induction on $k$. The case $k=0$ is verified in
Lemma~\ref{iwgoodgood}. Suppose that the claim is true for $k-1$ and
fix $s\in S$.
Since every cube $\widetilde Q\in\Q_n^{s,L_0,m_0}$, intersecting the interior of
$Q$, is independently $(k,\bL,1,m_0)$-good, there is at most one
$Q'\in\Q_{n+L_{k-1}}^{s,L_0-L_{k-1},m_0}(\widetilde Q)$ which is independently
$(k-1,\bL,1,m_0)$-bad. If $s\ne \mathbf{0}$, $Q'$ is not in the standard grid,
and we colour in green all cubes in $\Q_{n+L_{k-1}}(Q)$ intersecting $Q'$.
On the other hand, if $s=\mathbf{0}$, $Q'$ is in the standard grid, and we
colour it in green. Since $\#(S\setminus\{ \mathbf{0}\})=2^d$ and every shifted
cube intersects $2^d$ standard cubes of the same size, the number of green
cubes is at most $\mfc$.

Consider a cube $Q'\in \Q_{n+L_{k-1}}(Q)$ which is not green. Then $Q'$ is
independently $(k-1,\bL,1,m_0)$-good and, moreover, every
$\widehat Q\in\bigcup_{s\in S\setminus\{\mathbf{0}\}}\Q_{n+L_{k-1}}^{s,L_0-L_{k-1},m_0}$
intersecting
$Q'$ is independently $(k-1,\bL,1,m_0)$-good, since otherwise
$Q'$ would be green. By the induction hypothesis, $Q'$ is
$(k-1,\bL,\mfc,m_0)$-good. Since this applies to all cubes
$Q'\in \Q_{n+L_{k-1}}(Q)$ which are not green, and the number of green
cubes in $Q$ is bounded by $\mfc$, we obtain that $Q$ is
$(k,L,\mfc,m_0)$-good.
\end{proof}

 
Observe that if $Q\in\Q_n$ is $m_0$-good, it is not necessarily independently
$m_0$-good. This follows from the fact that there may be a line $\ell$
intersecting $\interm0 Q$ properly
such that $\ell$ intersects only one strongly $i$-deleted $Q'\in\Q_{n+m_0}(Q)$
and $Q'\not\subset\interm0 Q$. However, analogously to Lemma~\ref{goodforline}
and Proposition~\ref{probgood}, one can prove the following proposition.
 
\begin{proposition}\label{probgoodiw}
For all $0\le p<1$ and $m_0\in\N\setminus\{0\}$, there exists a number
$0\le p_{ig}=p_{ig}(m_0,p,d,N)\le 1$ with $\lim_{m_0\to\infty}p_{ig}(m_0,p,d,N)=1$
such that 
\[
\P(\{\omega\in\Omega\mid Q\text{ is independently }m_0\text{-good}\})\ge p_{ig}
\]
for all $n\in\mathbb N$ and $Q\in\bigcup_{s\in S}\Q_n^{s,0,m_0}$.
\end{proposition}

 In our proof, we will use a very specific choice of sequences $\bL$  
depending on parameters $k_0,m_1\in\N\setminus\{0\}$ to be fixed later.  
For every $i\in\N$, we will define in Section \ref{appendixa} inductively a
finite decreasing sequence $\bL^i(m_1,k_0):=(L^i(m_1,k_0)_j)_{j=0}^{k_i}$ of
nonnegative  integers.  Here we list a few properties we use in the
sequel. These properties are verified in Section~\ref{appendixa}. Set
\begin{equation}\label{e:Deltaa}
\DD{i}{j}:=\LL{i}{j-1}-\LL{i}{j}.
\end{equation}
We will suppose that we have  (see \eqref{stepestimate2}) 
\begin{equation}\label{stepestimate}
  \Delta_j^i\le 5\sqrt 2^jm_1\,\text{ for all }i\in\N \text{ and }
  j=1,\dots,k_i.
\end{equation}
We also assume that  (see  \eqref{L0def}, \eqref{initialdef1} and
\eqref{initialdef2} in  Definition~\ref{j-levels}) 
\begin{equation}\label{*hmr}
L^l(m_1,k_0)_0=(l+1)l_0\,\text{ for all }l\in \N, 
\end{equation}
 where  $l_0:=m_1(1+2+\cdots+k_0)$. An  index  $i_0$ will be
determined such that $m_1k_{i_0}=l_0$, see \eqref{defofi0}.
 In Lemma~\ref{samesizenumber}, we show that 
\begin{align}
k_i&\ge i+3\,\text{ for all }\,i\le i_0+2\,\text{ and} \label{j-leva}\\
k_{i_0+l}&\ge 2\log_2(i_0+l+1)\,\text{ for all }\,l\ge 3 \label{j-levb}.
\end{align}

 For the rest of this section, we use these special sequences
$\bL^i(m_1,k_0)$. 
For all $n,l\in\N$, $m_0,k_0,m_1\in\N\setminus\{0\}$, $k\in\{0,\dots,k_l\}$ and
$0\le p<1$, define
\begin{align*}
q_{k,l}=q_{k,l}(m_1,k_0,m_0,p,d,N):=\P(\{\omega\in\Omega\mid\,\,
  & Q\text{ is independently }\\
  &(k,\bL^l(m_1,k_0),1,m_0)\text{-bad}\}),
\end{align*}
where 
$Q\in\bigcup_{s\in S}\Q_n^{s,L^l(m_1,k_0)_0-L^l(m_1,k_0)_k,m_0}([N^{-n},1-N^{-n}]^d)$.
 Note that the value of $q_{k_,l}$ does not depend on the choice of $n$, $s$
or $Q$. Moreover, $q_{0,l}\le 1-p_{ig}$ does not depend on $l$, $m_1$ or
$k_0$.  

\begin{lemma}\label{qbkestiw}
For every $l\in\N$ and $k\in\{1,\dots,k_l\}$,
\[
q_{k,l}\le (q_{0,l}N^{25dm_1})^{2^k}.
\]
\end{lemma}

\begin{proof}
Since the claim does not depend on the choice of $s\in S$, for notational
simplicity, we assume that $s=\mathbf{0}$. 
Let $k\in\{1,\dots,k_l\}$. If
$Q\in Q_n$ is independently $(k,\bL^l(m_1,k_0),1,m_0)$-good, it contains at most
one independently $(k-1,\bL^l(m_1,k_0),1,m_0)$-bad cube
$Q'\in\Q_{n+\Delta_k^l}(Q)$.
By independence (recall Remark~\ref{iwlevelsremark}), we have that
\begin{align*}
1-q_{k,l}&=(1-q_{k-1,l})^{N^{d\Delta_k^l}}+N^{d\Delta_k^l}q_{k-1,l}
         (1-q_{k-1,l})^{N^{d\Delta_k^l}-1}\\
  &=(1-q_{k-1,l})^{N^{d\Delta_k^l}-1}(1-q_{k-1,l}+N^{d\Delta_k^l}q_{k-1,l})\\
  &\ge (1-(N^{d\Delta_k^l}-1)q_{k-1,l})(1+(N^{d\Delta_k^l}-1)q_{k-1,l})
   \ge 1-(N^{d\Delta_k^l}q_{k-1,l})^2,
\end{align*}  
implying that
\begin{equation}\label{badrecursion}
q_{k,l}\le(N^{d\Delta_k^l}q_{k-1,l})^2.
\end{equation}
Iterating Inequality~\eqref{badrecursion} and recalling
Inequality~\eqref{stepestimate}, we conclude that
\begin{align*}
q_{k,l}&\leq N^{d\sum_{j=0}^{k-1}2^{j+1}\Delta_{k-j}^l}\cdot q_{0,l}^{2^k}=q_{0,l}^{2^k}
       N^{2d2^k\sum_{j=1}^k2^{-j}\Delta_j^l}\\
  &\le q_{0,l}^{2^k} N^{10d2^km_1\sum_{j=1}^k\sqrt 2^{-j}}<(q_{0,l}N^{25dm_1})^{2^k}.
\end{align*}
\end{proof}

We proceed by defining the concept of hereditarily good cubes, which
will enable us to apply our length gain estimates repeatedly.

\begin{definition}\label{hereditary}
Fix $\omega\in\Omega$ and $m_0,\mfc,k_0,m_1\in\N\setminus\{0\}$. Let $n,q\in\N$.
 A cube  $Q\in\Q_n$ is {\it $(q,m_1,k_0,\mfc,m_0)$-hereditarily good}
if $Q$ is $(\bL^l(m_1,k_0),\mfc,m_0)$-good for all $l=0,\dots,q$.
\end{definition}

\begin{remark}\label{hereditaryremark}
(a) Whether a cube $Q\in\Q_n$ is $(q,m_1,k_0,\mfc,m_0)$-hereditarily good or not
depends only on
\[
\bigcup_{l=0}^q\C_{n+ L^l(m_1,k_0)_0  +m_0}(\omega)
 =\bigcup_{l=0}^q\C_{n+(l+1)l_0 +m_0}(\omega)
\]
by Remark~\ref{levelsremark}  and \eqref{*hmr}.

(b) The choice of the parameters $m_{0}$, $\mfc$, $k_{0}$ and $m_{1}$ will be
crucial in our proof. In the proof of Theorem~\ref{maingeneral}, we explain
how they are selected. In
 Sections~\ref{hereditarilyexists}--\ref{lengthincrease} some
restrictions on them are given, see in particular Lemmas~\ref{*iwg},
\ref{colourcontribution} and \ref{smallgain}, Propositions~\ref{hdprob},
\ref{Salphaissmall} and \ref{paintedinhereditarily} and
Construction~\ref{brokenline0}.  
\end{remark}

According to the next proposition, the probability that a cube is hereditarily
good is large when $m_0$ is large.

\begin{proposition}\label{hdprob}
Let  $0\le p<1$ and $k_0,m_1\in\N\setminus\{0\}$ with $k_0\ge 3$.   For
every $\e>0$, there exists $m_0=m_0(m_1,p,d,N,\e)\in\N$ such that
\[
\P(\{\omega\in\Omega\mid Q\text{ is }(q,m_1,k_0,\mfc,m_0)\text{-hereditarily
  good}\})\ge 1-\e
\]
for every $n,q\in\N$ and $Q\in\Q_n([N^{-n},1-N^{-n}]^d)$, where $\mfc$ is as
in Lemma~\ref{*iwg}.
\end{proposition} 

\medskip

\begin{proof}
From Lemma~\ref{*iwg} we obtain that
\[
\P(\{\omega\in\Omega\mid Q\text{ is }(\bL^l(m_1,k_0),\mfc,m_0)\text{-good}\})
  \ge 1-(2^d+1)q_{k_l,l}
\]
for all $l=0,\dots,q$. Further, by Remark~\ref{levelsremark}, the events
``$Q$ is $(\bL^l(m_1,k_0),\mfc,m_0)$-good'' and
``$Q$ is $(\bL^t(m_1,k_0),\mfc,m_0)$-good'' are independent for $l\ne t$ and,
by Proposition~\ref{probgoodiw}, we have $\lim_{m_0\to\infty}q_{0,l}=0$.  Recall
that $q_{0,l}$ does not depend on $l$, $m_1$ or $k_0$.  Thus,
 choosing sufficiently large $m_0\in\N$, we can
make $q_{0,l}N^{25dm_1}$ as small as we wish.  Combining \eqref{j-leva} and
\eqref{j-levb} with  Lemma~\ref{qbkestiw} leads to
\begin{align*}
\P&(\{\omega\in\Omega\mid Q\text{ is }
   (q,m_1,k_0,\mfc,m_0)\text{-hereditarily good}\})\\
&\ge\prod_{l=0}^q(1-(2^d+1)(q_{0,l}N^{25dm_1})^{2^{k_l}})
 \ge\prod_{l=0}^q(1-(2^d+1)(q_{0,l}N^{25dm_1})^{(l+1)^2})\ge 1-\e.
\end{align*}
\end{proof}

\medskip 

\medskip

Recalling Definition~\ref{properly}, set
\begin{equation}\label{neighbourhood}
k_Q:=\bigcup_{Q'\sim Q}Q'\text{ and }K_Q:=\bigcup_{Q'\sim Q}k_{Q'}
\end{equation}
for all $n\in\N$ and $Q\in\Q_n$, that is, $k_Q$ and $K_Q$ are cubes
with same centre as $Q$ having three and five times the side length of $Q$,
respectively.  We conclude this section with a proposition which states that
the hereditarily good cubes are abundant.

\begin{proposition}\label{Salphaissmall}
Let $0\le p<1$ and  $k_0,m_1\in\N\setminus\{0\}$ with $k_0\ge 3$.  Then
there exists a positive integer
$m_0=m_0(m_1,p,d,N)$ such that, for $\P$-almost all $\omega\in\Omega$, 
\begin{align*}
\dimh\Bigl\{x\in [0,1]^d\mid\, &\#\{n\in\N\mid Q'\text{ is }
   (q,m_1,k_0,\mfc,m_0)\text{-hereditarily good}\\
  &\text{for all }Q'\in\Q_n(K_{Q_n(x)})\}<\infty\Bigr\}<1
\end{align*}
for all $q\in\N$, where $\mfc$ is defined in Lemma~\ref{*iwg}.
\end{proposition}

\begin{proof}
Combining Proposition~\ref{hdprob} with Proposition~\ref{smallsets}
gives the claim (recall the proof of Theorem~\ref{mainstandard}). 
\end{proof}

\section{Pure $\alpha$-unrectifiability}\label{general}

In Section~\ref{standard}, we proved that a typical realisation of the fractal 
percolation is purely 1-unrectifiable, that is, all Lipschitz curves intersect 
the  fractal percolation in a set of zero 1-dimensional Hausdorff measure.
In view of~\cite{BCJM} (see Section~\ref{intro}), it is natural to
ask whether Lipschitz curves can be replaced by $\alpha$-H\"older curves,
and $\H^1$ by $\H^{\frac 1\alpha}$, for some $\alpha<1$. We define a concept of
$\alpha$-rectifiability for the purpose of answering this question positively.
Throughout the section, $I\subset\R$ is a generic closed and bounded
interval. 

\begin{definition}
Let $0<\alpha\le 1$ and $H\ge 0$. 
A curve $\gamma\colon I\to\R^d$ is {\it $(H,\alpha)$-H\"older (continuous) at 
$a\in I$}, if it satisfies the condition
\begin{equation*}
|\gamma(a)-\gamma(b)|\le H|a-b|^\alpha
\end{equation*}
for every $b\in I$. A curve $\gamma$ is {\it $(H,\alpha)$-H\"older}, if it 
is $(H,\alpha)$-H\"older at every $a\in I$. Finally, a curve $\gamma$ is 
{\it $\alpha$-H\"older}, if for every $a\in I$ there is $H_a\ge 0$ such that
$\gamma$ is $(H_a,\alpha)$-H\"older at $a\in I$.
\end{definition}

The following well-known lemma is an immediate corollary of definitions.  

\begin{lemma}\label{alphameasure}
Let $\gamma\colon I\to\R^d$ be an $(H,\alpha)$-H\"older curve for some 
$0<\alpha\le 1$ and $H\ge 0$. Assume that $A\subset I$. Then
$\H^{\frac 1\alpha}(\gamma(A))\le H^{\frac 1\alpha}\H^1(A)$.
\end{lemma}


According to the next lemma, every $\alpha$-H\"older curve can be covered by
a countable collection of $(H,\alpha)$-H\"older curves.

\begin{lemma}\label{holdercover}
Let $\gamma\colon I\to\R^d$ be an $\alpha$-H\"older curve for some
$0<\alpha\le 1$. Then there is a countable collection of curves
$\gamma_i\colon I\to\R^d$, $i\in\N$, such that $\gamma_i$ is
$(i,\alpha)$-H\"older and 
\[
\gamma(I)\subset\bigcup_{i=1}^\infty\gamma_i(I).
\]
\end{lemma}

\begin{proof}
For all $i\in\N$, let $F_i\subset I$ be the set of points where $\gamma$ is 
$(i,\alpha)$-H\"older. By definition, $I=\bigcup_{i=1}^\infty F_i$
and, therefore, it is enough to show that $\gamma(F_i)$ can be covered by
an $(i,\alpha)$-H\"older curve $\gamma_i\colon I\to\R^d$.

Fix $i\in\N$. Suppose that $(t_n)_{n\in\N}$ is a sequence in $F_i$
converging to $t\in I$. For all $u\in I$ and $n\in\N$, we have
\[
|\gamma(t)-\gamma(u)|\le |\gamma(t)-\gamma(t_n)|+|\gamma(t_n)-\gamma(u)|
  \le i|t-t_n|^\alpha+i|t_n-u|^\alpha
  \xrightarrow[n\to\infty]{} i|t-u|^\alpha,
\]
implying that $t\in F_i$. Thus $F_i$ is closed for all $i\in\N$. Since the 
complement of $F_i$ is open, we can write it as 
$F_i^c=\bigcup_{j=1}^\infty]a_j^i,b_j^i[$, where the intervals are disjoint.
(The case of finite union is included by adding infinitely many empty sets.)
Define $\gamma_i\colon I\to\R^d$ as $\gamma_i(t):=\gamma(t)$ for $t\in F_i$, and
on each interval $]a_j^i,b_j^i[$ with $j\in\N$, define $\gamma_i$ as the affine
map connecting $\gamma(a_j^i)$ and $\gamma(b_j^i)$.

We verify that $\gamma_i$ is $(i,\alpha)$-H\"older. Letting
$t,u\in I$ with $t<u$, we need to prove that
$|\gamma_i(t)-\gamma_i(u)|\le i|t-u|^\alpha$. This is
trivial if $t,u\in F_i$. Assume that $t\in F_i$ and $u\not\in F_i$. Then 
$u\in\mathopen]a_j^i,b_j^i\mathclose[$
for some $j\in\N$. Considering the functions 
$g_t,f_t\colon [t,b_j^i]\to[0,\infty[$, $g_t(s)=i|t-s|^\alpha$ and 
$f_t(s)=|\gamma_i(t)-\gamma_i(s)|$, it suffices to 
show that $f_t(s)\le g_t(s)$ for all $s\in\mathopen]a_j^i,b_j^i\mathclose[$. 
This follows from the concavity of $g_t$, since $f_t(a_j^i)\le g_t(a_j^i)$, 
$f_t(b_j^i)\le g_t(b_j^i)$ and $f_t$ is affine on $]a_j^i,b_j^i[$. By symmetry,
$f_t(u)\le g_t(u)$ when $u\in F_i$ and $t\not\in F_i$. Finally, let
$t\in\mathopen]a_k^i,b_k^i\mathclose[$ and 
$u\in\mathopen]a_j^i,b_j^i\mathclose[$ for some $k,j\in\N$. Since
$a_j^i,b_j^i\in F_i$, we have that $f_t(a_j^i)\le g_t(a_j^i)$ and
$f_t(b_j^i)\le g_t(b_j^i)$. Thus, concavity of $g_t$ 
and affinity of $f_t$ on $]a_j^i,b_j^i[$ imply that $f_t(u)\le g_t(u)$, 
completing the proof. 
\end{proof}

In view of Lemmas \ref{alphameasure} and \ref{holdercover}, the following
definition is natural.

\begin{definition}\label{alpharectifiable}
Let $0<\alpha\le 1$. A set $A\subset\R^d$ is {\it $\alpha$-rectifiable}, if
there exist $\alpha$-H\"older curves $\gamma_i\colon I\to\R^d$, $i\in\N$, such
that $\H^{\frac 1\alpha}(A\setminus(\bigcup_{i=1}^\infty\gamma_i(I)))=0$. A set 
$A\subset\R^d$ is {\it purely $\alpha$-unrectifiable}, if 
$\H^{\frac 1\alpha}(A\cap\gamma(I))=0$ for all $\alpha$-H\"older curves 
$\gamma\colon I\to\R^d$.
\end{definition}

\begin{remark}\label{alphaisone}
When $\alpha=1$, the above definition agrees with the standard
definition of 1-rectifiability and pure 1-unrectifiability.
\end{remark}

According to Lemma~\ref{alphameasure}, the images of $(H,\alpha)$-H\"older
curves have finite $\H^{\frac 1\alpha}$-measure. One may address the question
whether the images have
always positive $\H^{\frac 1\alpha}$-measure. The answer is negative, since
any $\beta$-H\"older curve is $\alpha$-H\"older for all $\alpha<\beta$ and
$\H^{\frac 1\beta}(\gamma(I))<\infty$ implies $\H^{\frac 1\alpha}(\gamma(I))=0$.
To avoid problems caused by this fact, we give the following definition.
  

\begin{definition}\label{tight}
Let $\eta,R>0$ and $0<\alpha\le 1$. A curve $\gamma\colon I\to\R^d$ is 
{\it $(\alpha,\eta,R)$-tight at $t\in I$} if
\begin{equation*}
\frac{\diam\gamma([t-r,t+r])}{(2r)^\alpha}\ge\eta
\end{equation*}
for all $r\le R$. We say that $\gamma$ is 
{\it $(\alpha,\eta,R)$-tight} if it is $(\alpha,\eta,R)$-tight at
every $t\in I$. A curve $\gamma$ is {\it $(\alpha,\eta)$-tight at 
$t\in I$} if it is $(\alpha,\eta,R)$-tight at $t\in I$ for some $R>0$. Finally,
a curve $\gamma$ is {\it $(\alpha,\eta)$-tight} if it is 
$(\alpha,\eta)$-tight at every $t\in I$. 
\end{definition}

\begin{remark}\label{tightexists}
There exist tight H\"older curves. For example, the natural parametrisation 
$\gamma\colon [0,1]\to\R^2$ of the von Koch curve is 
$(\frac{\log 3}{\log 4},\frac 1{\sqrt 3^3},1)$-tight since, for all 
$t\in[0,1]$, $n\in\N$ and $4^{-n}\le r<4^{-n+1}$, the set $\gamma([t-r,t+r])$ 
contains points whose distance is $3^{-n}$. It is also
$(1,\frac{\log 3}{\log 4})$-H\"older continuous. Modifying this example, one 
can construct $(\alpha,\eta,1)$-tight $(H,\alpha)$-H\"older curves for any
$0<\alpha,\eta\le 1$ and $H>0$.
\end{remark}

Next lemma states that a H\"older curve is tight at most points.

\begin{lemma}\label{mosttight}
Letting $\gamma\colon I\to\R^d$ be an $(H,\alpha)$-H\"older curve, define
\[
A_0:=\bigcap_{\eta>0}\{t\in I\mid\gamma\text{ is not }
      (\alpha,\eta)\text{-tight at }t\}.
\]
Then $\H^{\frac 1\alpha}(\gamma(A_0))=0$.
\end{lemma}

\begin{proof}
For all $i\in\N$, set 
\[
A_i:=\Big\{t\in I\mid\liminf_{r\searrow 0}\frac{\diam\gamma([t-r,t+r])}{(2r)^\alpha}
    <\frac 1i\Big\}.
\]
Then $A_{i+1}\subset A_i$ for all $i\in\N$ and $A_0=\bigcap_{i=1}^\infty A_i$.
So it is enough to show that $\lim_{i\to\infty}\H^{\frac 1\alpha}(\gamma(A_i))=0$.
Fixing $i\in\N$ and using Vitali's covering theorem 
\cite[Theorem 2.8]{M}, we find, for all $\delta>0$, disjoint balls 
$B_j\subset I$, $j\in\N$, such that $\diam B_j\le\delta$, 
$\H^1(A_i\setminus\bigcup_{j=1}^\infty B_j)=0$ and 
$\diam\gamma(B_j)\le\frac 1i(\diam B_i)^\alpha$. By Lemma~\ref{alphameasure},
$\H^{\frac 1\alpha}(\gamma(A_i\setminus\bigcup_{j=1}^\infty B_j))=0$ which implies
$\H_\e^{\frac 1\alpha}(\gamma(A_i\setminus\bigcup_{j=1}^\infty B_j))=0$ for all
$\e>0$. Therefore,
\[
\H_{\frac 1i\delta^\alpha}^{\frac 1\alpha}(\gamma(A_i))
  \le (\tfrac 1i)^{\frac 1\alpha}\sum_{j=1}^\infty\diam B_j
  \le (\tfrac 1i)^{\frac 1\alpha}\diam I,
\]
from which the claim follows.
\end{proof}

Next lemma guarantees that $\alpha$-H\"older curves can be essentially covered 
by tight curves.

\begin{lemma}\label{tightcover}
Let $\gamma\colon I\to\R^d$ be an $(H,\alpha)$-H\"older curve. Then there exist
sequences $(H_i)_{i\in\N}$, $(\eta_i)_{i\in\N}$ and $(R_i)_{i\in\N}$ of
positive real
numbers and $(\alpha,\eta_i,R_i)$-tight $(H_i,\alpha)$-H\"older curves
$\gamma_i\colon I\to\R^d$, $i\in\N$, such that
\[
\H^{\frac 1\alpha}\bigl(\gamma(I)\setminus\bigcup_{i=1}^\infty\gamma_i(I)\bigr)=0.
\]
\end{lemma}

\begin{proof}
For $\eta,R>0$, set
\[
F_{\eta,R}:=\{t\in I\mid\gamma\text{ is }(\alpha,\eta,R)\text{-tight at }t\}.
\]
By Lemma~\ref{mosttight},
\[
\H^{\frac 1\alpha}\bigl(\gamma(I\setminus\bigcup_{i,j=1}^\infty F_{\frac 1i,\frac 1j})
  \bigr)=0.
\]
Therefore, it suffices to show that, for all $\eta,R>0$, the set
$\gamma(F_{\eta,R})$ can be covered by an $(\alpha,\eta',R)$-tight
$(H',\alpha)$-H\"older curve for some $\eta',H'>0$. To that end, 
let $(t_n)_{n\in\N}$ be a sequence in $F_{\eta,R}$ tending to $t\in I$. Let
$0<r\le R$ and define $r_n:=r-|t-t_n|$ for all $n\in\N$. Then
\[
\diam\gamma([t-r,t+r])\ge\lim_{n\to\infty}\diam\gamma([t_n-r_n,t_n+r_n])
 \ge\eta(2r)^\alpha.
\]
Therefore, $F_{\eta,R}$ is closed. Write
$F_{\eta,R}^c=\bigcup_{n\in\N}\mathopen]a_n,b_n\mathclose[$, where the union is
disjoint. Define a curve $\gamma_\eta\colon I\to\R^d$ by setting
$\gamma_\eta(t):=\gamma(t)$ for all $t\in F_{\eta,R}$ and, on each interval
$]a_n,b_n[$ with $n\in\N$, define $\gamma_\eta$ as an
$(\alpha,\eta,\frac 12|a_n-b_n|)$-tight $(H,\alpha)$-H\"older curve connecting
    $\gamma(a_n)$ and $\gamma(b_n)$ (recall Remark~\ref{tightexists}).

Let $t\in I$ and $r\le R$. Choose a constant $K>1$ (to be fixed later). If 
\[
|a_{n_0}-b_{n_0}|:=\max\{|a_n-b_n|\mid [a_n,b_n]\subset [t-r,t+r]\}\ge\tfrac rK,
\]
then, by the $(\alpha,\eta,\frac 12|a_n-b_n|)$-tightness of
$\gamma|_{[a_n,b_n]}$, we have that
\[
\diam\gamma_\eta([t-r,t+r])\ge\diam\gamma_\eta([a_{n_0},b_{n_0}])
 \ge\eta\bigl(2\tfrac r{2K}\bigr)^\alpha=:\eta'(2r)^\alpha.
\]
Otherwise, for every $u\in[t-r,t+r]$, there is $s_u\in [t-r,t+r]\cap F_{\eta,R}$
with $|s_u-u|<\frac rK$. By H\"older continuity,
$|\gamma(u)-\gamma_\eta(s_u)|=|\gamma(u)-\gamma(s_u)|
  \le H\bigl(\frac rK\bigr)^\alpha$,
so
\begin{align*}
\diam\gamma_\eta([t-r,t+r])
  &\ge\diam\gamma([s_t-(1-\tfrac 1K)r,s_t+(1-\tfrac 1K)r])-2H(\tfrac rK)^\alpha\\
&\ge\eta(2(1-\tfrac 1K)r)^\alpha-2H(\tfrac rK)^\alpha=:\eta''(2r)^\alpha.
\end{align*}
Solving $K$ from the equation
\[
\frac\eta{(2K)^\alpha}=\eta'=\eta''=\eta\Bigl(\bigl(1-\frac 1K\bigr)^\alpha
  -\frac{2H}{\eta(2K)^\alpha}\Bigr),
\]
we obtain that $K=1+\frac 12(1+\frac{2H}\eta)^{\frac 1\alpha}$, which leads to
$\eta'=\eta''=\frac\eta{\bigl(2+(1+\frac{2H}\eta)^{\frac 1\alpha}\bigr)^\alpha}$.
Thus, $\gamma_\eta$ is $(\alpha,\eta',R)$-tight at $t$.

Finally, we prove that $\gamma_\eta$ is $(3H,\alpha)$-H\"older continuous. Let
$t,s\in I$ with $t<s$. If $t,s\in [a_n,b_n]$ for some $n\in\N$, then
$|\gamma_\eta(t)-\gamma_\eta(s)|\le H|t-s|^\alpha$ by construction. Otherwise, set
$t':=\min\{u\in F_{\eta,R}\mid u\ge t\}$ and
$s':=\max\{u\in F_{\eta,R}\mid u\le s\}$. Then
\begin{align*}
|\gamma_\eta(t)-\gamma_\eta(s)|&\le |\gamma_\eta(t)-\gamma_\eta(t')|+
     |\gamma(t')-\gamma(s')|+|\gamma_\eta(s')-\gamma_\eta(s)|\\
   &\le 3H|t-s|^\alpha.
\end{align*}
Thus, $\gamma_\eta$ is $(3H,\alpha)$-H\"older continuous.
\end{proof}

Next we define a concept of an $\e$-gap which depends on $\omega\in\Omega$.

\begin{definition}\label{gapdef}
Let $\gamma\colon I\to\R^d$ be a curve and let $\e>0$. Fix $\omega\in\Omega$.
For $r>0$, we say that {\it $\gamma$ has an $\e$-gap of scale $r$ at
$t\in I$} if there are $\tilde a,\tilde b\in [t-r,t+r]$ such that
$\gamma(]\tilde a,\tilde b[)\cap E(\omega)=\emptyset$ 
and $|\gamma(\tilde a)-\gamma(\tilde b)|\ge\e\diam\gamma([t-r,t+r])$.
\end{definition}

The following lemma shows that around typical points of tight H\"older curves
there are no $\e$-gaps of small scale.
  
\begin{lemma}\label{smallgaps}
Let $\gamma\colon I\to\R^d$ be an $(H,\alpha)$-H\"older and
$(\alpha,\eta)$-tight curve. Fix $\e>0$ and $\omega\in\Omega$. Denote by $A$ 
the set of $t\in I$ such that $\gamma$ has
$\e$-gaps of arbitrarily small scales at $t$. Then 
$\H^1\bigl(A\cap\gamma^{-1}(E(\omega))\bigr)=0$ and 
$\H^{\frac{1}{\alpha}}(\gamma(A)\cap E(\omega))=0$.
\end{lemma}

\begin{proof}
Fix $t\in A$. There exist a sequence $(r_i)_{i\in\N}$ of positive real numbers
tending to zero and points $\tilde a_i,\tilde b_i\in [t-r_i,t+r_i]$ for which 
$\gamma(]\tilde a_i,\tilde b_i[)\cap E(\omega)=\emptyset$ and
$|\gamma(\tilde a_i)-\gamma(\tilde b_i)|\ge\e\diam\gamma([t-r_i,t+r_i])$. Since 
$\gamma$ is $(H,\alpha)$-H\"older and $(\alpha,\eta)$-tight, we have, for all
large enough $i\in\N$, that
\[
|\tilde a_i-\tilde b_i|\ge\Bigl(\frac 1H|\gamma(\tilde a_i)-\gamma(\tilde b_i)|
  \Bigr)^{\frac 1\alpha}\ge\Bigl(\frac{\e}{H}\Bigr)^{\frac 1\alpha}
  \bigl(\diam\gamma([t-r_i,t+r_i])\bigr)^{\frac 1\alpha}
  \ge\Bigl(\frac{\e}{H}\Bigr)^{\frac 1\alpha}\eta^{\frac 1\alpha} 2r_i.
  \]
Combining this with the fact that
$\mathopen]\tilde a_i,\tilde b_i\mathclose[
    \subset I\setminus\gamma^{-1}(E(\omega))$ implies that 
$t$ cannot be a density point of $\gamma^{-1}(E(\omega))$.
This is true for all $t\in A$, so
$\H^1\bigl(A\cap\gamma^{-1}(E(\omega))\bigr)=0$. Since
$\gamma(A)\cap E(\omega)=\gamma\bigl(A\cap\gamma^{-1}(E(\omega))\bigr)$, we have
that 
\[
\H^{\frac 1\alpha}(\gamma(A)\cap E(\omega))\le H^{\frac 1\alpha}
  \H^1\bigl(A\cap\gamma^{-1}(E(\omega))\bigr)=0
\]
by Lemma~\ref{alphameasure}.  
\end{proof}

 The next proposition is the main result about the length gain of our
broken line approximations to H\"older curves.  Its  proof,  which
we postpone until the end of Section~\ref{lengthincrease},  is quite
technical and depends on several lemmas which are in the Appendices. 
 Recall the notation $l_0=m_1(1+2+\dots+k_0)$ from \eqref{*hmr} or
Definition~\ref{j-levels}. 

\begin{proposition}\label{notalphaholder}
Fix $\omega\in\Omega$ and $m_0,\mfc,q\in\N\setminus\{0\}$.  Suppose that 
$k_0,m_1\in\N\setminus\{0\}$ are large enough so that the assumptions in
Lemma~\ref{smallgain} concerning them are satisfied.  Let $n\in\N$ and
$Q\in\Q_n$. Assume that every $Q'\in\Q_n(K_Q)$ is
$(q,m_1,k_0,\mfc,m_0)$-hereditarily good. Further, suppose that
$\gamma\colon [a,b]\to K_Q$ is a curve passing through an
$(n,i)$-layer for some $i\in\{1,\dots,d\}$. Assume that, for all $h=0,\dots,q$,
there are no points $a^h,b^h\in[a,b]$ and $\tilde a^h,\tilde b^h\in[a^h,b^h]$
such that $\diam\gamma([a^h,b^h])\le 5\sqrt d N^{-hl_0-n}$, 
$\gamma(\mathopen]\tilde a^h,\tilde b^h\mathclose[)\cap E(\omega)=\emptyset$
and $|\gamma(\tilde a^h)-\gamma(\tilde b^h)|\ge d^{-1}N^{-m_0-(h+1)l_0-n}$. Then
there exist points $a\le b_1<d_1\le\dots\le b_{2M}<d_{2M}\le b$ such that
\begin{equation}\label{exponentialgain}
\sum_{j=1}^{2M}|\gamma(b_j)-\gamma(d_j)|\ge (1+C_3N^{-2m_0})^{q+1}
  |\gamma(a)-\gamma(b)|,
\end{equation}  
where $M<(C_0N^{l_0})^{q+1}$.
\end{proposition}

We are now ready to prove our main theorem.
  
\begin{theorem}\label{maingeneral}
For all $0\le p<1$, there exists $\alpha_0=\alpha_0(p,d,N)<1$ such that, for 
$\P$-almost all $\omega\in\Omega$, the set $E(\omega)$ is purely
$\alpha$-unrectifiable for all $\alpha_0<\alpha\le 1$.
\end{theorem}

\begin{proof}
By Lemmas~\ref{holdercover} and \ref{tightcover}, it is enough to show that, 
for $\P$-almost all $\omega\in\Omega$, all $(\eta,\alpha,R)$-tight
$(H,\alpha)$-H\"older curves $\gamma\colon I\to\R^d$ with
$\alpha_0<\alpha\le 1$ satisfy
\[
\H^{\frac 1\alpha}(\gamma(I)\cap E(\omega))=0
\]
for all $H,\eta,R>0$.

Clearly, we may assume that $\eta\le 1\le H$. Choose $\mfc=\mfc(d)$ as in
Lemma~\ref{*iwg} and $m_1=m_1(\mfc,d,N)$ as in
Lemma~\ref{smallgain}. Select $m_0=m_0(m_1,p,d,N)$ such that
Proposition~\ref{Salphaissmall} is valid. Finally, let
$k_0=k_0(m_1,m_0,p,d,N)$ be as in Lemma~\ref{smallgain}. Then the parameters
$\mfc$, $m_1$, $m_0$ and $k_0$ satisfy the assumptions of
Proposition~\ref{notalphaholder}.
Fix $H,\eta,R>0$. Let $0<\alpha\le 1$ and assume that
$\gamma\colon I\to\R^d$ is
$(\eta,\alpha,R)$-tight $(H,\alpha)$-H\"older curve. Consider $\omega\in\Omega$
satisfying the conclusion of Proposition~\ref{Salphaissmall}. We show that the
assumption
\begin{equation}\label{positivemeasure}
\H^{\frac 1\alpha}(\gamma(I)\cap E(\omega))>0
\end{equation}
implies that $\alpha$ is bounded away from 1.

To that end,   fix  $q\in\N$ and set $\e:=(5\sqrt dd)^{-1}N^{-m_0-l_0}$.
Let $G$ be the set whose dimension is proved to be less than 1 in
Proposition~\ref{Salphaissmall}. 
By assumption~\eqref{positivemeasure} and Lemma~\ref{alphameasure},
$\H^1(\gamma^{-1}(E(\omega)\setminus G))>0$. Set
\[
B_\varrho:=\{t\in I\mid\forall r\le\varrho,\gamma\text{ has no }
  \e\text{-gaps of scale }r\text{ at }t\}.
\]
By Lemma~\ref{smallgaps}, there exists $\varrho_0>0$ such that
$\H^1\bigl(B_{\varrho_0}\cap\gamma^{-1}(E(\omega)\setminus G)\bigr)>0$.
Let $t_0\in D:=B_{\varrho_0}\cap\gamma^{-1}(E(\omega)\setminus G)$ be a density
point of $D$. For all $0<\delta<1$,
there exists $0<r_\delta\le\min\{R,\varrho_0\}$ such that, for all
$r\le r_\delta$ and $s\in[t_0-r,t_0+r]$,
we have $[s-\delta r,s+\delta r]\cap D\ne\emptyset$. Fix $0<\alpha_1<\alpha$
such that $\alpha_1<\alpha_0$, where $\alpha_0$ is defined later.
Choose $\delta:=\bigl((20\sqrt dH^2)^{-1}\eta^2N^{-ql_0}\bigr)^{\frac 1{\alpha_1}}$.
Let $r\le r_\delta$ and assume that
$\diam\gamma([t_0-r,t_0+r])\le 5\sqrt d N^{-n}$.
If $\mathopen]c,d\mathclose[\subset [t_0-r,t_0+r]\setminus D$, then
$\mathopen]c,d\mathclose[\subset [s-\delta r,s+\delta r]$ for some
$s\in[t_0-r,t_0+r]$, and
\begin{align*}
\diam\gamma([c,d])&\le\diam\gamma([s-\delta r,s+\delta r])
   \le H\delta^\alpha(2r)^\alpha\le\tfrac H\eta
   \delta^\alpha\diam\gamma([t_0-r,t_0+r])\\
 &\le\tfrac H\eta\bigl((20\sqrt dH^2)^{-1}\eta^2N^{-ql_0}
   \bigr)^{\frac \alpha{\alpha_1}}5\sqrt d N^{-n}<\tfrac\eta{4H}N^{-ql_0-n}.
\end{align*}
In particular, for all $c,d\in[t_0-r,t_0+r]$, we have that
\begin{equation}\label{denseset}
\diam\gamma([c,d])\ge\tfrac\eta{4H}N^{-ql_0-n}\implies [c,d]\cap D\ne\emptyset.
\end{equation}

By Proposition~\ref{Salphaissmall}, there exist $0<r_1\le r_\delta$ and $n\in\N$
such that every $Q'\in\Q_n(K_{Q_m(\gamma(t_0))})$ is
$(q,m_1,k_0,\mfc,m_0)$-hereditarily good,
$\gamma([t_0-r_1,t_0+r_1])\subset K_{Q_m(\gamma(t_0))}$,
$\gamma|_{[t_0-r_1,t_0+r_1]}$ passes through an $(n,i)$-layer for some
$i\in\{1,\dots,d\}$ and $\gamma|_{[t_0-r_1,t_0+r_1]}$ is an
$(\eta,\alpha,r_1)$-tight $(H,\alpha)$-H\"older curve. Suppose that there are
$t\in [t_0-r_1,t_0+r_1]$ and $r<r_1$ such that
$N^{-hl_0-n}\le\diam\gamma([t-r,t+r])\le 5\sqrt dN^{-hl_0-n}$
for some $h\in\{0,\dots,q\}$. Since
$N^{-hl_0-n}\le\diam\gamma([t-r,t+r])\le H(2r)^\alpha$ and $h\le q$, we have 
\[
\diam\gamma([\tilde t,\tilde t+\tfrac r2])\ge\eta(\tfrac r2)^\alpha
   \ge\tfrac\eta{4^\alpha H}N^{-ql_0-n}\ge\tfrac\eta{4H}N^{-ql_0-n}
\]
for all $\tilde t\in [t-r,t+r]$. By \eqref{denseset}, we conclude that there are
$\tilde t_1\in [t-\frac r2,t]\cap D$ and $\tilde t_2\in [t,t+\frac r2]\cap D$.
Since $\tilde t_1\in B_{\varrho_0}$, we have, for all
$\tilde a,\tilde b\in [t-r,\tilde t_1+(\tilde t_1-(t-r))]$ with
$\gamma(\mathopen]\tilde a,\tilde b\mathclose[)\cap E(\omega)=\emptyset$, that
\begin{align*}    
|\gamma(\tilde a)-\gamma(\tilde b)|&<\e\diam
   \gamma([\tilde t_1-(\tilde t_1-(t-r)),\tilde t_1+(\tilde t_1-(t-r))])\\
 &\le\e 5\sqrt dN^{-hl_0-n}\le\tfrac{5\sqrt d}{5\sqrt dd}N^{-m_0-l_0}N^{-hl_0-n},
\end{align*}
and similarly for $\tilde a,\tilde b\in [\tilde t_2-(t+r-\tilde t_2),t+r]$.
Since
\[
[t-r,t+r]\subset [t-r,\tilde t_1+(\tilde t_1-(t-r))]
\cup [\tilde t_2-(t+r-\tilde t_2),t+r],
\]
$\gamma|_{[t_0-r_1,t_0+r_1]}$ satisfies the assumptions of
Proposition~\ref{notalphaholder}.

Set $L:=C_0N^{l_0}$, and denote by $a_i$, $i=1,\dots,\widetilde M+1$, the
increasing sequence of division points given by
Proposition~\ref{notalphaholder} including the points $a_1=t_0-r_1$ and
$a_{\widetilde M+1}=t_0+r_1$. If $\widetilde M<3L^{q+1}$, add extra division points
to obtain points $a_1,\dots,a_f=t_0+r_1$ with $f-1=3L^{q+1}$.  We remind that
the points $a_1$ and $a_f$ depend on $q$ via the definition of $\delta$.  By
Proposition~\ref{notalphaholder}, $(H,\alpha)$-H\"older continuity, Jensen's
inequality and $(\eta,\alpha,r_1)$-tightness, we obtain that
\begin{align*}
&(1+C_3N^{-2m_0})^{q+1}|\gamma(a_1)-\gamma(a_f)|\le\sum_{i=1}^{3L^{q+1}}
  |\gamma(a_i)-\gamma(a_{i+1})|\le H\sum_{i=1}^{3L^{q+1}}|a_i-a_{i+1}|^\alpha\\
&\le H3L^{q+1}(3^{-1}L^{-q-1}|a_1-a_f|)^\alpha=H3^{1-\alpha}L^{(1-\alpha)(q+1)}
  |a_1-a_f|^\alpha\\
&\le\eta^{-1}H3^{1-\alpha}L^{(1-\alpha)(q+1)}\diam(\gamma([a_1,a_f]))
   \le\eta^{-1}H3L^{(1-\alpha)(q+1)}5\sqrt d|\gamma(a_1)-\gamma(a_f)|.
\end{align*}
Hence $(1+C_3N^{-2m_0})^{q+1}\le C(H,\eta,d)(L^{1-\alpha})^{q+1}$, which is a
contradiction for large $q$ provided that
$(C_0N^{l_0})^{1-\alpha}=L^{1-\alpha}<1+C_3N^{-2m_0}$. Therefore, $\alpha_0$ can be
chosen to be the solution of the equation
\[
(C_0N^{l_0})^{1-\alpha}=1+C_3N^{-2m_0}.
\]
\end{proof}

To conclude, we pose a natural open question related to the results of Broman
et al.~\cite{BCJM} described in the Introduction.

\begin{question}
Is it possible to have $\alpha_0=\beta$, where $\beta$ is the constant obtained
by Broman et al.~in \cite{BCJM}? That is, is it true that
$\mathcal H^{\frac 1\alpha}(E\cap\gamma(I))=0$ for all $\alpha$-H\"older curves
$\gamma\colon I\to\R^d$ with $\beta<\alpha\le 1$?
\end{question}

\medskip

\section{Appendix A: Definition of  zoom levels}\label{appendixa}

In this appendix, we construct the sequences $\bL^i:=\bL^i(m_1,k_0)$, $i\in\N$,
used in Sections~\ref{hereditarilyexists}--\ref{lengthincrease} to
determine the appropriate zoom levels,
 and prove their basic properties. 
The explicit construction of the sequences $\bL^i$ is given in
Definition~\ref{j-levels} below, see also Figure~\ref{fig:levels}.
Note  that, according to Definition~\ref{levels}, the sequence $\bL$ determines
the levels with information about the distribution of bad cubes and
$\Delta_j:=L_{j-1}-L_j$ is the number of levels between a
$(j-1,\bL,\mfc,m_0)$-bad cube and its $(j,\bL,\mfc,m_0)$-good parent cube.
In Section~\ref{hereditarilyexists}, we  saw  that the larger the step
size $\Delta_j$, the smaller the probability that a given cube at level
$L_{j-1}$ is $(j-1,\bL,\mfc,m_0)$-bad. Hence, we will consider sequences with
increasing $\Delta_j$. However, increasing $\Delta_j$ increases the total number
of subcubes at level $L_{j-1}$ which, in turn, increases the probability that
some cubes at level $L_{j-1}$ are $(j-1,\bL,\mfc,m_0)$-bad and, therefore,
a balance between these two competing phenomena will be needed. 

 In addition to probability estimates, the sequences $\bL^i$ are used to
determine appropriate scales for broken line approximations of curves.
In Section~\ref{brokenlinesection}, we construct broken line approximations
for curves using information of the distribution of good cubes given by
hereditarily good cubes (see Definition~\ref{hereditary}). In order to do that,
it is essential that the zooming levels determined by different sequences
$\bL^i$ are ``synchronised'', see Figure~\ref{fig:levels}. Since  we are
zooming in as we go down into the fractal set, the level zero is on the top of
the figure and the levels are going downwards. In order to simplify the
notation, we are not adding negative signs to these levels. The basic idea
 is  as follows: We start with the first
basic block of $l_0$ successive levels, which will be covered by a decreasing
sequence $\bL^0$ with linearly increasing step size. The basic block is used to
obtain a macroscopic increase in length for a broken line approximation.
For the purpose of obtaining an exponential increase of length, basic blocks are
utilised iteratively as follows: We proceed by adding the second basic block
below the first one and by defining $\bL^1$ by means of the same  step sizes
 $\Delta_j$ as for $\bL^0$ but starting from the level $2l_0$ instead of
 level  $l_0$. Once the level $l_0$ is reached, we continue increasing
the step size linearly but round up the sizes so that we end up using a subset
of the levels utilised for $\bL^0$. Continue inductively by defining the
sequence $\bL^i$ similarly starting from level $(i+1)l_0$ and using the levels
utilised for $\bL^{i-1}$. For an illustration, see Example~\ref{j-levelsexample}
and Figure~\ref{fig:levels} below.  Once  the step size $l_0$ is
reached, we change our strategy by allowing several steps of the same size
followed by an exponential increase in the step size. This is explained in
Construction~\ref{blocksteps} and represented pictorially in Figure~\ref{figX}.

\begin{definition}\label{j-levels}
Fix  $k_0,m_1\in\N\setminus\{0\}$  and set
$l_0:=m_1(1+2+\dots+k_0)$. For every $i\in\N$, define inductively a finite
decreasing sequence $\bL^i(m_1,k_0):=(L^i(m_1,k_0)_j)_{j=0}^{k_i}$ of nonnegative
integers as follows: Set 
\begin{equation}\label{L0def}
\LL{0}{j}:=l_0-m_1(0+1+2+\dots+j)\text{ for }j=0,1,\dots,k_0.
\end{equation}
 Assume that the sequence $\bL^{i-1}(m_1,k_0)$ is defined for  some
$i\in\N\setminus\{0\}$. In order to  define $\bL^i(m_1,k_0)$,   
we distinguish  two cases.

\emph{Case 1}: $m_1 k_{i-1}<l_0$. In this case, let 
\begin{equation}\label{initialdef1}
\LL{i}{j}:=l_0+\LL{i-1}{j}\text{ for }j=0,\ldots,k_{i-1}.
\end{equation}
 When $j>k_{i-1}$, we define $\LL{i}{j}$ inductively.
If $\LL{i}{j-1}$ is defined, let 
\begin{equation}\label{roundupgeneral}
\begin{aligned}
\LL{i}{j}:=\max\{\LL{i-1}{l} \mid\, &  \LL{i-1}{l}\le\LL{i}{j-1}-m_1 j,\\
  & l\in\{0,\ldots,k_{i-1}\}\}
\end{aligned}
\end{equation}
provided we have $\LL{i}{j}\ge m_1 (j+1)$. Otherwise, set $k_i:=j$ and
let $\LL{i}{j}:=0$ be  the last element of the sequence $\bL^i(m_1,k_0)$.

\emph{Case 2}: $m_1 k_{i-1} \ge l_0$. In this case, let 
\begin{equation}\label{initialdef2}
\LL{i}{j}:=l_0+\LL{i-1}{j}\text{ for }j=0,\ldots,\frac{l_0}{m_1}-1.
\end{equation}
 When $j\ge \frac{l_0}{m_1}$, we define $\LL{i}{j}$
inductively by  setting  
\begin{equation}\label{e:Delta}
\LL{i}{j}:=\LL{i}{j-1} - \DD{i}{j},
\end{equation}
where the values of $\DD{i}{j}$ will be specified later in
Construction~\ref{blocksteps} and  are 
illustrated in Figure~\ref{figX}.
\end{definition}

\begin{example}\label{j-levelsexample}
We calculate the sequences $\bL^0(m_1,k_0)$, $\bL^1(m_1,k_0)$ and
$\bL^2(m_1,k_0)$ when
$k_{0}=8$ and illustrate them in Figure~\ref{fig:levels}. In this case,
$L^0(m_1,k_0)_0=L^0(m_1,8)_0=l_0=36\cdot m_1$, $L^0(m_1,8)_1=35\cdot m_1$,
$L^0(m_1,8)_2=33\cdot m_1$, $L^0(m_1,8)_3=30\cdot m_1$,
$L^0(m_1,8)_4=26\cdot m_1$, $L^0(m_1,8)_5=21\cdot m_1$,
$L^0(m_1,8)_6=15\cdot m_1$, $L^0(m_1,8)_7=8\cdot m_1$ and
$L^0(m_1,8)_8=0\cdot m_1$. These numbers
define the levels in the first column in Figure~\ref{fig:levels}. For
$j=0,\dots,8$, we have that $L^1(m_1,8)_j=36\cdot m_1+L^0(m_1,8)_j$. Since
$36\cdot m_1-9\cdot m_1=27\cdot m_1$ and the largest number in the sequence
$L^0(m_1,8)_j$ not exceeding $27\cdot m_1$ is $26\cdot m_1$, we conclude that
$L^1(m_1,8)_9=26\cdot m_1$. Further, $26\cdot m_1-10\cdot m_1=16\cdot m_1$
giving $L^1(m_1,8)_{10}=15\cdot m_1$, and $15\cdot m_1-11\cdot m_1=4\cdot m_1$
yielding $L^1(m_1,8)_{11}=0\cdot m_1$. So $k_1=11$. This gives the second
column in Figure~\ref{fig:levels}. For the third column in
Figure~\ref{fig:levels}, we calculate $36\cdot m_1-12\cdot m_1=24\cdot m_1$ and
take the largest element in the middle
column not exceeding it (keeping in mind that we are zooming into our fractal,
so in Figure~\ref{fig:levels} levels increase downwards). In this way, we
obtain $L^2(m_1,k_0)_{k_1+1}=L^2(m_1,k_0)_{12}=15\cdot m_1$. Finally,
$15\cdot m_1-13\cdot m_1=2\cdot m_1$
and hence we need to take $L^2(m_1,k_0)_{13}=0\cdot m_1$ and $k_2=13$. 
\end{example}

Before we specify the values of $\DD{i}{j}$ when $k_{i-1}\ge m_1 l_0$ and
$j\ge \frac{l_0}{m_1}$, we introduce the notation
$\DD{i}{j}$ for all $i$ and $j$ in harmony with \eqref{e:Delta}.

\begin{definition}\label{d:Delta}
Let
\[
\DD{i}{j}:=\LL{i}{j-1}-\LL{i}{j}\text{ for all }i\in\N\text{ and }
j=1,\ldots,k_i.
\]
\end{definition}

Note that in Definition~\ref{j-levels} (depending on the parameters $k_0$ and
$m_1$) we already specified the values of $\DD{i}{j}$ when $m_1 k_{i-1}<l_0$ or
$j<\frac{l_0}{m_1}$ and, in order to complete the definition of the sequences
$\bL^i(m_1,k_0)$, we need to specify the values of $\DD{i}{j}$
for the rest of the cases. 

\begin{figure}[h]
\centering{
\resizebox{1\textwidth}{!}{\input{fpcfiglevels.tex}}}
\caption{Definition of $\bL^i(m_1,k_0)$ for $i=0$, 1 and 2, when $k_0=8$.}
\label{fig:levels}
\end{figure}
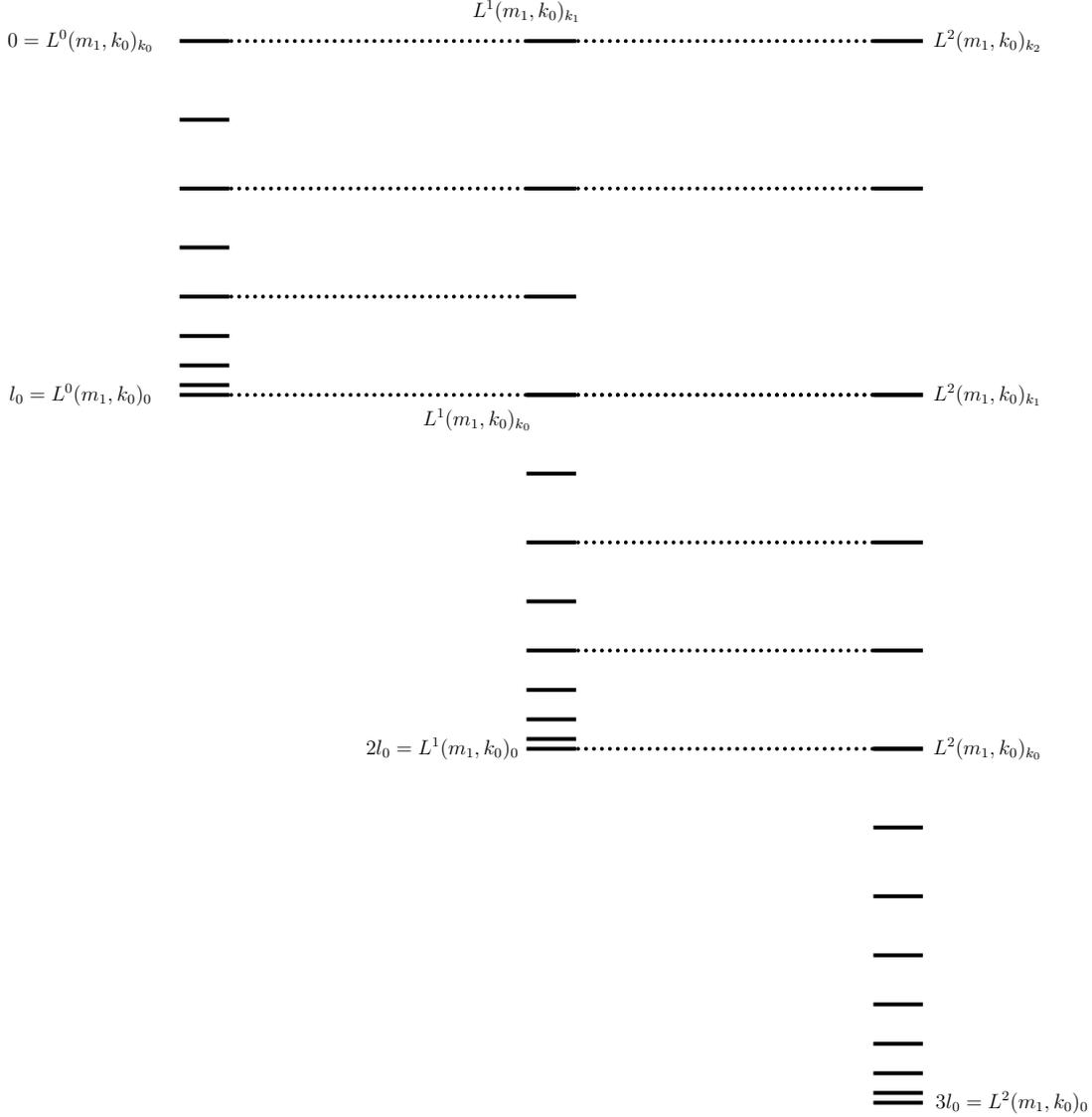

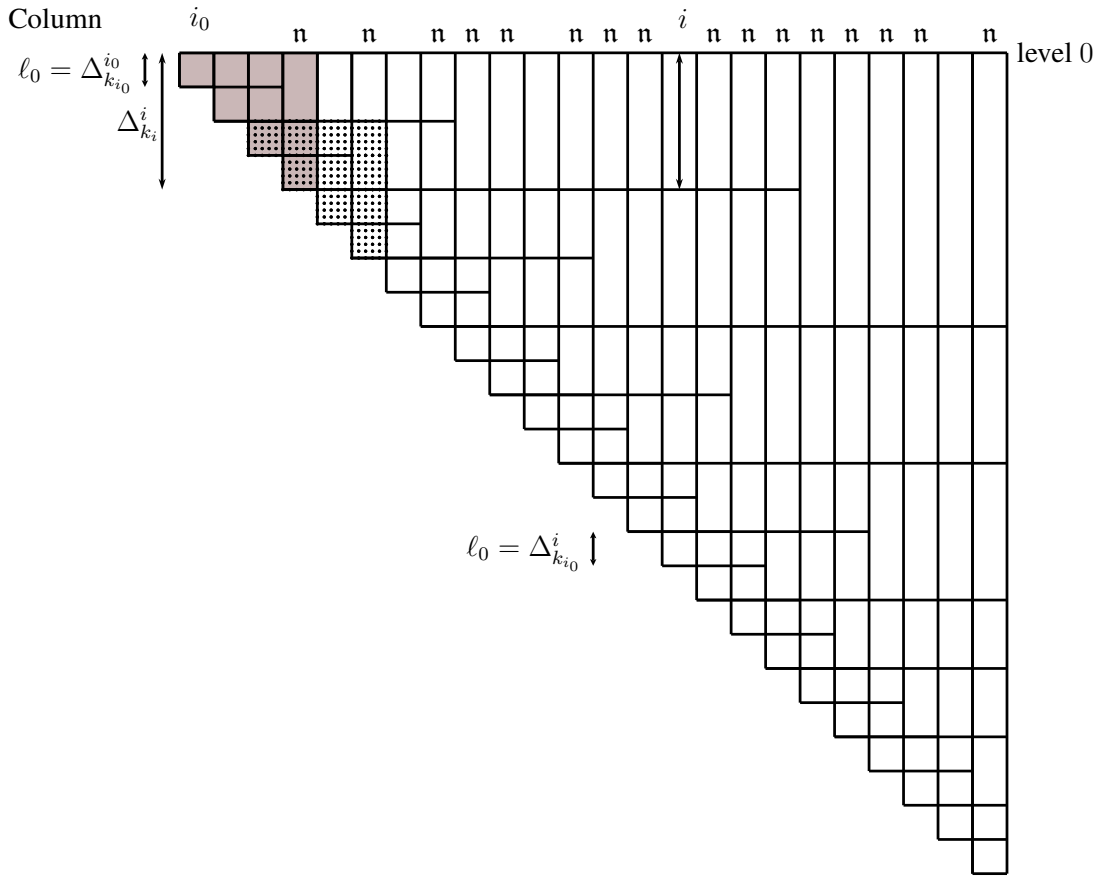
\begin{figure}[h]
\centering{
\resizebox{1\textwidth}{!}{\input{fpcfigX.tex}}}
\caption{Illustration of Construction \ref{blocksteps}.}
\label{figX}
\end{figure}

\begin{construction}\label{blocksteps}
Each column in Figure~\ref{figX} consists of a finite number of rectangles
called blocks. The columns are labelled by $i_0,i_0+1,\dots$ referring to the
sequences defined in Definition~\ref{j-levels}, where $i_0$ is
determined via the formula
\begin{equation}\label{defofi0}
m_1k_{i_0}=l_0.
\end{equation}
 From \eqref{roundupgeneral} we obtain for all $i=0,\dots,i_0$  that
\begin{equation}\label{*dj1}
\Delta_j^i\geq m_1j\text{ for all }j=k_{i-1}+1,...,k_{i},
\end{equation}
where we have used the convention $k_{-1}:=0$. 

In the $i$-th column the heights of the blocks determine $\Delta_j^i$ with
$j= k_{i_0},\dots,k_i$, the lowest and the highest ones being
$\Delta_{k_{i_0}}^i$ and $\Delta_{k_i}^i$, respectively (see Figure~\ref{figX}).
In every column, each block has height of the form $2^nl_0$ for some $n\in\N$,
the lowest block having height $l_0$, and the heights form
a non-decreasing sequence while moving upwards. We enumerate the blocks
such that the lowest one is the first, the second lowest one is the
second etc. Note that the $i$-th column starts from level $(i-i_0+1)l_0$ and
ends at level $0$. 

For all $i\in\mathbb N$ with $i\ge i_0$, let $p_i$ be the number of blocks in
column $i$. We choose the numbers $p_i$ and the heights of the blocks
such that the following properties hold:
\begin{align}
  &\text{Either } p_{i+1}=p_i \text{ or } p_{i+1}=p_i+1\text{ for all } i.
  \label{pchoice}\\
  &\text{In each column the block height is increased by factor 2 from }
  2^nl_0\text{ to }\label{nblock}\\
  &2^{n+1}l_0 \text{ when the inequality } p_i\le 2(n+1)+1
  \text{ would otherwise be violated}\notag.
\end{align}
For a detailed explanation see below. Note that \eqref{nblock} guarantees
that in each column the height of the 
$j$-th block is approximately $\sqrt 2^jl_0$.

Figure~\ref{figX} is constructed as follows: For simplicity, the common factor
$l_0$ is omitted from the notation. Clearly, the column $i_0$ consists of one
block of height $1=2^0$, and the column $i_0+1$ consists of two blocks
of height $1=2^0$, giving $p_{i_0}=1$ and $p_{i_0+1}=2$. In column $i_0+2$ we may
still utilise blocks of height $1=2^0$, giving $p_{i_0+2}=3$, since
$p_{i_0+2}=3\le 2(0+1)+1$ satisfying \eqref{nblock}. However, the column $i_0+3$
cannot consist of four blocks of height 1 since $4>2(0+1)+1$. Instead, we
use two blocks of height 1 and one block of height 2
in accordance with \eqref{pchoice} and \eqref{nblock}, giving $p_{i_0+3}=3$.
In Figure~\ref{figX} we indicate the fact that $p_{i_0+3}=p_{i_0+2}$ by the symbol
$\mfn$ on top of the column $i_0+3$. Note that in the column $i_0+3$ two blocks
of height 1 from the column $i_0+2$ are ``merged'' into one block of height 2.

We proceed by translating the pattern of the four columns obtained so far by
two steps along the diagonal, that is, the block of height 1 in the column
$i_0$ is moved to the lowest block in column $i_0+2$. In Figure~\ref{figX} the
grey shaded part is translated to the dotted one. 
The translated columns $i_0+2$ and $i_0+3$ form a part of the columns
$i_0+4$ and $i_0+5$ which are completed by adding a block of height 2 as the
topmost block according to \eqref{pchoice} and \eqref{nblock}.
Then  $p_{i_0+4}=4=p_{i_0+5}$, as indicated by the symbol $\mfn$ above column
$i_0+5$ in Figure~\ref{figX}.

Next we apply the same translation to the pattern of the six columns constructed
above and complete the columns $i_0+6$ and $i_0+7$ by adding a block of
height 2 as the topmost block in accordance with \eqref{pchoice} and
\eqref{nblock}. This leads to $p_{i_0+6}=5=p_{i_0+7}$, as indicated by $\mfn$
above column $i_0+7$ in Figure~\ref{figX}.

Next we need to modify the translation we are utilising. This is due to the
fact that shifting the column $i_0+6$ onto the lower part of the column
$i_0+8$ would imply $p_{i_0+8}=6$ violating Inequality~\eqref{nblock}
since $6>2\cdot(1+1)+1$. Therefore, we translate
the pattern of the first eight rows by four steps along the diagonal, that is,
the column $i_0$ is moved onto the lowest block in column $i_0+4$, and complete
the incomplete columns by adding a block of height $2^2$ to the top of each of
them in accordance with \eqref{pchoice} and \eqref{nblock}.

We complete the construction inductively by translating the initial pattern by
$2^{n+1}$ steps along the diagonal, whenever utilising the shift of $2^n$ steps
leads to a contradiction with \eqref{nblock}, and by placing blocks of height
$2^{n+1}$ to the top of each incomplete blocks. 
The symbol $\mfn$ is used above every column where $p_i$ is not increased.
These columns are called $\mfn$-columns. $\qed$
\end{construction}

\begin{remark}\label{k0versusk1}
(a) Observe that $l_0=\frac{m_1}2k_0(1+k_0)\approx\frac{m_1}2k_0^2$.
Without the round up process \eqref{roundupgeneral},
by simply defining $\LL{i}{j}$ as $\LL{i}{j-1}-m_1j$
for $j>k_{i-1}$ provided $\LL{i}{j}\ge m_1(j+1)$,
we would have
$l_0=\frac{m_1}2(k_1-k_0)(k_0+1+k_1)\approx\frac{m_1}2(k_1^2-k_0^2)$,
leading to $k_1\approx\sqrt 2k_0$ and, more precisely,
$k_1\le\frac 32 k_0$. Since the round up process \eqref{roundupgeneral}
reduces the number of steps needed to
reach the value 0 for $L^1(m_1,k_0)_j$, the number $k_1$ is, in fact, smaller
than the above calculation indicates. In particular, $k_1\le\frac 32 k_0$.
Similarly, we obtain that $k_2\approx\sqrt 3k_0<2k_0$. 

(b) We claim that,  for all $i\in\N$,
\begin{equation}\label{first1}
\Delta_j^i=m_1j\,\text{ for }j=1,\dots,k_0 \text{ and}
\end{equation}
\begin{equation}\label{first2}
\Delta_j^i \le (i+1)m_1j \le j^2m_{ 1}\,\text{ for } j=k_{i-1}+1,\dots,k_i<k_{i_0}. 
\end{equation}
Indeed, \eqref{first1} follows directly from the definition. The first
inequality in \eqref{first2} follows easily by induction since  
$\Delta_j^i\le m_1j+\Delta_l^{i-1}$ for some $l<j$ by \eqref{roundupgeneral},
see also Figure~\ref{fig:levels}. Since $k_{i'}>k_{i'-1}$ for all $i'\le i_0$ and
$k_0\ge 1$,  we have that $k_i\ge i+1$. This implies  the second inequality in
\eqref{first2}. Further, for $j=k_{i_0}+l$, \eqref{nblock} and \eqref{defofi0}
imply that
\begin{equation}\label{second}
\Delta_{k_{i_0}+l}^i\le\sqrt 2^ll_0=\sqrt 2^lm_1k_{i_0}\le 2\sqrt 2^{k_{i_0}+l}m_1.
\end{equation}
Combining Inequalities~\eqref{first1}, \eqref{first2} and \eqref{second}, we
conclude that  
\begin{equation}\label{stepestimate2}
  \Delta_j^i\le 5\sqrt 2^jm_1\,\text{ for all }i\in\N\text{ and }
  j=1,\dots,k_i,
\end{equation}
that is, \eqref{stepestimate} holds.
\end{remark}

\medskip

Next we estimate the number of blocks of height $2^nl_0$ ending at a given level
and derive lower bounds for $p_i$  and $k_i$.

\medskip

\begin{lemma}\label{samesizenumber} Consider Construction~\ref{blocksteps}
illustrated in Figure~\ref{figX}.  
For every $n,j\in\N$, let $y_j^n$ be the number of blocks of height $2^nl_0$
having upper side at level $jl_0$. Then
\begin{equation}\label{ybound}
y_j^n\le 3\cdot 2^n.
\end{equation}
 Further,
\begin{equation}\label{pibound}
p_i\ge\max\{2\log_2(i-i_0-3)-3,3\}\,\text{ for all }i\ge i_0+3.
\end{equation}
 Finally, provided that $k_0\ge 3$, we have that
\begin{align}
k_i&\ge i+3\,\text{ for all }\,i\le i_0+2\,\text{ and} \label{j-leva2}\\
k_{i_0+l}&\ge 2\log_2(i_0+l+1)\,\text{ for all }\,l\ge 3.\label{j-levb2}
\end{align}
\end{lemma}

\begin{proof}   
By the self-repeating structure of Figure~\ref{figX}, we have $y_j^n\le y_0^n$ for
all $j\in\N$. Thus, it is enough to prove the case $j=0$. As in
Construction~\ref{blocksteps}, for notational simplicity, the common factor
$l_0$ of the heights is omitted in what follows.

Clearly, the claim is true when $n=0$, since $y_0^0=3$. In the case $n\ge 1$, 
the proof is based on counting the number of $\mfn$-columns.
Let $i_n$ be the column where a block of height $2^n$ appears for
the first time. Note that $y_0^n=i_{n+1}-i_n$. By construction, the column $i_n$
is an $\mfn$-column. Let $x_n$ be the number of columns in the
maximal chain of successive $\mfn$-columns including the column
$i_n$. Then $x_1=1$ by Figure~\ref{figX}. 

Since the columns $i_0+3$ and $i_0+5$ are $\mfn$-columns (see
Figure~\ref{figX}), the
self-repeating structure implies that each column $i_0+2i+1$ is an 
$\mfn$-column when $i\ge 1$, that is, the number of blocks is not
increased between the columns $i_0+2i$ and  $i_0+2i+1$ for $i\ge 1$.
In particular, the height of the topmost block is never increased by factor two
in columns $i_0+2i+1$ for $i\ge 2$. (This only happens in column $i_0+3$.)
Furthermore, by Construction~\ref{blocksteps}, introducing a block of height
$2^2$ in column $i_2$ produces three successive $\mfn$-columns
$i_2-1$, $i_2$ and $i_2+1$, that is, $x_2=3$.

Again, by the self-repeating structure of Figure~\ref{figX}, there will be three
successive $\mfn$-columns $i_0+4i-1$, $i_0+4i$ and $i_0+4i+1$ for all $i\ge 2$.
In general, the self-repeating structure implies that there are $x_{n-1}$
successive $\mfn$-columns before and after the column $i_n$, and moreover,
the column $i_n$ is an $\mfn$-column. Therefore, $x_n=2x_{n-1}+1$. Since
$x_1=1$, we conclude that $x_n=2^n-1$ for all $n\ge 1$.

From \eqref{nblock} it follows that $p_{i_n}=2n+1$ and, thus,
$p_{i_{n+1}}=p_{i_n}+2$. Utilising the self-repeating
structure of Figure~\ref{figX}, we deduce that after the column $i_n$ there are
$x_{n-1}$ successive $\mfn$-columns followed by a column which is not an
$\mfn$-column, then $x_n$ successive $\mfn$-columns followed by a column
which is not an $\mfn$-column and after that again $x_n$ successive
$\mfn$-columns. Since $p_{i_{n+1}}=p_{i_n}+2$, the next one is the column $i_{n+1}$.
This implies that
\[
y_0^n=1+x_{n-1}+1+x_n+1+x_n=3+2^{n-1}-1+2(2^n-1)=\frac 52\cdot 2^n,
\]
completing the proof of  \eqref{ybound}. 

We observed above that $p_{i_n+x_{n-1}}=2n+1$ and
$p_{i_n+x_{n-1}+1}=2n+2$. Moreover,
\[
i_n=i_{n-1}+y_0^{n-1}=i_1+\sum_{j=1}^{n-1}y_0^j=i_0+3+\frac 52\sum_{j=1}^{n-1}2^n
   =i_0+3+5(2^{n-1}-1),
\]
giving $i_n+x_{n-1}= i_0+3+6(2^{n-1}-1)$. Writing $i=i_0+3+6(2^{n-1}-1)$, we
conclude that
\[
p_i\ge 2\log_2(\tfrac 16(i-i_0-3)+1)+3\ge 2\log_2(i-i_0-3)-3.
\]
This implies  \eqref{pibound}  since $p_{i_0+3}=3$.

 According to Definition~\ref{j-levels} (see the line after
\eqref{roundupgeneral}), we have that $k_i\ge k_{i-1}+1$ for all $i\le i_0$. This
is true also for $i=i_0+1$ and $i=i_0+2$ by Figure~\ref{figX}. Therefore,
$k_i\ge i+3$ for all $i\le i_0+2$ provided that $k_0\ge 3$, completing the proof
of \eqref{j-leva2}. To prove \eqref{j-levb2}, observe that, by the definition of
$p_i$, we get $k_{i_0+l}=k_{i_0}-1+p_{i_0+l}$ for all $l\in\N$. Thus,
combining \eqref{pibound} and \eqref{j-leva2}, we obtain
\[
k_{i_0+l}\ge i_0+2+\max\{2\log_2(l-3)-3,3\}\,\text{ for all }\,l\ge 3.
\]
Inequality~\eqref{j-levb2} follows from this by elementary calculations using
the fact that condition $k_0\ge 3$ implies that $i_0\ge 3$, which can be easily
checked from Definition~\ref{j-levels} and \eqref{defofi0}.
\end{proof}

\section{Appendix B: Broken line approximations}\label{brokenlinesection}

In this section, we derive one of our main tools -- a special
algorithm for constructing broken line approximations, having exponentially
increasing arc lengths at different scaling levels, for curves $\gamma$ which
are close to the fractal percolation set. Since the algorithm is quite
technical, we first try to give a heuristic outline to prepare the reader
to what is coming up in the  next two sections.

 Our final goal is to prove Proposition \ref{notalphaholder}.  
Definitions~\ref{levels},  \ref{hereditary}  and \ref{j-levels} are utilised
for the purpose of defining the appropriate scaling levels. The
reason behind the length gain is quite simple: If $\gamma$ passes through an
$m_0$-good cube (recall Definition~\ref{good}) and is close to the fractal
percolation set, it has to go around a strongly deleted cube. This increases
the arc length slightly compared to a straight line. The macroscopic increase in
length is a consequence of the fact that $m_0$-good cubes are abundant and
uniformly spread. To achieve this, we need to define several concepts depending
on $\omega\in\Omega$ and based on the definitions of $m_0$-good and $m_0$-bad
cubes.  If we had $m_0$-good cubes everywhere,  the proof  would be
straightforward. However, as illustrated on Figure~\ref{fig41}, there might be 
blue cubes without sufficiently many, sufficiently uniformly distributed
$m_0$-good cubes to guarantee that the lengths of the
broken line approximations  increase at an exponential rate. In
addition to this,  there might also occur  blue  and red  cubes coming from
 various 
scaling levels. We will paint our curve to mark on it the parts where we have
some ``bad'' sections. The parts staying white will be good, the blue sections
are bad, but not too bad, and the red sections are very bad. 

The construction of the broken line approximation is carried through
using five algorithms: colouring cubes (Algorithm \ref{colouring}),
priming curves with colour (Algorithm \ref{priming}),
primed curve modification (Algorithm \ref{modification}),
layer division (Algorithm \ref{layerdivision}) and painting curves
(Algorithm \ref{painting}). Roughly speaking, as a result of a repeated
application of these algorithms, we will end up with a modified painted curve
where sections with various properties are distinguished
by different colours. We also obtain some division points plus some $c$-points
on our curves at different zoom levels. The broken line approximation to our
curve will be determined by the division and the $c$-points. The $c$-points
will be responsible for the length gain of our broken line approximation
as we move to finer and finer broken line approximations. The existence of
these $c$-points is due to the fact that there are gaps in the fractal
percolation set $E(\omega)$ and we work with curves which cannot do large jumps
over gaps in $E(\omega)$. The increase of length will be achieved in white
sections, allowing us to iterate our construction. The blue colour also
indicates that the construction may be iterated in the corresponding sections
even though there is neither length gain, nor length loss. In red sections
we are unable to iterate the construction. These sections will be
disregarded when estimating the length of the broken line approximation.
In Proposition \ref{paintedinhereditarily}  below,  we will verify that
the white and blue sections make up most of our broken line approximation. The
probability estimates  were  given in Section~\ref{hereditarilyexists}.

 Next  we describe a process of colouring cubes that will be
used as a tool for constructing broken line approximations. As the result
of the colouring process, we will have blue cubes and $i$-red cubes for
different values of $i\ge 1$ . One can think of $i$-red colours as
different shades of red.

\begin{algorithm}\label{colouring}{\bf (Colouring cubes).}
Fix $\omega\in\Omega$ and $m_0,\mfc,k_0,m_1\in\N\setminus\{0\}$. Let
$n,q\in\mathbb N$ and let $\bL^i(m_1,k_0)$  and $\Delta_j^i$ be as in
Definitions~\ref{j-levels} and \ref{d:Delta} for all $i=0,\dots,q$ and
$j=1,\dots,k_i$.  Suppose that $I\subset\{0,\dots,q\}$ and
$\tilde k_i\in\{1,\dots,k_i\}$ for all $i\in I$. Assume that $Q\in\Q_n$ is
$(\tilde k_i,\bL^i(m_1,k_0),\mfc,m_0)$-good for all $i\in I$.

If $0\in I$, then $Q$ is $(\tilde k_0,\bL^0(m_1,k_0),\mfc,m_0)$-good and there
are at most $\mfc$ cubes $Q'\in\Q_{n+\Delta_{\tilde k_0}^0}(Q)$ that are
$(\tilde k_0-1,\bL^0(m_1,k_0),\mfc,m_0)$-bad. If such $Q'$ exist, colour them in
blue and attach a blue label $n$ to them to denote the level of the parent cube
$Q$ of $Q'$. In this case, $Q$ is called {\it the blue-labelled parent of $Q'$}
and it is denoted by $\lpb(Q')$.

Similarly, for every $i\in I$ with $i\ge 1$, there are at most $\mfc$ cubes
$Q^{(i+1)}\in\Q_{n+\Delta_{\tilde k_i}^i}(Q)$ that are
$(\tilde k_i-1,\bL^i(m_1,k_0),\mfc,m_0)$-bad. Colour them in $i$-red
and attach an $i$-red label $n$ to them to denote the level of the parent cube
$Q$ of $Q^{(i+1)}$. As above, $Q$ is called {\it the $i$-red-labelled parent of
$Q^{(i+1)}$} and it is denoted by $\lpi(Q^{(i+1)})$. $\qed$
\end{algorithm}

We proceed by giving algorithms for the purpose of constructing
a painted broken line approximation of a curve contained in hereditarily good
cubes. 
The next priming procedure is taking care of the problems which might arise if
our curve and the approximating broken line is going in and out of coloured
``bad'' cubes too many times.

\begin{algorithm}\label{priming}{\bf (Priming curves with colour).}
Let $l,n,m\in\N$ with $m>n$. Let $Q\in\Q_n$ and assume that the cubes
$Q_i\in\Q_m(K_Q)$ for $i=1,\dots,l$ are coloured in the same colour. Let
$\gamma\colon [a,b]\to K_Q$ be a curve. Fix $c\in\mathopen]a,b\mathclose[$. 

{\bf Step 1:}
We denote by $K$ the element of the collection $\{K_{Q_i}\}_{i=1}^l$ that
$\gamma$ enters first. If $K$ does not exist, Step 1 terminates.
If $K$ exists, let
\[
t_1:=\min\{t\in[a,b]\mid\gamma(t)\in K\}. 
\]
If $c\le t_1$ or $c\ge\max\{t\in[a,b]\mid\gamma(t)\in K\}$ or
$\dist(\gamma(c),K)\le N^{-m}$, define
\[
t_2:=\max\{t\in[a,b]\mid\gamma(t)\in K\}.
\]
Otherwise, set
\[
t_2:=\max\{t\in[t_1,c]\mid\gamma(t)\in K\}.
\]
Then $t_1$ is the first entrance time to $K$ whereas $t_2$ is the last
departure time from $K$ (or the last departure time from $K$ before $c$). We
call $t_1$ and $t_2$ priming division points.

{\bf Step 2:}
We proceed by applying Step 1 with $\gamma$ replaced by $\gamma|_{[t_2,b]}$.
This results in priming division points $t_3$ and $t_4$.
Continue in this way until $\gamma([t_{2p},b])$ does not hit any $K_{Q_i}$
or $t_{2p}=b$. As a result, we obtain priming division points $t_1,\dots,t_{2p}$.

{\bf Step 3:}
We prime $\gamma([t_{2j-1},t_{2j}])$ with the common colour of cubes $Q_i$ for
all $j=1,\dots,p$ and call
$\gamma$ a primed curve with priming division points $\{t_i\}_{i=1}^{2p}$.
$\qed$
\end{algorithm}

\begin{remark}\label{colourchoice}
  Algorithm~\ref{priming} will be applied using different priming colours.
  Note that an element of the collection $\{K_{Q_i}\}_{i=1}^l$ is chosen
  at most twice -- at most
once before $c$ and after $c$. Clearly, Algorithm~\ref{priming} 
may be applied (in a simpler manner) also in the
case where no point $c\in\mathopen]a,b\mathclose[$ is fixed.
\end{remark}

\begin{algorithm}\label{modification}{\bf (Primed curve modification).}
Assume that $\gamma\colon [a,b]\to\R^d$ is a primed curve with priming division
points $\{t_i\}_{i=1}^{2p}$. Modify $\gamma$ such that each part
$\gamma|_{[t_{2j-1},t_{2j}]}$, $j=1,\dots,p$, is replaced by
the line segment connecting $\gamma(t_{2j-1})$ to $\gamma(t_{2j})$ parameterised
by the interval $[t_{2j-1},t_{2j}]$. These line
segments inherit the prime colour from the corresponding parts of $\gamma$.
The modified curve is called $\tilde\gamma\colon [a,b]\to\R^d$.
$\qed$
\end{algorithm}

The above modification ``simplifies'' $\gamma$ by replacing certain parts of
it, corresponding to going in and out of coloured ``bad'' cubes, by line
segments.

\begin{remark}\label{stayinside}
If $\gamma([a,b])\subset K_Q$ for some $Q\in\Q_n$, then
$\tilde\gamma([a,b])\subset K_Q$.
\end{remark}

Next we introduce an algorithm that divides a curve
$\gamma\colon [a,b]\to K_Q$
into a collection of subcurves $\gamma\colon [a_j,a_{j+1}]\to\mathbb R^d$,
$j=0,\dots,p$ for some $p\in\N$, such that
\begin{equation}\label{sameproperties}
\begin{split}
  &\gamma|_{[a_j,a_{j+1}]}\text{ passes through an }(m,i)\text{-layer for some }
   i\in\{1,\dots,d\}\text{ and }\\  
  &\gamma([a_j,a_{j+1}])\subset K_{Q_j}\text{ for some }Q_j\in\Q_m\text{ with }
   \gamma([a_j,a_{j+1}])\cap Q_j\ne\emptyset.
\end{split}
\end{equation}
The points $\{a_j\}_{j=0}^{p+1}$ are called {\it layer division points}.
Algorithm~\ref{layerdivision} consists of 3 steps, the last of which is the most
complicated one. It is needed to guarantee that the  last  subcurve
$\gamma\colon [a_p,a_{p+1}]\to\mathbb R^d$ satisfies
\eqref{sameproperties}. This algorithm corresponds to zooming from grid level
$N^{-n}$ to grid level $N^{-m}$. It also defines a part of our approximating
broken line division points as intersection points of our curve with
the boundaries of certain $N^{-m}$-net cubes. In addition to these
division points, other $c$-points will be selected later.
These $c$-points will be responsible for the length gain of our broken line
approximation at this level.  

\begin{algorithm}\label{layerdivision}{\bf (Layer division).}
Let $n,m\in\N$ with  $m>n+\log 5/\log N$. 
Assume that $Q\in\Q_n$ and $\gamma\colon [a,b]\to K_Q$ is a curve which passes
through an $(n, i_1)$-layer for some $ i_1\in\{1,\dots,d\}$.

{\bf Step 1:}
Set $a_0:=a$ and define
\begin{align*}      
a_1:=\min\{t\in[a_0,b]\mid\, &\gamma|_{[a_0,t]}\text{ passes through an }
   (m,i)\text{-double-layer}\\
 &\text{for some }i\in\{1,\dots,d\}\}.
\end{align*}
Note that $a_1$ exists  since $m>n$ and $\gamma|_{[a,b]}$ passes through an
$(n, i_1)$-layer.  Clearly, $\gamma|_{[a_0,a_1]}$ satisfies
\eqref{sameproperties} since $\gamma([a_0,a_1])\subset K_{Q_m(\gamma(a_0))}$ (recall
the notation $Q_n(x)$ from the proof of Theorem~\ref{mainstandard}).

{\bf Step 2:}
Apply Step 1 to $\gamma|_{[a_1,b]}$ in order to define a point $a_2$.
Proceed by
applying Step 1 recursively and defining points $a_0,a_1,\dots,a_p$ until 
there is no $t\in[a_p,b]$ such that $\gamma|_{[a_p,t]}$ passes through an
$(m,i)$-double-layer for some $i\in\{1,\dots,d\}$. If $a_p=b$, the algorithm
terminates.

{\bf Step 3:}  The first standing assumption at this step is that there is
no $t\in[a_p,b]$ such that $\gamma|_{[a_p,t]}$ passes through an
$(m,i)$-double-layer for some $i\in\{1,\dots,d\}$. In particular,
$\gamma([a_p,b])\subset K_{Q_m(\gamma(a_p))}$. According to the second standing
assumption,  due to the construction, 
$\gamma([a_{p-1},a_p])\subset K_{Q_m(\gamma(a_{p-1}))}$.  
\begin{itemize}
\item[(1)] If $\gamma|_{[a_p,b]}$ passes through an $(m,i)$-layer,
  setting $a_{p+1}:=b$ terminates the algorithm. Obviously
  \eqref{sameproperties} is valid for $\gamma|_{[a_p,a_{p+1}]}$.
\item[(2)] If there is $Q'\in\Q_m$ such that
  $\gamma([a_{p-1},b])\subset K_{Q'}$ and $\gamma([a_{p-1},b])\cap Q'\ne\emptyset$,
  redefining $a_p:=b$ terminates the algorithm.  Let $\tilde a_p$ be the
  value of $a_p$ before we redefined it.  Note that $\gamma|_{[a_{p-1},b]}$
  passes through an $(m,i)$-layer, since $\gamma|_{[a_{p-1},\tilde a_p]}$
  passes through an $(m,i)$-double-layer and $\gamma|_{[\tilde a_p,b]}$ does
  not pass through any $(m,i)$-layer,  since the algorithm  did not
  terminate  at point (1).  
  So \eqref{sameproperties} is satisfied.
\item[(3)] If the algorithm has not terminated, we have 
  $\gamma(a_p)\in k_{Q_m(\gamma(b))}$, since $\gamma|_{[a_p,b]}$ does not pass
  through any $(m,i)$-layer. If $\gamma([a_p,b])\not\subset K_{Q_m(\gamma(b))}$,
  defining $a_{p+1}:=\sup\{t\in[a_p,b]\mid\gamma(t)\not\in K_{Q_m(\gamma(b))}\}$
  and $a_{p+2}:=b$ terminates the algorithm. The fact that
  $\gamma(a_p)\in k_{Q_m(\gamma(b))}$ implies that $\gamma|_{[a_p,a_{p+1}]}$ passes
  through an $(m,i)$-layer and $\gamma([a_p,a_{p+1}])\subset K_{Q_m(\gamma(a_p))}$
   by the first standing assumption.  Obviously,
  $\gamma|_{[a_{p+1},a_{p+2}]}$ satisfies \eqref{sameproperties}. 
\item[(4)] If $\gamma(a_{p-1})\in k_{Q_m(\gamma(b))}$, then
  \[
  \widetilde s:=\sup\{t\in [a_{p-1},b]\mid\gamma(t)\not\in K_{Q_m(\gamma(b))}\}
  \]
  exists  (since the algorithm  did not terminate  at point (2))  and is
  less than $a_p$  (since the algorithm  did not terminate  at point (3)).
   Redefining $a_p:=\widetilde s$ and setting $a_{p+1}:=b$ terminates the
  algorithm. Then $\gamma|_{[a_{p-1},a_p]}$ passes through an $(m,i)$-layer and
  $\gamma([a_{p-1},a_p])\subset K_{Q_m(\gamma(a_{p-1}))}$  (since $\widetilde s$
  is less than the original value of $a_p$).  Clearly
  $\gamma|_{[a_p,a_{p+1}]}$ satisfies \eqref{sameproperties}. 
\item[(5)] If the algorithm has not terminated, then
  $\gamma(a_{p-1})\not\in k_{Q_{m}(\gamma(b))}$ and there exists
  \[
  s:=\sup\{t\in [a_{p-1},a_p]\mid\gamma(t)\not\in k_{Q_m(\gamma(b))}\}.
  \]
\item[(6)] If $\gamma(s)\not\in\inter(k_{Q_m(\gamma(a_{p-1}))})$, redefining
  $a_p:=s$ and setting $a_{p+1}:=b$ terminates the algorithm.  Note that
  $\gamma([a_{p-1},a_p])\subset K_{Q_m(\gamma(a_{p-1}))}$ due to the second standing
  assumption, since the redefined value of $a_p$ is at most its original value.
  Thus \eqref{sameproperties} is valid for $\gamma|_{[a_{p-1},a_p]}$. Since the
  algorithm  did not terminate  at point (3), we have that
  $\gamma([a_p,a_{p+1}])\subset K_{Q_m(\gamma(a_{p+1}))}$ and, thus, also
  $\gamma|_{[a_p,a_{p+1}]}$ satisfies \eqref{sameproperties} due to the definition
  of $s$. 
\item[(7)] Recalling that $\gamma(s)\in\partial k_{Q_m(\gamma(b))}$, we may
  choose $Q_{\gamma(s)}\in\Q_m$ in such a way that $\gamma(s)\in Q_{\gamma(s)}$ and
  $Q_{\gamma(s)}\subset k_{Q_m(\gamma(b))}$.  In particular,
  $k_{Q_m(\gamma(b))}\subset K_{Q_{\gamma(s)}}$.  Then there is
  \[
  u:=\sup\{t\in [a_{p-1},b]\mid\gamma(t)\not\in K_{Q_{\gamma(s)}}\},
  \]
  since otherwise the algorithm would have terminated at  point (2). 
\item[(8)]  If $u<s$, redefine $a_p:=u$. Setting $a_{p+1}:=s$ and
  $a_{p+2}:=b$ terminates the algorithm. Now  the second standing assumption
  implies that  $\gamma([a_{p-1},a_p])\subset K_{Q_m(\gamma(a_{p-1}))}$. 
  Moreover, since the algorithm did not terminate at point (6), we have that
  $\gamma(s)\in\inter(k_{Q_m(\gamma(a_{p-1}))})$ which, in turn, implies that
  $\gamma(a_{p-1})\in k_{Q_{\gamma(s)}}$. Therefore,  $\gamma|_{[a_{p-1},a_p]}$
  passes through an $(m,i)$-layer.  Clearly
  $\gamma([a_p,a_{p+1}])\subset K_{Q_{\gamma(s)}}$, and \eqref{sameproperties} is
  satisfied since $\gamma(u)\in\partial K_{Q_{\gamma(s)}}$. Finally,
  $\gamma([a_{p+1},a_{p+2}])\subset K_{Q_m(\gamma(b))}$ due to the definition of $s$
  and $\gamma|_{[a_{p+1},a_{p+2}]}$ satisfies \eqref{sameproperties}  since
  $\gamma(s)\in\partial k_{Q_m(\gamma(b))}$. 
\item[(9)] If the algorithm has not terminated, we have $u\ge s$,
  implying $u>a_p$  since $k_{Q_m(\gamma(b))}\subset K_{Q_{\gamma(s)}}$ (see point
  (7)).  Define
  \[
  w:=\sup\{t\in [a_{p-1},s]\mid\gamma(t)\not\in K_{Q_{\gamma(s)}}\}.
  \]
\item[(10)] If $w$ exists, redefining $a_p:=w$ and setting $a_{p+1}:=s$,
  $a_{p+2}:=u$ and $a_{p+3}:=b$ terminates the algorithm.  Then
  $\gamma([a_{p-1},a_p])\subset K_{Q_m(\gamma(a_{p-1}))}$ by the second standing
  assumption and the fact that $s$, and thus $w$,  are  at most the original value
  of $a_p$. Since $\gamma(s)\in\inter(k_{Q_m(\gamma(a_{p-1}))})$ (the algorithm 
  did not terminate  at point (6)), we have that
  $\gamma(a_{p-1})\in k_{Q_{\gamma(s)}}$, implying that the first property of
  \eqref{sameproperties} is true for $\gamma|_{[a_{p-1},a_p]}$. Evidently,  
  $\gamma|_{[a_p,a_{p+1}]}$  satisfies  \eqref{sameproperties}. Since 
  the algorithm  did not terminate  at point (3), we  obtain  that
  $\gamma([\tilde a_p,b])\subset K_{Q_m(\gamma(b))}$, where $\tilde a_p$ is the
  original value of $a_p$. Therefore, 
  $\gamma([s,u]),\gamma([u,b])\subset K_{Q_m(\gamma(b))}$,  where also the
  definition of $s$ is used. The first property in \eqref{sameproperties} is
  valid for $\gamma|_{[a_{p+1},a_{p+2}]}$, since
  $\gamma(u)\in\partial K_{Q_{\gamma(s)}}$. It is also valid for
  $\gamma|_{[a_{p+2},a_{p+3}]}$ since $k_{Q_m(\gamma(b))}\subset K_{Q_{\gamma(s)}}$. 
\item[(11)] If $w$ does not exist, then
  \[
  \widetilde w:=\inf\{t\in [s,b]\mid\gamma(t)\not\in K_{Q_{\gamma(s)}}\}>a_p
  \]
  exists  since the algorithm  did not terminate  at point (2). The
  inequality $\widetilde w>a_p$ follows from the definition of $s$
  and the fact $k_{Q_m(\gamma(b))}\subset K_{Q_{\gamma(s)}}$.  Redefine
  $a_p:=\widetilde w$, set $a_{p+1}:=b$ and terminate the algorithm. Then
  $\gamma([a_{p-1},a_p])\subset K_{Q_{\gamma(s)}}$  since $w$ does not exist,
   and \eqref{sameproperties} is true since
  $\gamma(a_{p-1})\in k_{Q_{\gamma(s)}}$.
   As $k_{Q_m(\gamma(b))}\subset K_{Q_{\gamma(s)}}$,  the subcurve 
  $\gamma|_{[\widetilde w,b]}$ passes through an $(m,i)$-layer.  The definition
  of $s$ and the fact that the algorithm did not terminate at point (3) imply
  that  $\gamma([\widetilde w,b])\subset K_{Q_m(\gamma(b))}$ and, therefore,
  $\gamma|_{[a_p,a_{p+1}]}$ satisfies \eqref{sameproperties}. $\qed$
\end{itemize}
\end{algorithm}

\begin{remark}\label{caniterate}
If $\gamma(a)\in\partial Q'$ and $\gamma(b)\in\partial Q''$ for some
$Q',Q''\in\Q_m$ while applying Algorithm~\ref{layerdivision}, then, for all
$j=0,\dots,p+1$, there is $Q_j'\in\Q_m$ such that $\gamma(a_j)\in\partial Q_j'$.
 If  $p=0$, meaning  that
Algorithm~\ref{layerdivision} does not give any proper subcurves,  then
$\gamma([a,b])\subset K_{Q'}$ for some $Q'\in\Q_m$.  Since $\gamma$ passes
through an $(n, i_1)$-layer, this implies that $m\le n+\log 5/\log N$,  which
is a contradiction  due to the choice of $m$. 
\end{remark}

\begin{algorithm}\label{painting}{\bf (Painting curves).}
Let $\gamma\colon [a,b]\to\R^d$ be a primed curve with layer division points
$\{a_j\}_{j=0}^{p+1}$.  Suppose that a colour is given.  Let
$j\in\{0,\dots,p\}$. If a part of $\gamma(\mathopen]a_j,a_{j+1}\mathclose[)$ is
primed with  the given  colour, paint the corresponding closed set
$\gamma([a_j,a_{j+1}])$ with the  given  colour. $\qed$
\end{algorithm}

One of our main tools, Construction~\ref{brokenline} along with its 
special case, Construction~\ref{brokenline0}, will be applied
to curves contained in hereditarily good cubes, 
resulting in a painted modification of the curve with a collection of layer
division points. We will use white, blue and $i$-red paints for $i=1,2,\dots$.
Recall that white and blue parts will later enable us to iterate the
construction -- the difference between them being that only the white colour
indicates an increase of length. Red parts will be disregarded in length 
estimations of broken line approximations.

To illustrate the main ideas behind the construction, we begin with the
simplest case of $(0,m_1,k_0,\mfc,m_0)$-hereditarily good cubes. This
construction will also be utilised in the general case
discussed in Construction~\ref{brokenline}. In Construction~\ref{brokenline0}
only white and blue colours are used and the curve is not modified.
It corresponds to zooming from grid level $N^{-n_{0}}$ to
grid level $N^{-n_{0}-l_{0}}=N^{-n_{0}-L^{0}(m_{1},k_{0})_{0}}$ without the information
that the cubes are also $(q,m_1,k_0,\mfc,m_0)$-hereditarily good for some $q>0$,
which is available in Construction~\ref{brokenline}.  We remind that, 
by Definition~\ref{d:Delta} and Inequality~\eqref{*dj1},
$\Delta_j^0=L^0(m_1,k_0)_{j-1}-L^0(m_1,k_0)_j\ge m_1$. Therefore,  the 
lower bound imposed on $m_1$  in Construction~\ref{brokenline0} 
guarantees that Algorithm~\ref{layerdivision} may be applied. 

\begin{construction}\label{brokenline0}
Fix $\omega\in\Omega$ and $m_0,\mfc,k_0,m_1\in\N\setminus\{0\}$  with
$m_1>\log 5/\log N$.  Let $n_0\in\N$ and $Q\in\Q_{n_0}$.
Assume that every $Q'\in\Q_{n_0}(K_Q)$ is $(0,m_1,k_0,\mfc,m_0)$-hereditarily
good. Let $\gamma\colon [a,b]\to K_Q$ be a curve passing through an
$(n_0,j)$-layer for some $j\in\{1,\dots,d\}$. Fix
$c\in\mathopen]a,b\mathclose[$. We paint $\gamma$ and define layer division
points by applying the following steps.

{\bf Step1:}
By assumption, all cubes $Q'\in\Q_{n_0}(K_Q)$ are
$(k_0,\bL^0(m_1,k_0),\mfc,m_0)$-good. Apply Algorithm~\ref{colouring} to all of
them with $I=\{0\}$ and $\tilde k_0=k_0$. As a result, some cubes in
$\Q_{n_0+L^0(m_1,k_0)_{k_0-1}}(K_Q)$ are coloured in blue. Recall that
$\Delta_{k_0}^0=L^0(m_1,k_0)_{k_0-1}$.

{\bf Step 2:}
Let $Q_1,\dots,Q_l\in\Q_{n_0+L^0(m_1,k_0)_{k_0-1}}(K_Q)$ be the blue cubes obtained at
Step 1. Note that their blue-labelled parents belong to the set $\Q_{n_0}(K_Q)$.
Apply Algorithm~\ref{priming} with $n=n_0$ and $m=n_0+L^0(m_1,k_0)_{k_0-1}$ using
blue primer. Proceed by applying Algorithm~\ref{layerdivision} to $\gamma$
with $n=n_0 =n_0+L^0(m_1,k_0)_{k_0}$ and $m=n_0+L^0(m_1,k_0)_{k_0-1}$ and
denote the resulting layer division points by $\{a_{j_{k_0}}\}_{j_{k_0}=0}^{p+1}$.
Paint $\gamma$ with blue by means of Algorithm~\ref{painting} and, finally,
paint with white those parts $\gamma([a_{j_{k_0}},a_{j_{k_0}+1}])$ that are not
painted with blue.

{\bf Step 3:}
For all $j_{k_0}=0,\dots,p$, consider the curve
$\gamma\colon [a_{j_{k_0}},a_{j_{k_0}+1}]\to K_{\widetilde Q_{j_{k_0}}}$, where
$\widetilde Q_{j_{k_0}}\in\Q_{n_0+L^0(m_1,k_0)_{k_0-1}}$ (recall \eqref{sameproperties}).
\begin{itemize}
\item[$\bullet$] If $\gamma([a_{j_{k_0}},a_{j_{k_0}+1}])$ is painted blue, apply
  Algorithm~\ref{layerdivision} with $n=n_0+L^0(m_1,k_0)_{k_0-1}$ and
  $m=n_0+L^0(m_1,k_0)_{k_0-2}$ and denote the resulting layer division points by
  $\{a_{j_{k_0},j_{k_0-1}}\}_{j_{k_0-1}=0}^{p_{j_{k_0}}+1}$. Prime with blue all
  the sets $\gamma([a_{j_{k_0},j_{k_0-1}},a_{j_{k_0},j_{k_0-1}+1}])$. Go to Step 4.
\item[$\bullet$] If the curve $\gamma([a_{j_{k_0}},a_{j_{k_0}+1}])$ is white,
  then $\gamma(\mathopen]a_{j_{k_0}},a_{j_{k_0}+1}\mathclose[)\cap K_{Q_i}=\emptyset$
  for all $i=1,\dots,l$ by Algorithm~\ref{priming}. Further, since
  $\gamma([a_{j_{k_0}},a_{j_{k_0}+1}])\cap\widetilde Q_{j_{k_0}}\ne\emptyset$ by
  \eqref{sameproperties}, none of the cubes
  $Q'\in\Q_{n_0+L^0(m_1,k_0)_{k_0-1}}(K_{\widetilde Q_{j_{k_0}}})$ is blue, that is, they all
  are $(k_0-1,\bL^0(m_1,k_0),\mfc,m_0)$-good. Apply Algorithm~\ref{colouring} to
  all of them with $I=\{0\}$ and $\tilde k_0=k_0-1$.
  Let $Q'_1,\dots,Q'_{\tilde l}\in\Q_{n_0+L^0(m_1,k_0)_{k_0-2}}(K_{\widetilde Q_{j_{k_0}}})$ be
  the resulting blue cubes whose blue-labelled parents belong to the set
  $\Q_{n_0+L^0(m_1,k_0)_{k_0-1}}(K_{\widetilde Q_{j_{k_0}}})$.
  Apply Algorithm~\ref{priming} to $\gamma|_{[a_{j_{k_0}},a_{j_{k_0}+1}]}$ with
  $n=n_0+L^0(m_1,k_0)_{k_0-1}$ and $m=n_0+L^0(m_1,k_0)_{k_0-2}$ using blue 
  primer. Note that there is only one $j_{k_0}$ such that
  $c\in\mathopen]a_{j_{k_0}},a_{j_{k_0}+1}\mathclose[$. Proceed by applying
  Algorithm~\ref{layerdivision} with the same $n$ and $m$ and
  denote the resulting layer division points by
  $\{a_{j_{k_0},j_{k_0-1}}\}_{j_{k_0-1}=0}^{p_{j_{k_0}}+1}$. Go to Step 4.
\end{itemize}

{\bf Step 4:}
Using blue colour, paint $\gamma|_{[a_{j_{k_0}},a_{j_{k_0}+1}]}$ by means of
Algorithm~\ref{painting}. Finally, paint with white those sets
$\gamma([a_{j_{k_0},j_{k_0-1}},a_{j_{k_0},j_{k_0-1}+1}])$ that are not painted blue.

{\bf Step 5:}
Iterate Step 3 utilising curves determined by the layer division points
obtained in the previous iteration step, and selecting 
$n=n_0+L^0(m_1,k_0)_k$ and $m=n_0+L^0(m_1,k_0)_{k-1}$ for $k=k_0-2,\dots,1$.  
As a result, we obtain a curve with layer division points
$\{a_{j_{k_0},\dots,j_1}\}$, $j_k=0,\dots,p_{j_{k_0},\dots,j_{k+1}}+1$ for $k=k_0,\dots,1$,
such that  the sets $\gamma([a_{j_{k_0},\dots,j_k},a_{j_{k_0},\dots,j_k+1}])$ are painted
with either blue or white. $\qed$
\end{construction}

Now we are ready to present a general construction leading to a modification of
a curve $\gamma$ having layer division points that determine parts which are
painted white, blue or $i$-red for $i=1,\dots,q$. First we will define
inductively curves $\gamma_1,\dots,\gamma_q$, making use of
Algorithm~\ref{priming} with $i$-red primers, respectively, and
Algorithm~\ref{modification}. In particular, $\gamma_1,\dots,\gamma_{q-1}$ are
auxiliary curves that will be utilised when defining $\gamma_q$. Next we
apply Construction~\ref{brokenline0} to $\gamma_q$ in order to identify some
layer division points and painted curve segments. The final outcome is obtained
as a result of an iteration process. In this construction, we are zooming again
from grid level $N^{-n_{0}}$ to grid level $N^{-n_{0}-l_{0}}$. However, later (see
Remark~\ref{caniterate2}.(b)) we will zoom in to grid level
$N^{-n_{0}-(q+1)l_{0}} =N^{-n_{0}-L^q(m_1,k_0)_0}$
and, therefore, we need to take into consideration $i$-red cubes coming up from
deeper zoom levels of our construction. In Example~\ref{brokenline-example},
we illustrate Construction~\ref{brokenline} in the special case depicted in
Figure~\ref{fig:levels}.  In Construction~\ref{brokenline}, we are taking
advantage of the property that $L^i(m_1,k_0)_j$ for $j\ge k_0$ are defined using
the levels determined by $\bL^{i-1}(m_1,k_0)$, see Definition~\ref{j-levels} and
Figure~\ref{fig:levels}. 

\begin{construction}\label{brokenline}
Fix $\omega\in\Omega$ and $m_0,\mfc,k_0,m_1\in\N\setminus\{0\}$  with
$m_1>\log 5/\log N$.  Let
$n_0,q\in\N$ and $Q\in\Q_{n_0}$. Assume that every $Q'\in\Q_{n_0}(K_Q)$ is
$(q,m_1,k_0,\mfc,m_0)$-hereditarily good. Let $\gamma\colon [a,b]\to K_Q$ be a
curve passing through an $(n_0,j)$-layer for some $j\in\{1,\dots,d\}$. Fix
$c\in\mathopen]a,b\mathclose[$. We define a painted modification of $\gamma$
with layer division points by applying the following steps. During the
first three steps, we identify some ``bad parts'' of $\gamma$ and modify it at
these bad zones. Set $\hat q:=q$, $I:=\{0,\dots,\hat q\}$ and let
$\hat k_i:=k_i$ for all $i\in I$. 

{\bf Step 1:}
By assumption, all cubes $Q'\in\Q_{n_0}(K_Q)$ are
$(\hat k_i,\bL^i(m_1,k_0),\mfc,m_0)$-good for all $i\in I$.
Apply Algorithm~\ref{colouring} to all cubes $Q'\in\Q_{n_0}(K_Q)$ to colour some
of their subcubes in blue  corresponding  to  the case $i=0$  or $i$-red for
$i\in I\setminus\{0\}$.

{\bf Step 2:}
Set $n=n_0$ and $m=n_0+L^1(m_1,k_0)_{\hat k_1 -1}$ and let
$Q_1,\dots,Q_l\in\Q_m(K_Q)$ be
the 1-red cubes. (Recall that their labelled parents belong to the set
$\Q_{n_0}(K_Q)$.) Apply Algorithm~\ref{priming} to $\gamma$ using 1-red
primer. Next apply Algorithm~\ref{modification} and denote the modified
curve by $\gamma_1$.

{\bf Step 3:}
Set $n=n_0$ and $m=n_0+L^2(m_1,k_0)_{\hat k_2 -1}$ and let
$Q'_1,\dots,Q'_{\hat l}\in\Q_m(K_Q)$ be the 2-red
cubes. Apply Algorithm~\ref{priming} to $\gamma_1$ using 2-red
primer. Next apply Algorithm~\ref{modification} to $\gamma_1$ and denote the
modified curve by $\gamma_2$. When $n=n_0$ and
$m=n_0+L^i(m_1,k_0)_{\hat k_i -1}$, with $i=3,\dots,\hat q$, continue
iteratively until the curve $\gamma_{\hat q}$ is defined. Note that parts
of $\gamma_{\hat q}([a,b])$ are primed with $i$-red primer for $i\in I$,
and some parts may be primed with several $i$-red primers. 

Next we start to introduce the level division points used in our broken line
approximation at different levels.

{\bf Step 4:}
Apply Steps 2--5 of Construction~\ref{brokenline0} to the curve
$\gamma_{\hat q}$ until the curve
$\gamma_{\hat q}|_{[a_{j_{k_0},\dots,j_{k+1}},a_{j_{k_0},\dots,j_{k+1}+1}]}$ is considered, where
\[
m=n_0+L^0(m_1,k_0)_{k-1}=n_0+L^1(m_1,k_0)_{\hat k_1 -1}=\dots
 =n_0+L^{q'}(m_1,k_0)_{\hat k_{q'} -1}
\]
for some $q'\in\{1,\dots,q\}$, that is, until the size of blue cubes is same as
the size of $i$-red cubes for $i=1,\dots,q'$. Now apply Steps 2 and 3
of Construction~\ref{brokenline0}. Instead of applying Step 4 of
Construction~\ref{brokenline0}, proceed by painting as follows: Apply $q'$
times Algorithm~\ref{painting} to
$\gamma_{\hat q}|_{[a_{j_{k_0},\dots,j_{k+1}},a_{j_{k_0},\dots,j_{k+1}+1}]}$ using
$i$-red paint for $i=1,\dots,q'$.
\begin{itemize}
\item[$\bullet$] If
  $\gamma_{\hat q}|_{[a_{j_{k_0},\dots,j_{k+1}},a_{j_{k_0},\dots,j_{k+1}+1}]}$ is blue, those
  sets $\gamma_{\hat q}([a_{j_{k_0},\dots,j_k},a_{j_{k_0},\dots,j_k+1}])$ which are
  not painted $i$-red for any $i\in\{1,\dots,q'\}$ inherit
  the blue paint. Go to Step 5.
\item[$\bullet$] If 
  $\gamma_{\hat q}|_{[a_{j_{k_0},\dots,j_{k+1}},a_{j_{k_0},\dots,j_{k+1}+1}]}$ is white, apply
  Algorithm~\ref{painting} to the curve
  $\gamma_{\hat q}|_{[a_{j_{k_0},\dots,j_{k+1}},a_{j_{k_0},\dots,j_{k+1}+1}]}$ using blue
  paint, ignoring those curve segments
  $\gamma_{\hat q}([a_{j_{k_0},\dots,j_k},a_{j_{k_0},\dots,j_k+1}])$ which are
  painted $i$-red for some $i\in\{1,\dots,q'\}$, that is, if
  $\gamma_{\hat q}([a_{j_{k_0},\dots,j_k},a_{j_{k_0},\dots,j_k+1}])$ is painted
  $i$-red, do not paint it blue even though a part of it is primed with
  a blue primer. Finally, paint white those sets
  $\gamma_{\hat q}([a_{j_{k_0},\dots,j_k},a_{j_{k_0},\dots,j_k+1}])$ that are not
  painted $i$-red or blue. Go to Step 5.
\end{itemize}

{\bf Step 5:}
For all $j_{k_0},\dots,j_k$, consider 
$\gamma_{\hat q}\colon [a_{j_{k_0},\dots,j_k},a_{j_{k_0},\dots,j_k+1}]\to
K_{Q_{j_{k_0},\dots,j_k}}$.
\begin{itemize}
\item[$\bullet$] If 
  $\gamma_{\hat q}([a_{j_{k_0},\dots,j_k},a_{j_{k_0},\dots,j_k+1}])$ is
  $i$-red for some $i\in\{1,\dots,q'\}$, the construction terminates.
\item[$\bullet$] In the case that
  $\gamma_{\hat q}([a_{j_{k_0},\dots,j_k},a_{j_{k_0},\dots,j_k+1}])$ is blue, all
  the cubes in the collection 
  $\Q_{n_0+L^{0}(m_1,k_0)_{k-1}}(K_{Q_{j_{k_0},\dots,j_k}})$ are
  $(\hat k_i -1,\bL^{i}(m_1,k_0),\mfc,m_0)$-good for all $i=1,\dots,q'$,
  since otherwise
  $\gamma_{ \hat q }([a_{j_{k_0},\dots,j_k},a_{j_{k_0},\dots,j_k+1}])$ had
  received $i$-red paint for some $i\in\{1,\dots,q'\}$ (recall the argument
  from the second bullet of Step 3 in Construction~\ref{brokenline0}).  
  Repeat the construction from Step 1 utilising the curve
  $\gamma_{ \hat q }|_{[a_{j_{k_0},\dots,j_k},a_{j_{k_0},\dots,j_k+1}]}$ with
  $I=\{1,\dots,q'\}$, replacing $\hat k_i$ by $\hat k_i -1$ for
  $i=1,\dots,q'$ and keeping the value of $ \hat k_i $ for $i=q'+1,\dots,q$.
\item[$\bullet$] If 
  $\gamma_{\hat q}([a_{j_{k_0},\dots,j_k},a_{j_{k_0},\dots,j_k+1}])$ is white,
  all cubes in $\Q_{n_0+L^0(m_1,k_0)_{k-1}}(K_{Q_{j_{k_0},\dots,j_k}})$
  are $(k-1,\bL^0(m_1,k_0),\mfc,m_0)$-good and
  $(\hat k_i-1,\bL^i(m_1,k_0),\mfc,m_0)$-good for all $i=1,\dots,q'$. Repeat
  the construction from Step 1 using the curve
  $\gamma_{\hat q}|_{[a_{j_{k_0},\dots,j_k},a_{j_{k_0},\dots,j_k+1}]}$
  with $I=\{0,\dots,q'\}$, replacing $\hat k_i$ by $\hat k_i -1$ for
  $i=0,\dots,q'$ and letting $\hat k_i$ be as they are for $i=q'+1,\dots,q$.
\end{itemize}
While iterating the Steps 1--5, the curve $\gamma_{\hat q}$ is further modified.
For $i=q'+1,\dots,q$, the $i$-red parts are taken into account once
their levels are reached in the construction at Step 4. 

{\bf Step 6:}
The construction is complete once the level $n_0+L^0(m_1,k_0)_0$ is
reached, the modified painted curve $\gamma_{\tilde q}$ for some
$\tilde q \ge q$ and the layer division points $a_{j_{k_0},\dots,j_1}$ are
defined and the final curve segments
$\gamma_{\tilde q}([a_{j_{k_0},\dots,j_1},a_{j_{k_0},\dots,j_1+1}])$ are painted. $\qed$
\end{construction}

\begin{example}\label{brokenline-example}
We demonstrate Construction~\ref{brokenline} in the special case $k_0=8$ and
$q=2$ as shown in Figure~\ref{fig:levels}. During the Steps 1--3 a curve
$\gamma_2$, containing parts with 1-red and 2-red primer, is defined. Recall
from Example~\ref{j-levelsexample} that $k_1=11$, $k_2=13$ and
\begin{equation}\label{firstcoincidence}
L^0(m_1,8)_6=L^1(m_1,8)_{10}=L^2(m_1,8)_{12}=15\cdot m_1.
\end{equation}
At Step 4 we apply Step 2 of Construction~\ref{brokenline0} once with
$m=n_0+L^0(m_1,8)_7=n_0+8\cdot m_1$ and define curves
$\gamma_2|_{[a_{j_8},a_{j_8+1}]}$. While considering these curves, we have already
reached the level, where the size of blue cubes  
corresponding to level $n_0+L^0(m_1,8)_6=n_{0}+15\cdot m_{1}$   
is the same as the size of 1-red and 2-red cubes by
\eqref{firstcoincidence}, so $q'=2$. We continue with Step 4 and define curves
$\gamma_2|_{[a_{j_8,j_7},a_{j_8,j_7+1}]}$. For blue and white curves we proceed from
Step 5 and go back to Step 1, where we find new 1-red cubes at level
$n_0+26\cdot m_1$ and 2-red cubes at level $n_0+36\cdot m_1$. After Step 3 we
have defined the modified curve $\gamma_4$. At Step 4 we have again only one
``blue'' step and after that we have reached the level where
$m=n_0+L^0(m_1,8)_4=n_0+L^1(m_1,8)_9=n_0+26\cdot m_1$.
So in this case $q'=1$. After Step 4 we have constructed the curves
$\gamma_4|_{[a_{j_8,j_7,j_6,j_5},a_{j_8,j_7,j_6,j_5+1}]}$. For blue and white curves we
continue from Step 5 and go back to Step 1, where we produce new 1-red cubes at
level $n_0+36\cdot m_1$. After Step 2 we have defined the modified curve
$\gamma_5$ and go to Step 4. (There is no Step 3 since $q'=1$.) Next we have
three ``blue'' iteration steps at levels $n_0+30\cdot m_1$, $n_0+33\cdot m_1$
and $n_0+35\cdot m_1$, and we end
up with curves $\gamma_5|_{[a_{j_8,\dots,j_2},a_{j_8,\dots,j_2+1}]}$. Now we have reached
the final level $n_0+36\cdot m_1$, where the sizes of blue, 1-red and 2-red
cubes are the same. We complete Step 4 in order to define the final curves
$\gamma_5|_{[a_{j_8,\dots,j_1},a_{j_8,\dots,j_1+1}]}$.
\end{example}  
 
\begin{remark}\label{caniterate2}
(a) Denote by $b_j$, $j=0,\dots,p$, the layer division points
$a_{j_{k_0},\dots,j_1}$ obtained in Construction~\ref{brokenline}. Then the family
of curves $\gamma_{\tilde q}\colon [b_j,b_{j+1}]\to K_{Q_j}$, $j=0,\dots,p$,
satisfies \eqref{sameproperties} with $m=n_0+l_0$.
If $\gamma_{\tilde q}([b_j,b_{j+1}])$ is not red  and $q\le i_0+2$ (recall
\eqref{defofi0}),  all the cubes $Q'\in\Q_{n_0+l_0}(K_{Q_j})$ are
$(q-1,m_1,k_0,\mfc,m_0)$-hereditarily good  due to the iterative
construction of sequences $\bL^i(m_1,k_0)$ (see Definition~\ref{j-levels}), 
$\gamma_{\tilde q}(b_j)=\gamma(b_j)$ and $\gamma_{\tilde q}(b_{j+1})=\gamma(b_{j+1})$.
On the other hand, if $\gamma_{\tilde q}([b_j,b_{j+1}])$ is red, it may happen that
$\gamma_{\tilde q}(b_j)\ne\gamma(b_j)$ or
$\gamma_{\tilde q}(b_{j+1})\ne\gamma(b_{j+1})$. In this case,
$\gamma_{\tilde q}(b_j)$, respectively $\gamma_{\tilde q}(b_{j+1})$, is on a line
segment produced by Algorithm~\ref{modification} and, therefore, 
$\gamma_{\tilde q}([b_{j-1},b_j])$, respectively $\gamma_{\tilde q}([b_{j+1},b_{j+2}])$,
is red.

(b) If $q>i_0+2$, some parts of $\gamma_{\tilde q}$ may be primed with $i$-red for
$i>i_0 +2$. These parts are not painted with $i$-red in
Construction~\ref{brokenline}, since the corresponding $i$-red cubes are at
higher levels than $n_0+L^0(m_1,k_0)_0$, where the construction terminates. For
later purposes (see Proposition~\ref{notalphaholder}), we emphasise
that,  due to the iterative construction of sequences $\bL^i(m_1,k_0)$, 
Construction~\ref{brokenline} may be continued until the level
$n_0+L^q(m_1,k_0)_0$ is reached. In this case, there will be no primed parts
that are not painted.
 
(c) Construction~\ref{brokenline} may be applied also in the case when no
$c\in\mathopen]a,b\mathclose[$ is fixed.

(d) Let $j_{k_0}=i_c$ be an index such that $c\in [a_{i_c},a_{i_c+1}]$. If
$\gamma_{\tilde q}([a_{i_c},a_{i_c+1}])$ does not contain parts which are painted or
primed with red, we have 
\[
|\gamma_{\tilde q}(a_{i_c})-\gamma(c)|=|\gamma(a_{i_c})-\gamma(c)|
\le 5\sqrt d N^{-n_0-L^0(m_1,k_0)_{k_0-1}}.
\]
Otherwise, it may happen that $\gamma_{\tilde q}(c)\ne\gamma(c)$. By construction,
$\gamma_q([a_{i_c},a_{i_c+1}])\subset K_{Q'}$ for some
$Q'\in\Q_{n_0+L^0(m_1,k_0)_{k_0-1}}$, where $\gamma_q$ is the modified curve obtained
after Step 3 in Construction~\ref{brokenline}. By Algorithms~\ref{priming} and
\ref{modification} and Definition~\ref{j-levels}, we deduce that
\[
|\gamma(c)-\gamma_q(c)|\le 6\sqrt dN^{-n_0}\sum_{i=1}^qN^{-L^i(m_1,k_0)_{k_i-1}}
\le 6\sqrt dN^{-n_0-L^0(m_1,k_0)_{k_0-1}}.
\]
Furthermore, according to Remark~\ref{stayinside},
$\gamma_{\tilde q}([a_{i_c},a_{i_c+1}])\subset K_{Q'}$, where $\gamma_{\tilde q}$ is
the final modified curve in Construction~\ref{brokenline}. Therefore,
\[
|\gamma_{\tilde q}(a_{i_c})-\gamma(c)|\le 11\sqrt d N^{-n_0-L^0(m_1,k_0)_{k_0-1}}.
\]
In both cases, recalling that $L^0(m_1,k_0)_{k_0-1}=m_1k_0$  (see
Definition~\ref{j-levels})  and
$|\gamma(a)-\gamma(c)|+|\gamma(c)-\gamma(b)|\ge N^{-n_0}$, we conclude
that
\[
\sum_{j_{k_0}=0}^p|\gamma_{\tilde q}(a_{j_{k_0}})-\gamma_{\tilde q}(a_{j_{k_0}+1})|
 \ge (1-CN^{-m_1k_0})\bigl(|\gamma(a)-\gamma(c)|+|\gamma(c)-\gamma(b)|\bigr),
\]
where $C:=22\sqrt d$.

(e) By definition, the modified curve $\gamma_{\tilde q}$ enters every expanded
red cube $K_{Q'}$ at most twice -- at most once before $c$ and at most once
after $c$. 
\end{remark}

\medskip

\section{Appendix C: Increase of length}\label{lengthincrease}

\medskip

In this section, we estimate the length of white and non-red parts of
$\gamma_{\tilde q}$ obtained in Construction~\ref{brokenline}. We begin by
verifying a lemma concerning diameters of blue and red cubes inside white or
blue cubes. Recall from Definition~\ref{j-levels} and
Equation~\eqref{defofi0} that $l_0=m_1(1+2+\dots+k_0)$ and $m_1k_{i_0}=l_0$.

\begin{lemma}\label{colourcontribution}
Fix $\omega\in\Omega$ and $m_0,\mfc,k_0,m_1\in\N\setminus\{0\}$ with $k_0\ge 2$.
Let $n,q\in\N$ and $Q\in\Q_n$. Assume that $Q$ is
$(q,m_1,k_0,\mfc,m_0)$-hereditarily good.
Apply Algorithm~\ref{colouring} iteratively as in Construction~\ref{brokenline}
until the level $n+L^q(m_1,k_0)_0=n+(q+1)l_0$ is reached (recall
Remark~\ref{caniterate2}.(b)). Set $m:=n+\tilde ql_0+L^0(m_1,k_0)_{j_0}$, where
$\tilde q\in\{0,\dots,q\}$ and $j_0\in\{1,\dots,k_0\}$. Assume that
$\widetilde Q\in\Q_m(Q)$ is white. Then
\[
\sum_{Q',\,\lpb(Q')=\widetilde Q}\diam Q'\le\mfc N^{-m_1j_0}\diam\widetilde Q.
\]
Suppose that, for some $q'\le q$, there exists $j_i\in\{k_{i-1}+1,\dots,k_i\}$
for $i=1,\dots,q'$ such that
$m=n+\tilde ql_0+L^i(m_1,k_0)_{j_i}$. If $\widehat Q\in\Q_m(Q)$ is white or
blue, then
\[
  \sum_{Q',\,\lp1(Q')=\widehat Q}\diam Q'\le\mfc N^{-m_1j_1}\diam\widehat Q
\]  
and
\[
\sum_{i=1}^{q'}\sum_{Q',\,\lpi(Q')=\widehat Q}\diam Q'
  \le 6\mfc N^{-m_1j_1}\diam\widehat Q.
\]
Further, if $q'\ge 2$ and $k_0\ge 3$, then
\[
\sum_{i=2}^{q'}\sum_{Q',\,\lpi(Q')=\widehat Q}\diam Q'
  \le 6\mfc N^{-m_1j_2}\diam\widehat Q.
\]
\end{lemma}

\begin{proof}
Since $\widetilde Q\in\Q_m(Q)$ is white, it is
$(j_0,\bL^0(m_1,k_0),\mfc,m_0)$-good and, therefore, the blue-labelled parent
of at most $\mfc$ blue cubes at level $m+\Delta_{j_0}^0=m+m_1j_0$. This implies
the first claim. The second claim follows similarly.

Since $\widehat Q\in\Q_m(Q)$ is white or blue, it is
$(j_i,\bL^i(m_1,k_0),\mfc,m_0)$-good for $i=1,\dots,q'$, where $q'\le q$ and
$j_1<j_2<\dots<j_{i_0}\le j_{i_0+1}\le\dots\le j_{q'}$. Therefore, the number of
$i$-red cubes at level $m+\Delta_{j_i}^i$, having $i$-red-labelled parent
$\widehat Q$, is at most $\mfc$ for all $i=1,\dots,q'$. By
 Inequality~\eqref{*dj1},  $\Delta_{j_i}^i\ge m_1j_i$ for all
$i=1,\dots,i_0$. According to  Inequality~\eqref{ybound} in 
Lemma~\ref{samesizenumber}, 
the number of $i$'s with $i>i_0$ and $\Delta_{j_i}^i=2^kl_0$ is at most
$3\cdot 2^k$. Combining the above facts, we
conclude that
\begin{align*}
\sum_{i=1}^{q'}\sum_{\substack{Q'\\ \lpi(Q')=\widehat Q}}\diam Q'&\le\mfc\diam\widehat Q
  \Bigl(\sum_{j=j_1}^{i_0} N^{-m_1j}+\sum_{k=0}^\infty 3\cdot 2^k N^{-2^kl_0}\Bigr)\\
&\le\mfc\diam\widehat Q(\tfrac 1{1-N^{-m_1}}N^{-m_1j_1}+4N^{-l_0})\le 6\mfc
  N^{-m_1j_1}\diam\widehat Q,
\end{align*}
where we used the fact that, by Remark~\ref{k0versusk1}.(a), for $k_0\ge 2$,
\[
j_1\le k_1\le\frac 32 k_0\le\frac 12k_0+\frac 12k_0^2=\frac{l_0}{m_1}.
\]
The last claim follows in a similar manner for $k_0\ge 3$.
\end{proof}

\begin{remark}\label{parameters}
We have quite a few parameters in our construction and, as mentioned in
Remark~\ref{hereditaryremark}.(b), the order in which they are selected is very
delicate. It  is  done in the second paragraph of the proof of
Theorem~\ref{maingeneral}, but in this Section we impose some restrictions on
them. The parameter $m_1$ controls the contribution of blue curve segments (see
\eqref{*nredsum} below). We can make that contribution small, but we have
to fix $m_1$ before we can choose the parameter $m_0$, which is used to
tune the length gain in our broken line approximation (see
Lemma~\ref{lengthgain} and Propositions~\ref{hdprob} and \ref{Salphaissmall}).
The role of parameter $k_0$ is to make the contribution of red curve segments
arbitrarily small (see \eqref{*redsum}).
\end{remark}
 
The following proposition is the key result of this section. It is
essential in the proof of Lemma~\ref{smallgain} (see \eqref{*critineq}) which,
in turn, is the basis of Proposition~\ref{notalphaholder},  whose proof is
the main goal of this section. A curve is called
red, if it is $i$-red for some $i\in\N\setminus\{0\}$.

\begin{proposition}\label{paintedinhereditarily}
Fix $\omega\in\Omega$ and $m_0,\mfc,k_0\in\N\setminus\{0\}$ with $k_0\ge 3$. Let
$m_1,n,q\in\N$  with $m_1>\log 5/\log N$  and
$Q\in\Q_n$. Assume that every $Q'\in\Q_n(K_Q)$
is $(q,m_1,k_0,\mfc,m_0)$-hereditarily good. Let $\gamma\colon[a,b]\to K_Q$ be a
curve passing through an $(n,i)$-layer for some $i\in\{1,\dots,d\}$
and assume that
$\gamma(x)\in\partial Q_x$ for some $Q_x\in\Q_n$, where $x\in\{a,b\}$. Fix
$c\in\mathopen]a,b\mathclose[$. Applying Construction~\ref{brokenline} to
$\gamma$ and denoting the modified painted curve by $\gamma_{\tilde q}$, there
exist positive constants $C_1$ and $M_1$ depending
only on $d$ and $\mfc$ such that, for all $m_1\ge M_1$, 
\begin{align}\nonumber
\sum_{\substack{j_{k_0},\dots,j_1\\\gamma_{\tilde q}([a_{j_{k_0},\dots,j_1},a_{j_{k_0},\dots,j_1+1}])
      \text{ is white}}}
  &|\gamma_{\tilde q}(a_{j_{k_0},\dots,j_1})-\gamma_{\tilde q}(a_{j_{k_0},\dots,j_1+1})|\\ \label{*nredsum}
  &\ge(1-C_1N^{-m_1})\bigl(|\gamma(a)-\gamma(c)|+|\gamma(c)-\gamma(b)|\bigr)
\end{align}
and
\begin{align}\nonumber
\sum_{\substack{j_{k_0},\dots,j_1\\\gamma_{\tilde q}([a_{j_{k_0},\dots,j_1},a_{j_{k_0},\dots,j_1+1}])
      \text{ is not red}}}
  &|\gamma_{\tilde q}(a_{j_{k_0},\dots,j_1})-\gamma_{\tilde q}(a_{j_{k_0},\dots,j_1+1})|\\ \label{*redsum}
  &\ge(1-C_1N^{-m_1k_0})\bigl(|\gamma(a)-\gamma(c)|+|\gamma(c)-\gamma(b)|\bigr).
\end{align}
\end{proposition}

\begin{proof}
Let $l\in\{k_0,\dots,1\}$. Here we use the convention
$[a_{j_{k_0+1}},a_{j_{k_0+1}+1}]:=[a,b]$. Consider
$\gamma_{\tilde q}\colon [a_{j_{k_0},\dots,j_{l+1}},a_{j_{k_0},\dots,j_{l+1}+1}]\to\R^d$. 
By Remarks~\ref{caniterate} and \ref{stayinside} and
Condition~\eqref{sameproperties}, there exist cubes
$Q_{j_{k_0},\dots,j_{l+1}},Q_{j_{k_0},\dots,j_{l+1}}^\eta\in Q_{n+L^0(m_1,k_0)_l}$,
$\eta\in\{\alpha,\beta\}$, such that
\begin{align}
&\gamma_{\tilde q}([a_{j_{k_0},\dots,j_{l+1}},a_{j_{k_0},\dots,j_{l+1}+1}])\subset
  K_{Q_{j_{k_0},\dots,j_{l+1}}},\label{indexfixcubes1}\\
&\gamma_{\tilde q}(a_{j_{k_0},\dots,j_{l+1}})\in\partial Q_{j_{k_0},\dots,j_{l+1}}^\alpha
  \text{ and }
  \gamma_{\tilde q}(a_{j_{k_0},\dots,j_{l+1}+1})\in\partial Q_{j_{k_0},\dots,j_{l+1}}^\beta
  \text{ and}\label{indexfixcubes1.5}\\
&|\gamma_{\tilde q}(a_{j_{k_0},\dots,j_{l+1}})-\gamma_{\tilde q}(a_{j_{k_0},\dots,j_{l+1}+1})|
  \ge N^{-n-L^0(m_1,k_0)_l} 
  =\frac 1{\sqrt d}\diam Q_{j_{k_0},\dots,j_{l+1}} \label{indexfixcubes2}.
\end{align}

If $\gamma_{\tilde q}([a_{j_{k_0},\dots,j_{l+1}},a_{j_{k_0},\dots,j_{l+1}+1}])$ is red, the
interval $[a_{j_{k_0},\dots,j_{l+1}},a_{j_{k_0},\dots,j_{l+1}+1}]$ does not contain any
further layer division points and, in particular, it does not contain any
white
or blue curve segments. If
$\gamma_{\tilde q}([a_{j_{k_0},\dots,j_{l+1}},a_{j_{k_0},\dots,j_{l+1}+1}])$ is blue, it does
not contain any white curve segments of the form 
$\gamma_{\tilde q}([a_{j_{k_0},\dots,j_l},a_{j_{k_0},\dots,j_l+1}])$ but may contain
blue and red ones. Finally, if
$\gamma_{\tilde q}([a_{j_{k_0},\dots,j_{l+1}},a_{j_{k_0},\dots,j_{l+1}+1}])$ is white, it
may contain white, blue and red curve segments 
$\gamma_{\tilde q}([a_{j_{k_0},\dots,j_l},a_{j_{k_0},\dots,j_l+1}])$. 

Suppose that $\gamma_{\tilde q}([a_{j_{k_0},\dots,j_{l+1}},a_{j_{k_0},\dots,j_{l+1}+1}])$ is
white. We will first estimate the contribution of blue curve segments
$\gamma_{\tilde q}([a_{j_{k_0},\dots,j_l},a_{j_{k_0},\dots,j_l+1}])$ it contains.
If the segment $\gamma_{\tilde q}([a_{j_{k_0},\dots,j_l},a_{j_{k_0},\dots,j_l+1}])$
is blue,
it intersects $K_{Q'}$ for some blue cube
$Q'\in\Q_{n+L^0(m_1,k_0)_{l-1}}(K_{Q_{j_{k_0},\dots,j_{l+1}}})$. Applying
\eqref{indexfixcubes1} with $l$ replaced by $l-1$, we conclude that
$\gamma_{\tilde q}([a_{j_{k_0},\dots,j_l},a_{j_{k_0},\dots,j_l+1}])\subset 15Q'$. Since
$\gamma_{\tilde q}([a_{j_{k_0},\dots,j_{l+1}},a_{j_{k_0},\dots,j_{l+1}+1}])$ is white, all
cubes in the set $\Q_{n+L^0(m_1,k_0)_l}(K_{Q_{j_{k_0},\dots,j_{l+1}}})$ are white
according to Step 3
of Construction~\ref{brokenline0}. Combining Inequality~\eqref{indexfixcubes2}
and the first claim in Lemma~\ref{colourcontribution} gives
\begin{equation}\label{notblue}
\begin{split}
&\sum_{\substack{j_l\\\gamma_{\tilde q}([a_{j_{k_0},\dots,j_l},a_{j_{k_0},\dots,j_l+1}])\text{ is not blue}}}
 |\gamma_{\tilde q}(a_{j_{k_0},\dots,j_l})-\gamma_{\tilde q}(a_{j_{k_0},\dots,j_l+1})|\\
&\quad\ge|\gamma_{\tilde q}(a_{j_{k_0},\dots,j_{l+1}})
 -\gamma_{\tilde q}(a_{j_{k_0},\dots,j_{l+1}+1})|\\
&\qquad -\sum_{\widetilde Q\in\Q_{n+L^0(m_1,k_0)_l}(K_{Q_{j_{k_0},\dots,j_{l+1}}})}
 \sum_{\substack{Q'\\ \lpb(Q')=\widetilde Q}}15\diam Q'\\
&\quad\ge(1-5^d\mfc 15\sqrt dN^{-m_1l})
  |\gamma_{\tilde q}(a_{j_{k_0},\dots,j_{l+1}})-\gamma_{\tilde q}(a_{j_{k_0},\dots,j_{l+1}+1})|.
\end{split}
\end{equation}

Next we estimate the contribution of red parts which are contained in the curve
$\gamma_{\tilde q}([a_{j_{k_0},\dots,j_{l+1}},a_{j_{k_0},\dots,j_{l+1}+1}])$. If
$\gamma_{\tilde q}([a_{j_{k_0},\dots,j_l},a_{j_{k_0},\dots,j_l+1}])$ is red, it
intersects $K_{Q'}$ for some $i$-red cube
$Q'\in\Q_{n+L^0(m_1,k_0)_{l-1}}(K_{Q_{j_{k_0},\dots,j_{l+1}}})$ and, as above, we obtain that
$\gamma_{\tilde q}([a_{j_{k_0},\dots,j_l},a_{j_{k_0},\dots,j_l+1}])\subset 15Q'$. Now
$\lpi(Q')\in\Q_{n+L^0(m_1,k_0)_{r(i,l)}}$ for $r(i,l)\ge l+1$.
Suppose first that
$c\not\in [a_{j_{k_0},\dots,j_{r(i,l)+1}},a_{j_{k_0},\dots,j_{r(i,l)+1}+1}]$. Let
\begin{align*}
t_1&:=\min\{t\in [a_{j_{k_0},\dots,j_{l+1}},a_{j_{k_0},\dots,j_{l+1}+1}]\mid
\gamma_{\tilde q}(t)\in K_{Q'}\}\text{ and }\\
t_2&:=\max\{t\in [a_{j_{k_0},\dots,j_{l+1}},a_{j_{k_0},\dots,j_{l+1}+1}]\mid
\gamma_{\tilde q}(t)\in K_{Q'}\}.
\end{align*}
By Algorithm~\ref{modification}, the curve $\gamma_{\tilde q}([t_1,t_2])$ is a
straight line inside $K_{Q'}$. Assuming that
$K_{Q'}\subset\widetilde Q\in\Q_{n+L^0(m_1,k_0)_l}$, a repeated application of
condition \eqref{indexfixcubes1.5} with $l$ replaced by $l+1,\dots,r(i,l)$
implies the existence of a unique sequence $j_{r(i,l)},\dots,j_{l+1}$ such that
the curve
$\gamma_{\tilde q}([a_{j_{k_0},\dots,j_{r(i,l)},\dots,j_{l+1}},
  a_{j_{k_0},\dots,,j_{r(i,l)},\dots,j_{l+1}+1}])$
intersects $K_{Q'}$. If there exists an index
$\tilde l\in\{l,\dots,r(i,l)\}$ such that
$K_{Q'}$ intersects the interiors of at least two cubes in
$\Q_{n+L^0(m_1,k_0)_{\tilde l}}$, there are at most two sequences
$j_{r(i,l)},\dots,j_{l+1}$ with the property that the curve
$\gamma_{\tilde q}([a_{j_{k_0},\dots,j_{r(i,l)},\dots,j_{l+1}},
  a_{j_{k_0},\dots,,j_{r(i,l)},\dots,j_{l+1}+1}])$
intersects $K_{Q'}$, and these sequences are next to each other with respect to
the natural order of sequences given by the layer division points they are
labelling. We
pick up the first one of these sequences. If
$c\in [a_{j_{k_0},\dots,j_{r(i,l)+1}},a_{j_{k_0},\dots,j_{r(i,l)+1}+1}]$,
by Remark~\ref{caniterate2}.(e), there may be two
sequences $j_{r(i,l)},\dots,j_{l+1}$ of this type  --
one before $c$ and one after $c$.

Let $\tilde r(i,l)$ be such that
$L^0(m_1,k_0)_{r(i,l)}=L^i(m_1,k_0)_{\tilde r(i,l)}$. For all
$j_{k_0},\dots,j_{r(i,l)+1}$, define a function
$\chi_{i,l}^{j_{k_0},\dots,j_{r(i,l)+1}}$ by setting
$\chi_{i,l}^{j_{k_0},\dots,j_{r(i,l)+1}}(j_{r(i,l)},\dots,j_{l+1})=1$ provided that
$j_{r(i,l)},\dots,j_{l+1}$ is a sequence determined by some $i$-red cube $Q'$ as
above and, otherwise,
$\chi_{i,l}^{j_{k_0},\dots,j_{r(i,l)+1}}(j_{r(i,l)},\dots,j_{l+1})=0$.
In particular, if $i$ and $l$ are such that there are no $i$-red cubes at level
$n+L^0(m_1,k_0)_{l-1}$, then $\chi_{i,l}^{j_{k_0},\dots,j_{r(i,l)+1}}\equiv 0$.
Note that the function $\chi$ depends on $\omega$ but, for simplicity,
we suppress it from the notation.
Set $\widetilde C:=5^d\mfc 360\sqrt d$ (the factor 360 instead of 15 will be
needed at later stages of the proof). Combining the above
information with Inequality~\eqref{notblue} and multiplying the contribution
of the red cubes by an extra factor 2, to be utilised in the proof of
Lemma~\ref{smallgain}, leads to
\begin{equation}\label{onesum}
\begin{split}  
&\sum_{\substack{j_l\\\gamma_{\tilde q}([a_{j_{k_0},\dots,j_l},a_{j_{k_0},\dots,j_l+1}])\text{ is white}}}
   |\gamma_{\tilde q}(a_{j_{k_0},\dots,j_l})-\gamma_{\tilde q}(a_{j_{k_0},\dots,j_l+1})|\\
 &\ge(1-\widetilde CN^{-m_1l})(|\gamma_{\tilde q}(a_{j_{k_0},\dots,j_{l+1}})
   -\gamma_{\tilde q}(a_{j_{k_0},\dots,j_{l+1}+1})|)\\
 &\quad -\sum_{i=1}^q\chi_{i,l}^{j_{k_0},\dots,j_{r(i,l)+1}}30\sqrt d
   N^{-\Delta_{\tilde r(i,l)}^i}N^{-n-L^i(m_1,k_0)_{\tilde r(i,l)}}.
\end{split}
\end{equation}

By Definition~\ref{j-levels}, the level $n+l_0$, corresponding to the sum over
$j_1$, is such that there may be $i$-red cubes for all $i=1,\dots,i_0+2$ and
$\chi_{i,l}^{j_{k_0},\dots,j_{r(i,l)+1}}\equiv 0$ for $i>i_0+2$. Recalling that only
white curves contain white curve segments, we conclude from \eqref{onesum} that
\begin{align*}
&\sum_{j_{k_0},\dots,j_2}
   \sum_{\substack{j_1\\\gamma_{\tilde q}([a_{j_{k_0},\dots,j_1},a_{j_{k_0},\dots,j_1+1}])\text{ is white}}}
   |\gamma_{\tilde q}(a_{j_{k_0},\dots,j_1})-\gamma_{\tilde q}(a_{j_{k_0},\dots,j_1+1})|\\
 &\ge\sum_{j_{k_0},\dots,j_3}
   \sum_{\substack{j_2\\\gamma_{\tilde q}([a_{j_{k_0},\dots,j_2},a_{j_{k_0},\dots,j_2+1}])\text{ is white}}}
   (1-\widetilde CN^{-m_1})
   (|\gamma_{\tilde q}(a_{j_{k_0},\dots,j_2})-\gamma_{\tilde q}(a_{j_{k_0},\dots,j_2+1})|)\\
 &\quad -\sum_{i=1}^q\sum_{j_{k_0},\dots,j_{r(i,1)+1}}30\sqrt dN^{-\Delta_{\tilde r(i,1)}^i}
   N^{-n-L^i(m_1,k_0)_{\tilde r(i,1)}}\\
 &\quad\times\sum_{j_{r(i,1)},\dots,j_2}
   \chi_{i,1}^{j_{k_0},\dots,j_{r(i,1)+1}}(j_{r(i,1)},\dots,j_2)=:A.
\end{align*}

While summing over $j_2,\dots,j_{r(1,1)-1}$, we only need to subtract the
contribution of blue cubes, since there are 1-red cubes next time at level
$n+L^1(m_1,k_0)_{\tilde r(1,1)}=n+L^0(m_1,k_0)_{r(1,1)}$ corresponding the summing
index $j_{r(1,1)+1}$. Further, the contribution of
$\chi_{1,1}^{j_{k_0},\dots,j_{r(1,1)+1}}(j_{r(1,1)},\dots,j_2)$ cannot be estimated before
we reach the summing index $j_{r(1,1)}$. By construction, every $i$-red cube
$Q'\in\Q_{n+L^0(m_1,k_0)_{l-1}}$ defines at most two sequences
$j_{r(i,l)},\dots,j_{l+1}$ such that
$\chi_{i,l}^{j_{k_0},\dots,j_{r(i,l)+1}}(j_{r(i,l)},\dots,j_{l+1})\ne 0$.
By \eqref{indexfixcubes1} with $l$ replaced by $r(i,l)$, there are at most 
$5^d$ cubes which are $i$-red-labelled parents for some $Q'$ related to a fixed
$\chi_{i,l}^{j_{k_0},\dots,j_{r(i,l)+1}}$. Thus, by the second claim of
Lemma~\ref{colourcontribution} and Inequality~\eqref{indexfixcubes2}, we have
that
\begin{equation}\label{1-redcontribution}
\begin{split}
30&\sqrt dN^{-\Delta_{\tilde r(1,1)}^1}N^{-n-L^1(m_1,k_0)_{\tilde r(1,1)}}
  \sum_{j_{r(1,1)},\dots,j_2}\chi_{1,1}^{j_{k_0},\dots,j_{r(1,1)+1}}(j_{r(1,1)},\dots,j_2)\\
&\le 5^d\mfc 60\sqrt d N^{-\Delta_{\tilde r(1,1)}^1}
  |\gamma_{\tilde q}(a_{j_{k_0},\dots,j_{r(1,1)+1}})
  -\gamma_{\tilde q}(a_{j_{k_0},\dots,j_{r(1,1)+1}+1})|.
\end{split}
\end{equation}
Combining the above facts, we obtain that 
\begin{align*}
A&\ge\sum_{\substack{j_{k_0},\dots,j_{r(1,1)+1}\\
   \gamma_{\tilde q}([a_{j_{k_0},\dots,j_{r(1,1)+1}},a_{j_{k_0},\dots,j_{r(1,1)+1}+1}])\text{ is white}}}
  \Bigl[\Bigl(\prod_{r=1}^{r(1,1)}(1-\widetilde CN^{-m_1r})\Bigr)\\
 &\quad\times|\gamma_{\tilde q}(a_{j_{k_0},\dots,j_{r(1,1)+1}})
  -\gamma_{\tilde q}(a_{j_{k_0},\dots,j_{r(1,1)+1}+1})|\\
 &\quad -\widetilde CN^{-\Delta_{\tilde r(1,1)}^1}
  |\gamma_{\tilde q}(a_{j_{k_0},\dots,j_{r(1,1)+1}})
  -\gamma_{\tilde q}(a_{j_{k_0},\dots,j_{r(1,1)+1}+1})|\Bigr]\\
 &\quad -\sum_{i=2}^q\sum_{j_{k_0},\dots,j_{r(i,1)+1}}30\sqrt dN^{-\Delta_{\tilde r(i,1)}^i}
  N^{-n-L^i(m_1,k_0)_{\tilde r(i,1)}}\\
 &\quad\times\sum_{j_{r(i,1)},\dots,j_2}
  \chi_{i,1}^{j_{k_0},\dots,j_{r(i,1)+1}}(j_{r(i,1)},\dots,j_2)=:B.
\end{align*}
There exists $M_1$, depending only on $d$ and $\mfc$, such that,
for all $m_1\ge M_1$, we have that
$\prod_{r=1}^\infty(1-2\widetilde CN^{-m_1r})>\frac 12$ (the factor 2 appearing
in front of $\widetilde C$ is only needed at later stages of the proof), which
implies that
\begin{align*}
\Bigl( &\prod_{r=1}^{r(1,1)}(1-\widetilde CN^{-m_1r})\Bigr)
   |\gamma_{\tilde q}(a_{j_{k_0},\dots,j_{r(1,1)+1}})
   -\gamma_{\tilde q}(a_{j_{k_0},\dots,j_{r(1,1)+1}+1})|\\
 &-\widetilde CN^{-\Delta_{\tilde r(1,1)}^1}
   |\gamma_{\tilde q}(a_{j_{k_0},\dots,j_{r(1,1)+1}})
   -\gamma_{\tilde q}(a_{j_{k_0},\dots,j_{r(1,1)+1}+1})|\\ 
\ge & (1-2\widetilde CN^{-\Delta_{\tilde r(1,1)}^1})\Bigl(\prod_{r=1}^{r(1,1)}
   (1-\widetilde CN^{-m_1r})\Bigr)|\gamma_{\tilde q}(a_{j_{k_0},\dots,j_{r(1,1)+1}})
   -\gamma_{\tilde q}(a_{j_{k_0},\dots,j_{r(1,1)+1}+1})|. 
\end{align*}
Therefore, using Inequality~\eqref{onesum} with $l=r(1,1)+1$ and recalling
that $i=1$ gives the only non-zero contribution to the sum in \eqref{onesum},
we deduce that
\begin{align*} 
B&\ge\sum_{\substack{j_{k_0},\dots,j_{r(1,1)+2}\\
   \gamma_{\tilde q}([a_{j_{k_0},\dots,j_{r(1,1)+2}},a_{j_{k_0},\dots,j_{r(1,1)+2}+1}])\text{ is white}}}
  (1-2\widetilde CN^{-\Delta_{\tilde r(1,1)}^1})\\
 &\quad\times\Bigl(\prod_{r=1}^{r(1,1)+1}(1-\widetilde CN^{-m_1r})\Bigr)
  |\gamma_{\tilde q}(a_{j_{k_0},\dots,j_{r(1,1)+2}})
  -\gamma_{\tilde q}(a_{j_{k_0},\dots,j_{r(1,1)+2}+1})|\\
 &\quad -\sum_{j_{k_0},\dots,j_{r(1,r(1,1)+1)+1}}30\sqrt dN^{-\Delta_{\tilde r(1,r(1,1)+1)}^1}
  N^{-n-L^1(m_1,k_0)_{\tilde r(1,r(1,1)+1)}}\\
 &\quad\times\sum_{j_{r(1,r(1,1)+1)},\dots,j_{r(1,1)+2}}
  \chi_{1,r(1,1)+1}^{j_{k_0},\dots,j_{r(1,r(1,1)+1)+1}}(j_{r(1,r(1,1)+1)},\dots,j_{r(1,1)+2})\\
 &\quad -\sum_{i=2}^q\sum_{j_{k_0},\dots,j_{r(i,1)+1}}30\sqrt dN^{-\Delta_{\tilde r(i,1)}^i}
  N^{-n-L^i(m_1,k_0)_{\tilde r(i,1)}}\\
 &\quad\times\sum_{j_{r(i,1)},\dots,j_2}
  \chi_{i,1}^{j_{k_0},\dots,j_{r(i,1)+1}}(j_{r(i,1)},\dots,j_2)=:D.
\end{align*}

We proceed by estimating $D$ in a similar manner. When computing
the contribution of $i$-red cubes for $i=1,\dots,q'$, where
$q'\le q$, we apply the third claim of
Lemma~\ref{colourcontribution}, that is, instead of
Inequality~\eqref{1-redcontribution}, we have
\begin{align*}
30&\sqrt d\sum_{i=1}^{q'}\sum_{j_{r(i,l_i)},\dots,j_{l_i+1}}N^{-\Delta_{\tilde r(i,l_i)}^i}
  N^{-n-L^i(m_1,k_0)_{\tilde r(i,l_i)}}
  \chi_{i,l_i}^{j_{k_0},\dots,j_{r(i,l_i)+1}}(j_{r(i,l_i)},\dots,j_{l_i+1})\\
 &\le 5^d\cdot 6\mfc 60\sqrt d N^{-\Delta_{\tilde r(1,l_1)}^1}
  |\gamma_{\tilde q}(a_{j_{k_0},\dots,j_{r(1,l_1)+1}})
  -\gamma_{\tilde q}(a_{j_{k_0},\dots,j_{r(1,l_1)+1}+1})|,
\end{align*}
explaining the factor 360 in the definition of $\widetilde C$.
Proceeding in this way, recalling Remark~\ref{caniterate2}.(d) while summing
over $j_{k_0}$ and recalling \eqref{*dj1}, we end up with the estimate
\begin{align*}
D&\ge (1-CN^{-m_1k_0})(|\gamma(a)-\gamma(c)|+|\gamma(c)-\gamma(b)|)\Bigl(
  \prod_{r=1}^{k_0}(1-\widetilde CN^{-m_1r})\Bigr)\\
 &\quad\,\times\prod_{s=k_0+1}^{k_1}(1-2\widetilde CN^{-\Delta_s^1})\\
 &\ge (1-\widetilde C_1N^{-m_1})(|\gamma(a)-\gamma(c)|+|\gamma(c)-\gamma(b)|),
\end{align*} 
where $\widetilde C_1$ depends only on $d$ and $\mfc$.

The proof of the second claim is similar. The role of blue cubes is taken
by 1-red cubes and the last claim of Lemma~\ref{colourcontribution} is
utilised. Instead of the factor
\[
\Bigl(\prod_{r=1}^{k_0}(1-\widetilde CN^{-m_1r})\Bigr)
  \prod_{s=k_0+1}^{k_1}(1-2\widetilde CN^{-\Delta_s^1})
\]
the computation results, by using \eqref{*dj1} for $i=2$, in the factor 
\[
\Bigl(\prod_{r=k_0+1}^{k_1}(1-\widetilde CN^{-m_1r})\Bigr)
  \prod_{s=k_1+1}^{k_2}(1-2\widetilde CN^{-\Delta_s^2})
\]
leading to a constant different from $\widetilde C_1$, denoted by $C_1$.
\end{proof}

\begin{remark}\label{restricteddivision}
(a) By construction (recall Remark~\ref{caniterate2}.(a)), we have that
\[
|\gamma(a)-\gamma(c)|+|\gamma(c)-\gamma(b)|\le 10\sqrt d N^{-n}\text{ and }
|\gamma_{\tilde q}(a_{j_{k_0},\dots,j_1})-\gamma_{\tilde q}(a_{j_{k_0},\dots,j_1+1})|
\ge N^{-l_0-n}
\]
for all $(j_{k_0},\dots,j_1)$. Therefore, one may choose a subcollection
$J\subset\{(j_{k_0},\dots,j_1)\mid j_l=0,\dots,p_l+1, l=k_0,\dots,1\}$
with $\# J\le\lfloor 10\sqrt d N^{l_0}\rfloor+1< C_0N^{l_0}$, where
$C_0:=11\sqrt d$, such that Proposition~\ref{paintedinhereditarily} is valid
when the sum is restricted to the indices in $J$.

(b) If $q>i_0+2$ in Proposition~\ref{paintedinhereditarily}, there may be curve
segments which are primed with $i$-red primer but not painted with
$i$-red paint for $i>i_0+2$. These curve segments play no role in
Proposition~\ref{paintedinhereditarily} (recall
Remark~\ref{caniterate2}.(b)).
However, in the proof of Proposition~\ref{paintedinhereditarily}  we used
Lemma~\ref{colourcontribution}, which takes into account also the contribution
of these $i$-red segments. We will use this fact in the proof of
Proposition~\ref{notalphaholder} later.
\end{remark}

Next we estimate the increase rate of the length of the broken line
approximation provided the curve stays close to the fractal percolation set
$E(\omega)$.  Next lemma is a key observation guaranteeing length gain.

\begin{lemma}\label{lengthgain}
Fix $\omega\in\Omega$ and $m_0\in\N\setminus\{0\}$. Let $n\in\N$ and $Q\in\Q_n$.
Suppose that $Q'$ is $m_0$-good for all $Q'\in\Q_n(K_Q)$. Let
$\gamma\colon [a,b]\to K_Q$ be a curve passing through an $(n,i)$-layer for some
$i\in\{1,\dots,d\}$. 
Suppose further that there are no points $\tilde a,\tilde b\in[a,b]$ such that 
$\gamma(\mathopen]\tilde a,\tilde b\mathclose[)\cap E(\omega)=\emptyset$
    and
$|\gamma(\tilde a)-\gamma(\tilde b)|\ge d^{-1}N^{-m_0-n}$. Then there is 
$c\in\mathopen]a,b\mathclose[$ such that
\begin{equation}\label{*cgain}
|\gamma(a)-\gamma(c)|+|\gamma(c)-\gamma(b)|\ge(1+C_2N^{-2m_0})
  |\gamma(a)-\gamma(b)|,
\end{equation}
where $C_2$ depends only on $d$.
\end{lemma}

\begin{proof}
Observe that $|\gamma(a)-\gamma(b)|\le 5\sqrt d N^{-n}$. If $\gamma$ passes
through an $(n,j)$-layer for several
$j\in\{1,\dots,d\}$, we consider the index $j$ which maximises the length of
$\Pi_j(L(\gamma(a),\gamma(b)))$. By the contrapositive form of Lemma~\ref{gap},
there exists $c\in\mathopen]a,b\mathclose[$ such that
\[
\dist(\gamma(c),L(\gamma(a),\gamma(b)))\ge\frac 1{2\sqrt 2}N^{-n-m_0}
  \ge (10\sqrt{2d})^{-1}N^{-m_0}|\gamma(a)-\gamma(b)|.
\]
Under this condition, the minimum of
$|\gamma(a)-\gamma(c)|+|\gamma(c)-\gamma(b)|$ is attained when
$\gamma(c)$ is in the 
hyperplane consisting of the points which are equally far away from
$\gamma(a)$ and $\gamma(b)$. Combining this with the fact that 
$\sqrt{1+x}\ge1+\frac x{2\sqrt 2}$ for $0<x<1$ leads to
\begin{align*}
|\gamma(a)-\gamma(c)|+|\gamma(c)-\gamma(b)|
 &\ge 2\sqrt{(\tfrac 12)^2+(10\sqrt{2d})^{-2}N^{-2m_0}}|\gamma(a)-\gamma(b)|\\
 &\ge (1+C_2N^{-2m_0})|\gamma(a)-\gamma(b)|,
\end{align*}
where $C_2:=(100\sqrt 2d)^{-1}$.
\end{proof}

Lemma~\ref{lengthgain} guarantees that $m_0$-good cubes produce a
relative length gain of order $N^{-2m_0}$ provided a curve has no gaps
of relative order $N^{-m_0}$ in $E(\omega)$ (see
Definition~\ref{gapdef}). According to the next lemma, a length gain of the
same relative order is also produced by $(\bL^0(m_1,k_0),\mfc,m_0)$-good cubes.
Note that in Lemma~\ref{lengthgain} the diameter of the curve is of order
$N^{-n}$ and the gap size is $N^{-n-m_0}$.
In the next lemma, the diameter of the curve is still of order $N^{-n}$ but the
gap size is $N^{-n-l_0-m_0}$. The reason why we nevertheless obtain the length
gain of relative order $N^{-2m_0}$ is that the painted curve $\gamma_{\tilde q}$
has at least $N^{l_0}$ white curve segments with diameter of order $N^{-n-l_0}$,
each of them producing a length gain $N^{-2m_0-n-l_0}$ by Lemma~\ref{lengthgain}.
The point $c\in\mathopen]a,b\mathclose[$ in the next lemma is needed later
in  the proof of  Proposition~\ref{notalphaholder} when we iterate our construction.
 
\begin{lemma}\label{smallgain}
Fix $\omega\in\Omega$ and $m_0,\mfc\in\N\setminus\{0\}$.
Let $k_0,m_1\in\N\setminus\{0\}$ be such that
$C_1N^{-m_1}\le\frac 14$, Proposition~\ref{paintedinhereditarily} is valid and
$C_1N^{-m_1k_0}\le\frac 14C_2N^{-2m_0}$, where $C_1$ is
as in Proposition~\ref{paintedinhereditarily} and $C_2$ as in
Lemma~\ref{lengthgain}. Let $q,n\in\N$ and $Q\in\Q_n$. Assume that every
$Q'\in\Q_n(K_Q)$ is $(q,m_1,k_0,\mfc,m_0)$-hereditarily good.
Further, suppose that $\gamma\colon [a,b]\to K_Q$ is a curve passing through an
$(n,i)$-layer for some $i\in\{1,\dots,d\}$. Fix
$c\in\mathopen]a,b\mathclose[$. Assume that there are no
points $\tilde a,\tilde b\in[a,b]$ such that 
$\gamma(\mathopen]\tilde a,\tilde b\mathclose[)\cap E(\omega)=\emptyset$ and
$|\gamma(\tilde a)-\gamma(\tilde b)|\ge (4d)^{-1}N^{-m_0-l_0-n}$.
Let $\gamma_{\tilde q}$
be the painted curve obtained when applying Construction~\ref{brokenline} to
$\gamma$. Then there exist $C_3>0$, depending only on $d$, and a
sequence of points $a\le b_1<d_1\le\dots\le b_{2M}<d_{2M}\le b$ such that
\begin{equation}\label{smallgainineq}
\sum_{\substack{j=1\\\gamma_{\tilde q}([b_j,d_j])\text{ is not red}}}^{2M}
  |\gamma_{\tilde q}(b_j)-\gamma_{\tilde q}(d_j)|
  \ge (1+C_3N^{-2m_0})\bigl(|\gamma(a)-\gamma(c)|+|\gamma(c)-\gamma(b)|\bigr),
\end{equation}
where $M<C_0N^{l_0}$.
\end{lemma}

\begin{proof}
Let $a_{j_{k_0},\dots,j_1}$ be the layer division points defined in
Construction~\ref{brokenline}. If $q\le i_0+2$, there are no segments in
$\gamma_{\tilde q}$ which are primed but not painted (recall
Remark~\ref{caniterate2}.(b)). Therefore, if
$\gamma_{\tilde q}([a_{j_{k_0},\dots,j_1},a_{j_{k_0},\dots,j_1+1}])$ is white, then the
curve
$\gamma_{\tilde q}\colon [a_{j_{k_0},\dots,j_1},a_{j_{k_0},\dots,j_1+1}]
\to K_{Q_{j_{k_0},\dots,j_1}}$
satisfies the assumptions of Lemma~\ref{lengthgain}, since
$\gamma_{\tilde q}|_{[a_{j_{k_0},\dots,j_1},a_{j_{k_0},\dots,j_1+1}]}
  =\gamma|_{[a_{j_{k_0},\dots,j_1},a_{j_{k_0},\dots,j_1+1}]}$.
Let $c_{j_{k_0},\dots,j_1}\in [a_{j_{k_0},\dots,j_1},a_{j_{k_0},\dots,j_1+1}]$ be the point
given by Lemma~\ref{lengthgain}. We get
\begin{equation}\label{individualgain}
\begin{split}
 &|\gamma_{\tilde q}(a_{j_{k_0},\dots,j_1})-\gamma_{\tilde q}(c_{j_{k_0},\dots,j_1})|
   +|\gamma_{\tilde q}(c_{j_{k_0},\dots,j_1})-\gamma_{\tilde q}(a_{j_{k_0},\dots,j_1+1})|\\
 &\ge (1+C_2N^{-2m_0})|\gamma_{\tilde q}(a_{j_{k_0},\dots,j_1})
   -\gamma_{\tilde q}(a_{j_{k_0},\dots,j_1+1})|.
\end{split}
\end{equation}
If $q>i_0+2$, $\gamma_{\tilde q}$ and $\gamma$ may differ even on white
segments if they contain parts which are primed with $i$-red primer but not
painted with $i$-red paint for some $i>i_0+2$. In this case, $\gamma_{\tilde q}$
contains a line segment included in $K_{Q'}$ for some $i$-red cube $Q'$ (recall
Algorithm~\ref{modification}). If the sum of the side lengths of all such
cubes $K_{Q'}$ is at least $(2d)^{-1}N^{-m_0-l_0-n}$, we can use the extra factor
2 introduced in \eqref{onesum} to obtain the length gain
\eqref{individualgain}, recalling that $C_2N^{-2m_0}<(2d)^{-1}N^{-m_0}$. If the
above sum is less than $(2d)^{-1}N^{-m_0-l_0-n}$ and if the conclusion of
Lemma~\ref{lengthgain} is not valid for
$\gamma_{\tilde q}|_{[a_{j_{k_0},\dots,j_1},a_{j_{k_0},\dots,j_1+1}]}$, there is
$j\in\{1,\dots,d\}$ and a $j$-layer $\L$ of width $d^{-1}N^{-m_0-l_0-n}$ such
that $\gamma_{\tilde q}|_{[a_{j_{k_0},\dots,j_1},a_{j_{k_0},\dots,j_1+1}]}$ passes through $\L$
without intersecting $E(\omega)$ (recall the proof of Lemma~\ref{gap}).
Since the total length of the modified part of
$\gamma_{\tilde q}|_{[a_{j_{k_0},\dots,j_1},a_{j_{k_0},\dots,j_1+1}]}$ is less than
$(2d)^{-1}N^{-m_0-l_0-n}$, there is $t\in [a_{j_{k_0},\dots,j_1},a_{j_{k_0},\dots,j_1+1}]$
such that $\gamma_{\tilde q}(t)=\gamma(t)$ and
$\dist(\gamma(t),\L^c)\ge (4d)^{-1}N^{-m_0-l_0-n}$, leading to a contradiction with
assumptions on $\gamma$. Therefore, we obtain \eqref{individualgain} in all
cases.

If $\gamma_{\tilde q}([a_{j_{k_0},\dots,j_1},a_{j_{k_0},\dots,j_1+1}])$ is blue, we will
not have any length gain but, for notational simplicity, we set
$c_{j_{k_0},\dots,j_1}:=\frac 12(a_{j_{k_0},\dots,j_1}+a_{j_{k_0},\dots,j_1+1})$.
By Remark~\ref{restricteddivision}, we may choose a set $J$ of indices
$(j_{k_0},\dots,j_1)$ such that $\#J<C_0N^{l_0}$ and
Proposition~\ref{paintedinhereditarily} remains valid when summing
over $J$.
Define points $b_i$ and $d_i$ by setting
\begin{align*}
&\{b_{2p-1}\mid p=1,\dots M\}=\{a_{j_{k_0},\dots,j_1}\mid (j_{k_0},\dots,j_1)\in J\},\\
&\{d_{2p-1}\mid p=1,\dots M\}=\{c_{j_{k_0},\dots,j_1}\mid (j_{k_0},\dots,j_1)\in J\}
=\{b_{2p}\mid p=1,\dots M\}\text{ and }\\
&\{d_{2p}\mid p=1,\dots M\}=\{a_{j_{k_0},\dots,j_1+1}\mid (j_{k_0},\dots,j_1)\in J\}.
\end{align*}  
Using \eqref{individualgain}, Proposition~\ref{paintedinhereditarily} and the
choices of $m_1$ and $k_0$, we obtain that
\begin{align*}
&\sum_{\substack{j=1\\\gamma_{\tilde q}([b_j,d_j])\text{ is not red}}}^{2M}
  |\gamma_{\tilde q}(b_j)-\gamma_{\tilde q}(d_j)|\\
&=\sum_{\substack{(j_{k_0},\dots,j_1)\in J\\\gamma_{\tilde q}([a_{j_{k_0},\dots,j_1},a_{j_{k_0},\dots,j_1+1}])
  \text{ is white}}}\Bigl(|\gamma_{\tilde q}(a_{j_{k_0},\dots,j_1})
  -\gamma_{\tilde q}(c_{j_{k_0},\dots,j_1})|\\
&\qquad\qquad\qquad\qquad\qquad\qquad\quad
  +|\gamma_{\tilde q}(c_{j_{k_0},\dots,j_1})
  -\gamma_{\tilde q}(a_{j_{k_0},\dots,j_1+1})|\Bigr)\\
&\quad+\sum_{\substack{(j_{k_0},\dots,j_1)\in J\\
  \gamma_{\tilde q}([a_{j_{k_0},\dots,j_1},a_{j_{k_0},\dots,j_1+1}])
  \text{ is blue}}}\Bigl(|\gamma_{\tilde q}(a_{j_{k_0},\dots,j_1})
  -\gamma_{\tilde q}(c_{j_{k_0},\dots,j_1})|\\
&\qquad\qquad\qquad\qquad\qquad\qquad\quad
  +|\gamma_{\tilde q}(c_{j_{k_0},\dots,j_1})
  -\gamma_{\tilde q}(a_{j_{k_0},\dots,j_1+1})|\Bigr)\\ 
&\ge\sum_{\substack{(j_{k_0},\dots,j_1)\in J\\
  \gamma_{\tilde q}([a_{j_{k_0},\dots,j_1},a_{j_{k_0},\dots,j_1+1}])
  \text{ is white}}}(1+C_2N^{-2m_0})
  |\gamma_{\tilde q}(a_{j_{k_0},\dots,j_1})-\gamma_{\tilde q}(a_{j_{k_0},\dots,j_1+1})|\\
&\quad+\sum_{\substack{(j_{k_0},\dots,j_1)\in J\\
  \gamma_{\tilde q}([a_{j_{k_0},\dots,j_1},a_{j_{k_0},\dots,j_1+1}])
  \text{ is blue}}}|\gamma_{\tilde q}(a_{j_{k_0},\dots,j_1})
  -\gamma_{\tilde q}(a_{j_{k_0},\dots,j_1+1})|\\
&\ge\sum_{\substack{(j_{k_0},\dots,j_1)\in J\\
   \gamma_{\tilde q}([a_{j_{k_0},\dots,j_1},a_{j_{k_0},\dots,j_1+1}])\text{ is not red}}}
  |\gamma_{\tilde q}(a_{j_{k_0},\dots,j_1})-\gamma_{\tilde q}(a_{j_{k_0},\dots,j_1+1})|\\
&\quad+C_2N^{-2m_0}(1-C_1N^{-m_1})(|\gamma(a)-\gamma(c)|+|\gamma(c)-\gamma(b)|)
\end{align*}
\begin{equation}\label{*critineq}
\ge (1-C_1N^{-m_1k_0}+\frac 34C_2N^{-2m_0})
   (|\gamma(a)-\gamma(c)|+|\gamma(c)-\gamma(b)|),
\end{equation}  
which gives the claim with $C_3:=\frac 12C_2$.
\end{proof}

 Next we utilise Lemma~\ref{smallgain} iteratively to prove Proposition \ref{notalphaholder}. 


\begin{proof}[Proof of Proposition \ref{notalphaholder}]
Apply Construction~\ref{brokenline} to $\gamma$ without fixing any
$c\in\mathopen]a,b\mathclose[$. Choosing $h=0$ in the assumptions, we see that
the assumptions of Lemma~\ref{smallgain} are satisfied. Applying
Lemma~\ref{smallgain} to $\gamma$, we obtain points
$\tilde b_j$ and $\tilde d_j$, $j=1,\dots,\widetilde M$, with
$\widetilde M<C_0N^{l_0}$. For every
$p\in\{1,\dots,\widetilde M\}$, there is $(j_{k_0},\dots,j_1)$ such that
$\tilde b_{2p-1}=a_{j_{k_0},\dots,j_1}$, $\tilde d_{2p}=a_{j_{k_0},\dots,j_1+1}$,
$\tilde d_{2p-1}=\tilde b_{2p}=c_{j_{k_0},\dots,j_1}\in
    \mathopen]a_{j_{k_0},\dots,j_1},a_{j_{k_0},\dots,j_1+1}\mathclose[$
and $\gamma_{\tilde q}([a_{j_{k_0},\dots,j_1},a_{j_{k_0},\dots,j_1+1}])$ is either white or
blue. By Remark~\ref{caniterate2}.(a), the conditions of
Construction~\ref{brokenline} are valid for
$\gamma_{\tilde q}|_{[\tilde b_{2p-1},\tilde d_{2p}]}$
with $q$ replaced by $q-1$. If $\gamma_{\tilde q}([\tilde b_{2p-1},\tilde d_{2p}])$
is white, apply Construction~\ref{brokenline} to it with $c=\tilde b_{2p}$
recalling that, if $q>i_0+2$, a part of
$\gamma_{\tilde q}([\tilde b_{2p-1},\tilde d_{2p}])$
may already be primed with an $i$-red
primer for $i>i_0$ (recall Remark~\ref{caniterate2}.(b)). If
$\gamma_{\tilde q}([\tilde b_{2p-1},\tilde d_{2p}])$ is blue, apply
Construction~\ref{brokenline} to it without fixing any $c$ and interpreting that
it is white, that is, the blue paint of
$\gamma_{\tilde q}([\tilde b_{2p-1},\tilde d_{2p}])$ is not inherited by its
subcurves. Choosing $h=1$ in the assumptions, we see that
$\gamma|_{[\tilde b_{2p-1},\tilde d_{2p}]}$ satisfies the assumptions of
Lemma~\ref{smallgain} for all $p\in\{1,\dots,\widetilde M\}$. By means of
Lemma~\ref{smallgain}, we find new points $\tilde b_j$ and $\tilde d_j$ as
above. Repeat
this process $q+1$ times until the level $n+L^q(m_1,k_0)_0$ is reached and the
final points $a\le b_1<d_1\le\dots\le b_{2M}<d_{2M}\le b$ are defined. Note that
at this level all primed parts are painted which implies that $\gamma_{\tilde q}$
and $\gamma$ agree on white and blue segments. Therefore,
$\gamma_{\tilde q}(b_i)=\gamma(b_i)$ and $\gamma_{\tilde q}(d_i)=\gamma(d_i)$ for
all $i=1,\dots,2M$.
By Lemma~\ref{smallgain}, at every step the number of chosen subcurves is less
than $C_0N^{l_0}$, so $M<(C_0N^{l_0})^{q+1}$. Recalling
Remark~\ref{restricteddivision}.(b), the claim follows by using
recursively the estimate given by Lemma~\ref{smallgain} starting from the
lowest level $n+L^q(m_1,k_0)_0$ and finishing at level $n$. 
\end{proof}

\section*{Acknowledgement}

We thank the referee for  several useful comments improving the
readability and clarifying the structure of the paper.

\end{document}

%% file: fig33.tex
\scalebox{1} 
{
\begin{pspicture}(0,-4.621339)(21.659063,4.6246624)
\definecolor{color96c}{rgb}{1,1,1}
\definecolor{color134c}{rgb}{1,1,1}
\definecolor{color3b}{rgb}{0.00392156862745098,0.00392156862745098,0.00392156862745098}
\definecolor{color6b}{rgb}{0.6078431372549019,0.6,0.6}
\rput(11.848062,-4.562783){\rput{-270.0}(9.1,0.0){\psgrid[gridwidth=0.0182,subgridwidth=0.014111111,gridlabels=0.0pt,unit=1.3cm,subgridcolor=color96c](0,0)(0,0)(7,7)
\psset{unit=1.0cm}}}
\rput(0.0,-4.601662){\psgrid[gridwidth=0.0182,subgridwidth=0.014111111,gridlabels=0.0pt,unit=1.3cm,subgridcolor=color134c](0,0)(0,0)(7,7)
\psset{unit=1.0cm}}
\psframe[linewidth=0.04,dimen=outer,fillstyle=solid,fillcolor=black](2.6066666,1.9116713)(1.2866666,-2.008329)
\psframe[linewidth=0.016,dimen=outer,fillstyle=solid,fillcolor=color3b](5.243333,3.1983378)(3.9066667,0.6016712)
\psframe[linewidth=0.016,dimen=outer,fillstyle=solid,fillcolor=color3b](6.54,1.9183378)(5.18,0.6016712)
\pscircle[linewidth=0.016,dimen=outer,fillstyle=solid,fillcolor=color6b](1.9533334,0.025004547){0.13333334}
\psline[linewidth=0.046](11.890284,-4.553895)(20.890284,4.512772)(20.95695,4.601662)
\psline[linewidth=0.046cm](11.890284,-4.5983386)(20.934729,3.201661)
\psline[linewidth=0.046cm](11.934728,-4.5983386)(20.912506,1.8905512)
\psline[linewidth=0.046cm](11.934728,-4.5761156)(20.95695,0.64610523)
\psline[linewidth=0.046cm](11.868061,-4.531672)(20.890284,-0.6650038)
\psline[linewidth=0.046cm](11.8458395,-4.5761156)(20.95695,-2.0205607)
\psline[linewidth=0.046cm](11.912506,-4.553895)(20.912506,-3.3094497)
\psline[linewidth=0.046cm](11.847009,-4.53752)(20.962797,-4.53752)
\usefont{T1}{ptm}{m}{n}
\rput(11.104531,-4.387172){$v$}
\usefont{T1}{ptm}{m}{n}
\rput(21.29453,-4.1871724){$F$}
\end{pspicture} 
}

%% file: figpercol3.tex
\scalebox{1} 
{
\begin{pspicture}(0,-2.3664062)(13.88828,2.3264062)
\definecolor{color527c}{rgb}{0.5019607843137255,0.5019607843137255,0.5019607843137255}
\definecolor{color529b}{rgb}{0.7647058823529411,0.7647058823529411,0.7647058823529411}
\rput(1.5282811,-1.6735938){\psgrid[gridwidth=0.04,subgridwidth=0.014111111,gridlabels=0.0pt,subgriddiv=1,unit=0.25cm,subgridcolor=color527c](0,0)(0,0)(16,16)
\psset{unit=1.0cm}}
\psframe[linewidth=0.04,dimen=outer,fillstyle=solid,fillcolor=color529b](2.296406,-0.67453104)(2.021875,-0.94906235)
\rput(8.52828,-1.6735938){\psgrid[gridwidth=0.04,subgridwidth=0.014111111,gridlabels=0.0pt,subgriddiv=1,unit=0.25cm,subgridcolor=color527c](0,0)(0,0)(16,16)
\psset{unit=1.0cm}}
\usefont{T1}{ptm}{m}{n}
\rput(0.62453127,-0.84359324){$\ell_a$}
\usefont{T1}{ptm}{m}{n}
\rput(7.6945314,-0.48359334){$\ell_b$}
\usefont{T1}{ptm}{m}{n}
\rput(1.2745311,-1.9635936){$v$}
\usefont{T1}{ptm}{m}{n}
\rput(8.434532,-1.9435939){$v$}
\usefont{T1}{ptm}{m}{n}
\rput(3.3845313,-2.1635938){$Q$}
\usefont{T1}{ptm}{m}{n}
\rput(10.384531,-2.1635938){$Q$}
\usefont{T1}{ptm}{m}{n}
\rput(5.8045306,0.05640661){$F$}
\usefont{T1}{ptm}{m}{n}
\rput(12.804532,0.07640663){$F$}
\psframe[linewidth=0.04,dimen=outer,fillstyle=solid,fillcolor=color529b](2.536406,-0.41453105)(2.261875,-0.6890625)
\psframe[linewidth=0.04,dimen=outer,fillstyle=solid,fillcolor=color529b](2.796406,-0.41453105)(2.5218751,-0.6890625)
\psframe[linewidth=0.04,dimen=outer,fillstyle=solid,fillcolor=color529b](3.056406,-0.41453105)(2.7818751,-0.6890625)
\psframe[linewidth=0.04,dimen=outer,fillstyle=solid,fillcolor=color529b](9.296407,-0.15453108)(9.021875,-0.42906234)
\psframe[linewidth=0.04,dimen=outer,fillstyle=solid,fillcolor=color529b](9.5364065,-0.15453108)(9.261875,-0.42906234)
\psframe[linewidth=0.04,dimen=outer,fillstyle=solid,fillcolor=color529b](9.796407,-0.15453108)(9.521874,-0.42906234)
\psframe[linewidth=0.04,dimen=outer,fillstyle=solid,fillcolor=color529b](10.0364065,0.08546893)(9.761875,-0.18906236)
\psline[linewidth=0.04cm](1.0082811,-1.1435933)(6.7082815,0.8264068)
\psline[linewidth=0.04cm](8.168282,-0.65359336)(13.868279,1.0464065)
\psdots[dotsize=0.12](1.5254687,-0.9735938)
\usefont{T1}{ptm}{m}{n}
\rput(1.3145312,-1.2635938){$a$}
\usefont{T1}{ptm}{m}{n}
\rput(8.324532,-0.8635937){$b$}
\psdots[dotsize=0.12](8.525469,-0.5535937)
\psdots[dotsize=0.12](1.52,-1.6735938)
\psdots[dotsize=0.12](8.52,-1.6735938)
\end{pspicture} 
}

%% file: fpcfig39.tex
\scalebox{1} 
{
\begin{pspicture}(0,-6.0242186)(17.025469,6.0092187)
\definecolor{color26b}{rgb}{0.6,0.6,0.6}
\definecolor{color27b}{rgb}{0.8,0.8,0.8}
\psframe[linewidth=0.016,dimen=outer,fillstyle=solid,fillcolor=color26b](9.011875,-0.5807812)(8.011875,-1.5807812)
\psframe[linewidth=0.016,dimen=outer,fillstyle=solid,fillcolor=color27b](9.011875,-1.5807812)(8.011875,-2.5807812)
\psframe[linewidth=0.016,dimen=outer,fillstyle=solid,fillcolor=color27b](9.011875,0.41921875)(8.011875,-0.5807812)
\psline[linewidth=0.04cm,arrowsize=0.07291667cm 1.68,arrowlength=1.56,arrowinset=0.32]{->}(2.0918748,-4.5807815)(13.091875,-4.5807815)
\psline[linewidth=0.04cm](0.8118749,-2.6407812)(14.811875,-0.30078125)
\psline[linewidth=0.04cm](3.0118747,4.4192185)(12.011875,4.4192185)
\psline[linewidth=0.04cm](3.0118747,4.4192185)(3.0118747,-4.5807815)
\psline[linewidth=0.04cm](3.0118747,3.4192188)(12.011875,3.4192188)
\psline[linewidth=0.04cm](4.0118747,4.4192185)(4.0118747,-4.5807815)
\psline[linewidth=0.04cm](3.0118747,2.4192188)(12.011875,2.4192188)
\psline[linewidth=0.04cm](5.011875,4.4192185)(5.011875,-4.5807815)
\psline[linewidth=0.04cm](3.0118747,1.4192188)(12.011875,1.4192188)
\psline[linewidth=0.04cm](6.011875,4.4192185)(6.011875,-4.5807815)
\psline[linewidth=0.04cm](3.0118747,0.41921875)(12.011875,0.41921875)
\psline[linewidth=0.04cm](7.011875,4.4192185)(7.011875,-4.5807815)
\psline[linewidth=0.04cm](3.0118747,-0.5807812)(12.011875,-0.5807812)
\psline[linewidth=0.04cm](8.011875,4.4192185)(8.011875,-4.5807815)
\psline[linewidth=0.04cm](3.0118747,-1.5807812)(12.011875,-1.5807812)
\psline[linewidth=0.04cm](9.011875,4.4192185)(9.011875,-4.5807815)
\psline[linewidth=0.04cm](3.0118747,-2.5807812)(12.011875,-2.5807812)
\psline[linewidth=0.04cm](10.011875,4.4192185)(10.011875,-4.5807815)
\psline[linewidth=0.04cm](3.0118747,-3.5807812)(12.011875,-3.5807812)
\psline[linewidth=0.04cm](11.011875,4.4192185)(11.011875,-4.5807815)
\psline[linewidth=0.04cm](3.0118747,-4.5807815)(12.011875,-4.5807815)
\psline[linewidth=0.04cm](12.011875,4.4192185)(12.011875,-4.5807815)
\psline[linewidth=0.04cm](14.571875,-2.0407813)(0.9918749,-5.040781)
\psline[linewidth=0.04cm](14.811874,0.31921875)(1.231875,-2.6807814)
\psdots[dotsize=0.2](3.011875,-4.5807815)
\psdots[dotsize=0.2](12.011875,-3.1107812)
\psdots[dotsize=0.2](12.011875,-2.5807812)
\psdots[dotsize=0.2](14.931874,-0.08078125)
\psdots[dotsize=0.2](12.011875,-0.7607812)
\psdots[dotsize=0.2](0.8318749,-2.9607813)
\psdots[dotsize=0.2](3.011875,-2.2807813)
\psline[linewidth=0.05cm,linestyle=dashed,dash=0.16cm 0.16cm](5.371875,-5.340781)(5.38,5.9842186)
\psdots[dotsize=0.2](5.371875,-4.5807815)
\usefont{T1}{ppl}{m}{n}
\rput(2.2945309,-4.411406){$v$}
\usefont{T1}{ppl}{m}{n}
\rput(12.474531,-3.3014061){$u$}
\usefont{T1}{ppl}{m}{n}
\rput(14.810937,-1.7214062){$\ell'$}
\usefont{T1}{ppl}{m}{n}
\rput(12.534532,-2.3114061){$u'$}
\psline[linewidth=0.04cm](0.9118749,-4.960781)(14.911875,-2.6207812)
\usefont{T1}{ppl}{m}{n}
\rput(12.390937,-1.1892188){$\tilde y$}
\usefont{T1}{ppl}{m}{n}
\rput(1.0009375,-1.6192187){$L(x,y)$}
\usefont{T1}{ppl}{m}{n}
\rput(13.420937,0.41078126){$\ell''$}
\usefont{T1}{ppl}{m}{n}
\rput(15.800938,-0.04921875){$\gamma(b)$}
\usefont{T1}{ppl}{m}{n}
\rput(2.5445313,-1.9314063){$\tilde x$}
\usefont{T1}{ppl}{m}{n}
\rput(1.1145313,-3.1892188){$\gamma(a)$}
\usefont{T1}{ppl}{m}{n}
\rput(5.8145313,-4.911406){$z$}
\usefont{T1}{ppl}{m}{n}
\rput(12.4,2.4185936){$F$}
\usefont{T1}{ppl}{m}{n}
\rput(5.660937,-5.789219){$\Pi_{j}^{-1}(z)$}
\usefont{T1}{ppl}{m}{n}
\rput(14.620937,-0.70140624){$\ell$}
\usefont{T1}{ppl}{m}{n}
\rput(13.055469,-5.020781){$j$-th axis}
\pscustom[linewidth=0.04]
{
\newpath
\moveto(0.84546876,-2.9457812)
\lineto(0.92546874,-2.9057813)
\curveto(0.96546876,-2.8857813)(1.0504688,-2.8257813)(1.0954688,-2.7857811)
\curveto(1.1404687,-2.7457812)(1.2754687,-2.6857812)(1.3654687,-2.6657813)
\curveto(1.4554688,-2.6457813)(1.5954688,-2.5957813)(1.6454687,-2.5657814)
\curveto(1.6954688,-2.5357811)(1.8004688,-2.4757812)(1.8554688,-2.4457812)
\curveto(1.9104687,-2.4157813)(2.0404687,-2.3657813)(2.1154687,-2.3457813)
\curveto(2.1904688,-2.3257813)(2.3704689,-2.2857811)(2.4754686,-2.2657812)
\curveto(2.5804687,-2.2457812)(2.7604687,-2.2157812)(2.8354688,-2.2057812)
\curveto(2.9104688,-2.1957812)(3.0704687,-2.1807814)(3.1554687,-2.1757812)
\curveto(3.2404687,-2.1707811)(3.4804688,-2.1707811)(3.6354687,-2.1757812)
\curveto(3.7904687,-2.1807814)(4.030469,-2.1807814)(4.115469,-2.1757812)
\curveto(4.2004685,-2.1707811)(4.3554688,-2.1507812)(4.425469,-2.1357813)
\curveto(4.4954686,-2.1207812)(4.6604686,-2.0857813)(4.755469,-2.0657814)
\curveto(4.8504686,-2.0457811)(5.025469,-2.0057812)(5.1054688,-1.9857812)
\curveto(5.1854687,-1.9657812)(5.3304687,-1.9307812)(5.3954687,-1.9157813)
\curveto(5.460469,-1.9007813)(5.6204686,-1.8757813)(5.715469,-1.8657813)
\curveto(5.8104687,-1.8557812)(5.990469,-1.8307812)(6.0754685,-1.8157812)
\curveto(6.1604686,-1.8007812)(6.335469,-1.7707813)(6.425469,-1.7557813)
\curveto(6.5154686,-1.7407813)(6.7204685,-1.7007812)(6.835469,-1.6757812)
\curveto(6.9504685,-1.6507813)(7.260469,-1.6007812)(7.4554687,-1.5757812)
\curveto(7.650469,-1.5507812)(7.8754687,-1.5207813)(7.905469,-1.5157813)
\curveto(7.9354687,-1.5107813)(7.990469,-1.5007813)(8.015469,-1.4957813)
\curveto(8.040469,-1.4907813)(8.105469,-1.4857812)(8.145469,-1.4857812)
\curveto(8.185469,-1.4857812)(8.255468,-1.4857812)(8.285469,-1.4857812)
\curveto(8.315469,-1.4857812)(8.370469,-1.4857812)(8.395469,-1.4857812)
\curveto(8.420468,-1.4857812)(8.475469,-1.4707812)(8.505468,-1.4557812)
\curveto(8.535469,-1.4407812)(8.585468,-1.4157813)(8.605469,-1.4057813)
\curveto(8.625469,-1.3957813)(8.670468,-1.3807813)(8.695469,-1.3757813)
\curveto(8.7204685,-1.3707813)(8.780469,-1.3507812)(8.815469,-1.3357812)
\curveto(8.850469,-1.3207812)(8.965468,-1.2857813)(9.045468,-1.2657813)
\curveto(9.125469,-1.2457813)(9.235469,-1.2107812)(9.265469,-1.1957812)
\curveto(9.295468,-1.1807812)(9.350469,-1.1507813)(9.375469,-1.1357813)
\curveto(9.400469,-1.1207813)(9.455469,-1.0857812)(9.485469,-1.0657812)
\curveto(9.515469,-1.0457813)(9.605469,-1.0057813)(9.665469,-0.98578125)
\curveto(9.725469,-0.9657813)(9.8454685,-0.93078125)(9.905469,-0.91578126)
\curveto(9.965468,-0.9007813)(10.050468,-0.8857812)(10.075469,-0.8857812)
\curveto(10.100469,-0.8857812)(10.155469,-0.8857812)(10.185469,-0.8857812)
\curveto(10.215468,-0.8857812)(10.285469,-0.8857812)(10.325469,-0.8857812)
\curveto(10.365469,-0.8857812)(10.450469,-0.8957813)(10.495469,-0.90578127)
\curveto(10.540469,-0.91578126)(10.615469,-0.93078125)(10.645469,-0.93578124)
\curveto(10.675468,-0.94078124)(10.755468,-0.94578123)(10.805469,-0.94578123)
\curveto(10.855469,-0.94578123)(11.000469,-0.9507812)(11.0954685,-0.9557812)
\curveto(11.190469,-0.9607813)(11.320469,-0.9657813)(11.355469,-0.9657813)
\curveto(11.390469,-0.9657813)(11.500469,-0.9657813)(11.575469,-0.9657813)
\curveto(11.650469,-0.9657813)(11.810469,-0.9507812)(11.895469,-0.93578124)
\curveto(11.980469,-0.92078125)(12.120469,-0.8957813)(12.175468,-0.8857812)
\curveto(12.230469,-0.87578124)(12.360469,-0.86578125)(12.435469,-0.86578125)
\curveto(12.5104685,-0.86578125)(12.630468,-0.85078126)(12.675468,-0.8357813)
\curveto(12.7204685,-0.82078123)(12.805469,-0.80078125)(12.8454685,-0.79578125)
\curveto(12.8854685,-0.79078126)(13.005468,-0.7757813)(13.085468,-0.7657812)
\curveto(13.165469,-0.75578123)(13.285469,-0.74578124)(13.325469,-0.74578124)
\curveto(13.365469,-0.74578124)(13.435469,-0.74578124)(13.465468,-0.74578124)
\curveto(13.495469,-0.74578124)(13.550468,-0.73078126)(13.575469,-0.7157813)
\curveto(13.600469,-0.7007812)(13.665469,-0.65578127)(13.705469,-0.62578124)
\curveto(13.745469,-0.59578127)(13.810469,-0.56078124)(13.835468,-0.55578125)
\curveto(13.860469,-0.55078125)(13.905469,-0.53578126)(13.925468,-0.5257813)
\curveto(13.945469,-0.5157812)(13.995469,-0.50078124)(14.025469,-0.49578124)
\curveto(14.055469,-0.49078125)(14.115469,-0.47578126)(14.145469,-0.46578124)
\curveto(14.175468,-0.45578125)(14.270469,-0.43078125)(14.335468,-0.41578126)
\curveto(14.400469,-0.40078124)(14.480469,-0.36578125)(14.495469,-0.34578124)
\curveto(14.5104685,-0.32578126)(14.535469,-0.30078125)(14.545468,-0.29578125)
\curveto(14.555469,-0.29078126)(14.575469,-0.26578125)(14.585468,-0.24578124)
\curveto(14.5954685,-0.22578125)(14.625469,-0.20078126)(14.645469,-0.19578125)
\curveto(14.665469,-0.19078125)(14.710468,-0.18078125)(14.735469,-0.17578125)
\curveto(14.7604685,-0.17078125)(14.795468,-0.14578125)(14.805469,-0.12578125)
\curveto(14.815469,-0.10578125)(14.875469,-0.12578125)(14.925468,-0.16578124)
}
\end{pspicture} 
}

%% file: fig41.tex
%
%

\psscalebox{1.0 1.0} 
{
\begin{pspicture}(0,-4.1497626)(8.283259,4.1497626)
\definecolor{colour4}{rgb}{0.5019608,0.5019608,0.5019608}
\definecolor{colour5}{rgb}{0.41568628,0.41960785,0.83137256}
\rput(0.0,-4.1497626){\psgrid[gridwidth=0.028222222, subgridwidth=0.014111111, gridlabels=0.0pt, subgriddiv=8, unit=2.0708148cm, subgridcolor=colour4](0,0)(0,0)(4,4)
\psset{unit=1.0cm}}
\psframe[linecolor=black, linewidth=0.046, fillstyle=solid,fillcolor=colour5, dimen=outer](4.1981816,4.149762)(2.0381818,2.0697625)
\psframe[linecolor=black, linewidth=0.046, fillstyle=solid,fillcolor=colour5, dimen=outer](6.278182,0.0071536684)(4.1181817,-2.0728464)
\psframe[linecolor=black, linewidth=0.046, fillstyle=solid,fillcolor=colour5, dimen=outer](7.5407906,2.352371)(7.1807904,2.032371)
\psframe[linecolor=black, linewidth=0.046, fillstyle=solid,fillcolor=colour5, dimen=outer](7.8807907,3.372371)(7.5207906,3.052371)
\psframe[linecolor=black, linewidth=0.046, fillstyle=solid,fillcolor=colour5, dimen=outer](7.3407903,1.312371)(6.9807906,0.9923711)
\psframe[linecolor=black, linewidth=0.046, fillstyle=solid,fillcolor=colour5, dimen=outer](8.10079,-0.7476289)(7.7407904,-1.067629)
\psframe[linecolor=black, linewidth=0.046, fillstyle=solid,fillcolor=colour5, dimen=outer](7.3007903,-1.2476289)(6.9407907,-1.567629)
\psframe[linecolor=black, linewidth=0.046, fillstyle=solid,fillcolor=colour5, dimen=outer](1.0807905,1.052371)(0.7207905,0.73237103)
\psframe[linecolor=black, linewidth=0.046, fillstyle=solid,fillcolor=colour5, dimen=outer](1.3607905,3.372371)(1.0007905,3.052371)
\psframe[linecolor=black, linewidth=0.046, fillstyle=solid,fillcolor=colour5, dimen=outer](1.3607905,-3.0876288)(1.0007905,-3.407629)
\psframe[linecolor=black, linewidth=0.046, fillstyle=solid,fillcolor=colour5, dimen=outer](2.4129643,2.0949798)(2.0529644,1.7749797)
\psframe[linecolor=black, linewidth=0.046, fillstyle=solid,fillcolor=colour5, dimen=outer](3.4407904,-0.76762897)(3.0807905,-1.087629)
\psframe[linecolor=black, linewidth=0.046, fillstyle=solid,fillcolor=colour5, dimen=outer](7.8207903,-3.0876288)(7.4607906,-3.407629)
\psframe[linecolor=black, linewidth=0.046, fillstyle=solid,fillcolor=colour5, dimen=outer](5.4903555,-3.0824115)(5.130356,-3.4024115)
\psframe[linecolor=black, linewidth=0.046, fillstyle=solid,fillcolor=colour5, dimen=outer](6.012095,-2.5641506)(5.652095,-2.8841507)
\end{pspicture}
}

%% file: fpcfiglevels.tex
\scalebox{1} 
{
\begin{pspicture}(0,0)(22,24)

\usefont{T1}{ptm}{m}{n}
\rput(1.0,23){$0=L^0(m_1,k_0)_{k_0}$}
\rput(10.0,23.6){$L^1(m_1,k_0)_{k_1}$}
\rput(19.3,23){$L^2(m_1,k_0)_{k_2}$}

\rput(1,15.8){$l_0=L^0(m_1,k_0)_{0}$}
\rput(9.0,15.3){$L^1(m_1,k_0)_{k_0}$}
\rput(19.3,15.8){$L^2(m_1,k_0)_{k_1}$}

\rput(8.3,8.6){$2l_0=L^1(m_1,k_0)_{0}$}
\rput(19.3,8.6){$L^2(m_1,k_0)_{k_0}$}

\rput(19.8,1.4){$3l_0=L^2(m_1,k_0)_{0}$}

\psline[linewidth=0.08cm](3,15.800000)(4,15.800000)
\psline[linewidth=0.08cm](10,8.600000)(11,8.600000)
\psline[linewidth=0.08cm](17,1.400000)(18,1.400000)
\psline[linewidth=0.08cm](3,16.000000)(4,16.000000)
\psline[linewidth=0.08cm](10,8.800000)(11,8.800000)
\psline[linewidth=0.08cm](17,1.600000)(18,1.600000)
\psline[linewidth=0.08cm](3,16.400000)(4,16.400000)
\psline[linewidth=0.08cm](10,9.200000)(11,9.200000)
\psline[linewidth=0.08cm](17,2.000000)(18,2.000000)
\psline[linewidth=0.08cm](3,17.000000)(4,17.000000)
\psline[linewidth=0.08cm](10,9.800000)(11,9.800000)
\psline[linewidth=0.08cm](17,2.600000)(18,2.600000)
\psline[linewidth=0.08cm](3,17.800000)(4,17.800000)
\psline[linewidth=0.08cm](10,10.600000)(11,10.600000)
\psline[linewidth=0.08cm](17,3.400000)(18,3.400000)
\psline[linewidth=0.08cm](3,18.800000)(4,18.800000)
\psline[linewidth=0.08cm](10,11.600000)(11,11.600000)
\psline[linewidth=0.08cm](17,4.400000)(18,4.400000)
\psline[linewidth=0.08cm](3,20.000000)(4,20.000000)
\psline[linewidth=0.08cm](10,12.800000)(11,12.800000)
\psline[linewidth=0.08cm](17,5.600000)(18,5.600000)
\psline[linewidth=0.08cm](3,21.400000)(4,21.400000)
\psline[linewidth=0.08cm](10,14.200000)(11,14.200000)
\psline[linewidth=0.08cm](17,7.000000)(18,7.000000)
\psline[linewidth=0.08cm](3,23.000000)(4,23.000000)
\psline[linewidth=0.08cm](10,15.800000)(11,15.800000)
\psline[linewidth=0.08cm](17,8.600000)(18,8.600000)
\psline[linewidth=0.08cm](10,15.800000)(11,15.800000)
\psline[linewidth=0.08cm,linestyle=dotted,dotsep=0.08cm](4,15.800000)(10,15.800000)
\psline[linewidth=0.08cm](17,8.600000)(18,8.600000)
\psline[linewidth=0.08cm,linestyle=dotted,dotsep=0.08cm](11,8.600000)(17,8.600000)
\psline[linewidth=0.08cm](10,17.800000)(11,17.800000)
\psline[linewidth=0.08cm,linestyle=dotted,dotsep=0.08cm](4,17.800000)(10,17.800000)
\psline[linewidth=0.08cm](17,10.600000)(18,10.600000)
\psline[linewidth=0.08cm,linestyle=dotted,dotsep=0.08cm](11,10.600000)(17,10.600000)
\psline[linewidth=0.08cm](10,20.000000)(11,20.000000)
\psline[linewidth=0.08cm,linestyle=dotted,dotsep=0.08cm](4,20.000000)(10,20.000000)
\psline[linewidth=0.08cm](17,12.800000)(18,12.800000)
\psline[linewidth=0.08cm,linestyle=dotted,dotsep=0.08cm](11,12.800000)(17,12.800000)
\psline[linewidth=0.08cm](10,23.000000)(11,23.000000)
\psline[linewidth=0.08cm,linestyle=dotted,dotsep=0.08cm](4,23.000000)(10,23.000000)
\psline[linewidth=0.08cm](17,15.800000)(18,15.800000)
\psline[linewidth=0.08cm,linestyle=dotted,dotsep=0.08cm](11,15.800000)(17,15.800000)
\psline[linewidth=0.08cm](17,15.800000)(18,15.800000)
\psline[linewidth=0.08cm,linestyle=dotted,dotsep=0.08cm](11,15.800000)(17,15.800000)
\psline[linewidth=0.08cm](17,20.000000)(18,20.000000)
\psline[linewidth=0.08cm,linestyle=dotted,dotsep=0.08cm](11,20.000000)(17,20.000000)
\psline[linewidth=0.08cm](17,23.000000)(18,23.000000)
\psline[linewidth=0.08cm,linestyle=dotted,dotsep=0.08cm](11,23.000000)(17,23.000000)
\end{pspicture}
}

%% file: fpcfigX.tex
\scalebox{1} 
{
\begin{pspicture}(3,-12.466719)(35,12.506719)
\definecolor{color602c}{rgb}{0.5019607843137255,0.5019607843137255,0.5019607843137255}
\definecolor{color0}{rgb}{1.0,0.0,0.5137254901960784}
\definecolor{color0c}{rgb}{0.7450980392156863,1.0,0.0}
\definecolor{color118b}{rgb}{0.792156862745098,0.7176470588235294,0.7176470588235294}
\psframe[linewidth=0.04,dimen=outer,fillstyle=solid,fillcolor=color118b](9.6,11.793282)(8.6,10.793282)
\psframe[linewidth=0.04,dimen=outer,fillstyle=solid,fillcolor=color118b](10.6,11.793282)(9.6,10.793282)
\psframe[linewidth=0.04,dimen=outer,fillstyle=solid,fillcolor=color118b](11.6,11.793282)(10.6,10.793282)
\psframe[linewidth=0.04,dimen=outer,fillstyle=solid,fillcolor=color118b](10.6,10.793282)(9.6,9.793282)
\psframe[linewidth=0.04,dimen=outer,fillstyle=solid,fillcolor=color118b](11.6,10.793282)(10.6,9.793282)
\psframe[linewidth=0.04,dimen=outer,fillstyle=solid,fillcolor=color118b](11.6,9.793282)(10.6,8.793282)
\psframe[linewidth=0.04,dimen=outer,fillstyle=solid,fillcolor=color118b](12.6,8.793282)(11.6,7.793281)
\psframe[linewidth=0.04,dimen=outer,fillstyle=solid,fillcolor=color118b](12.6,11.793282)(11.6,9.793282)
\rput(10.6,8.793282){\psgrid[gridwidth=0.1,subgridwidth=0.014111111,gridlabels=0.0pt,subgriddiv=1,griddots=1,unit=0.2cm,subgridcolor=color602c](0,0)(0,0)(5,5)
\psset{unit=1.0cm}}
\psframe[linewidth=0.04,dimen=outer,fillstyle=solid,fillcolor=color118b](12.6,9.793282)(11.6,8.793282)
\rput(11.6,8.793282){\psgrid[gridwidth=0.1,subgridwidth=0.014111111,gridlabels=0.0pt,subgriddiv=1,griddots=1,unit=0.2cm,subgridcolor=color602c](0,0)(0,0)(5,5)
\psset{unit=1.0cm}}
\psset{unit=1.0cm}
\psframe[linewidth=0.04,dimen=outer,fillstyle=solid,fillcolor=color118b](12.6,9.793282)(11.6,8.793282)
\rput(11.6,8.793282){\psgrid[gridwidth=0.1,subgridwidth=0.014111111,gridlabels=0.0pt,subgriddiv=1,griddots=1,unit=0.2cm,subgridcolor=color602c](0,0)(0,0)(5,5)
\psset{unit=1.0cm}}
\rput(12.6,8.793282){\psgrid[gridwidth=0.1,subgridwidth=0.014111111,gridlabels=0.0pt,subgriddiv=1,griddots=1,unit=0.2cm,subgridcolor=color602c](0,0)(0,0)(5,5)
\psset{unit=1.0cm}}
\rput(13.6,8.793282){\psgrid[gridwidth=0.1,subgridwidth=0.014111111,gridlabels=0.0pt,subgriddiv=1,griddots=1,unit=0.2cm,subgridcolor=color602c](0,0)(0,0)(5,5)
\psset{unit=1.0cm}}
\rput(12.6,7.793282){\psgrid[gridwidth=0.1,subgridwidth=0.014111111,gridlabels=0.0pt,subgriddiv=1,griddots=1,unit=0.2cm,subgridcolor=color602c](0,0)(0,0)(5,5)
\psset{unit=1.0cm}}
\rput(12.6,6.793282){\psgrid[gridwidth=0.1,subgridwidth=0.014111111,gridlabels=0.0pt,subgriddiv=1,griddots=1,unit=0.2cm,subgridcolor=color602c](0,0)(0,0)(5,5)
\psset{unit=1.0cm}}
\rput(13.6,6.793282){\psgrid[gridwidth=0.1,subgridwidth=0.014111111,gridlabels=0.0pt,subgriddiv=1,griddots=1,unit=0.2cm,subgridcolor=color602c](0,0)(0,0)(5,5)
\psset{unit=1.0cm}}

\rput(13.6,7.793282){\psgrid[gridwidth=0.1,subgridwidth=0.014111111,gridlabels=0.0pt,subgriddiv=1,griddots=1,unit=0.2cm,subgridcolor=color602c](0,0)(0,0)(5,5)
\psset{unit=1.0cm}}

\rput(13.6,5.793282){\psgrid[gridwidth=0.1,subgridwidth=0.014111111,gridlabels=0.0pt,subgriddiv=1,griddots=1,unit=0.2cm,subgridcolor=color602c](0,0)(0,0)(5,5)
\psset{unit=1.0cm}}

\psframe[linewidth=0.04,dimen=outer,fillstyle=solid,fillcolor=color118b](12.6,8.793282)(11.6,7.793281)
\rput(11.6,7.793281){\psgrid[gridwidth=0.1,subgridwidth=0.014111111,gridlabels=0.0pt,subgriddiv=1,griddots=1,unit=0.2cm,subgridcolor=color602c](0,0)(0,0)(5,5)
\psset{unit=1.0cm}}
\psframe[linewidth=0.04,dimen=outer](13.6,11.793282)(12.6,9.793282)
\psframe[linewidth=0.04,dimen=outer](14.6,11.793282)(13.6,9.793282)
\psline[linewidth=0.08cm](8.6,10.793282)(8.6,11.793282)
\usefont{T1}{ptm}{m}{n}
\rput(5,12.803281){\scalebox{2}[2]{Column}}
\rput(9.2,12.803281){\scalebox{2}[2]{$i_0$}}
\rput(23.2,12.803281){\scalebox{2}[2]{$i$}}
\rput(5.6,11.2){\scalebox{2}[2]{$\ell_0=\Delta_{k_{i_0}}^{i_0}$}}
\rput(18.6,-2.8){\scalebox{2}[2]{$\ell_0=\Delta_{k_{i_0}}^{i}$}}
\psline[linewidth=0.08cm,arrowsize=0.05291667cm 1.68,arrowlength=1.06,arrowinset=0.22]{<->}(7.6,10.793282)(7.6,11.793282)
\psline[linewidth=0.08cm,arrowsize=0.05291667cm 1.68,arrowlength=1.06,arrowinset=0.22]{<->}(20.6,-3.206718)(20.6,-2.206718)
\psline[linewidth=0.08cm,arrowsize=0.07291667cm 1.68,arrowlength=1.56,arrowinset=0.32]{<->}(23.1,7.793282)(23.1,11.793282)
\psline[linewidth=0.08cm,arrowsize=0.07291667cm 1.68,arrowlength=1.56,arrowinset=0.32]{<->}(8.1,7.793282)(8.1,11.793282)
\rput(7.4,9.7){\scalebox{2}[2]{$\Delta_{k_{i}}^{i}$}}
\rput(34,11.793282){\scalebox{2}[2]{level $0$}}
\psline[linewidth=0.08cm](31.6,11.793282)(31.6,-12.206718)
\psline[linewidth=0.08cm](31.6,-12.206718)(32.6,-12.206718)
\psline[linewidth=0.08cm](30.6,11.793282)(30.6,-11.206718)
\psline[linewidth=0.08cm](30.6,-11.206718)(32.6,-11.206718)
\psline[linewidth=0.08cm](29.6,11.793282)(29.6,-10.206718)
\psline[linewidth=0.08cm](29.6,-10.206718)(32.6,-10.206718)
\psline[linewidth=0.08cm](28.6,11.793282)(28.6,-9.206718)
\psline[linewidth=0.08cm](28.6,-9.206718)(31.6,-9.206718)
\psline[linewidth=0.08cm](27.6,11.793282)(27.6,-8.206718)
\psline[linewidth=0.08cm](27.6,-8.206718)(30.6,-8.206718)
\psline[linewidth=0.08cm](26.6,11.793282)(26.6,-7.206718)
\psline[linewidth=0.08cm](26.6,-7.206718)(29.6,-7.206718)
\psline[linewidth=0.08cm](25.6,11.793282)(25.6,-6.206718)
\psline[linewidth=0.08cm](25.6,-6.206718)(28.6,-6.206718)
\psline[linewidth=0.08cm](24.6,11.793282)(24.6,-5.206718)
\psline[linewidth=0.08cm](24.6,-5.206718)(27.6,-5.206718)
\psline[linewidth=0.08cm](23.6,11.793282)(23.6,-4.206718)
\psline[linewidth=0.08cm](23.6,-4.206718)(26.6,-4.206718)
\psline[linewidth=0.08cm](22.6,11.793282)(22.6,-3.206718)
\psline[linewidth=0.08cm](22.6,-3.206718)(25.6,-3.206718)
\psline[linewidth=0.08cm](21.6,11.793282)(21.6,-2.206718)
\psline[linewidth=0.08cm](21.6,-2.206718)(24.6,-2.206718)
\psline[linewidth=0.08cm](20.6,11.793282)(20.6,-1.206718)
\psline[linewidth=0.08cm](20.6,-1.206718)(23.6,-1.206718)
\psline[linewidth=0.08cm](19.6,11.793282)(19.6,-.206718)
\psline[linewidth=0.08cm](19.6,-.206718)(22.6,-.206718)
\psline[linewidth=0.08cm](18.6,11.793282)(18.6,.793282)
\psline[linewidth=0.08cm](18.6,.793282)(21.6,.793282)
\psline[linewidth=0.08cm](17.6,11.793282)(17.6,1.793282)
\psline[linewidth=0.08cm](17.6,1.793282)(20.6,1.793282)
\psline[linewidth=0.08cm](16.6,11.793282)(16.6,2.793282)
\psline[linewidth=0.08cm](16.6,2.793282)(19.6,2.793282)
\psline[linewidth=0.08cm](15.6,11.793282)(15.6,3.793282)
\psline[linewidth=0.08cm](15.6,3.793282)(18.6,3.793282)
\psline[linewidth=0.08cm](14.6,11.793282)(14.6,4.793282)
\psline[linewidth=0.08cm](14.6,4.793282)(17.6,4.793282)
\psline[linewidth=0.08cm](13.6,11.793282)(13.6,5.793282)
\psline[linewidth=0.08cm](13.6,5.793282)(16.6,5.793282)
\psline[linewidth=0.08cm](12.6,11.793282)(12.6,6.793282)
\psline[linewidth=0.08cm](12.6,6.793282)(15.6,6.793282)
\psline[linewidth=0.08cm](11.6,11.793282)(11.6,7.793282)
\psline[linewidth=0.08cm](11.6,7.793282)(14.6,7.793282)
\psline[linewidth=0.08cm](10.6,11.793282)(10.6,8.793282)
\psline[linewidth=0.08cm](10.6,8.793282)(13.6,8.793282)
\psline[linewidth=0.08cm](9.6,11.793282)(9.6,9.793282)
\psline[linewidth=0.08cm](9.6,9.793282)(12.6,9.793282)
\rput(12.1,12.303281){$\mfnn$}
\rput(14.1,12.303281){$\mfnn$}
\rput(16.1,12.303281){$\mfnn$}
\rput(17.1,12.303281){$\mfnn$}
\rput(18.1,12.303281){$\mfnn$}
\rput(20.1,12.303281){$\mfnn$}
\rput(21.1,12.303281){$\mfnn$}
\rput(22.1,12.303281){$\mfnn$}
\rput(24.1,12.303281){$\mfnn$}
\rput(25.1,12.303281){$\mfnn$}
\rput(26.1,12.303281){$\mfnn$}
\rput(27.1,12.303281){$\mfnn$}
\rput(28.1,12.303281){$\mfnn$}
\rput(29.1,12.303281){$\mfnn$}
\rput(30.1,12.303281){$\mfnn$}
\rput(32.1,12.303281){$\mfnn$}
\psline[linewidth=0.08cm](9.6,9.793282)(16.6,9.793282)
\psline[linewidth=0.08cm](13.6,5.793282)(20.6,5.793282)
\psline[linewidth=0.08cm](17.6,1.793282)(24.6,1.793282)
\psline[linewidth=0.08cm](21.6,-2.206718)(28.6,-2.206718)
\psline[linewidth=0.08cm](25.6,-6.206718)(32.6,-6.206718)
\psline[linewidth=0.08cm](11.6,7.793282)(26.6,7.793282)
\psline[linewidth=0.08cm](15.6,3.793282)(32.6,3.793282)
\psline[linewidth=0.08cm](19.6,-.206718)(32.6,-.206718)
\psline[linewidth=0.08cm](23.6,-4.206718)(32.6,-4.206718)
\psline[linewidth=0.08cm](27.6,-8.206718)(32.6,-8.206718)
\psline[linewidth=0.08cm](8.6,11.793282)(32.6,11.793282)
\psline[linewidth=0.08cm](8.6,10.793282)(9.6,10.793282)
\psline[linewidth=0.08cm](32.6,11.793282)(32.6,-12.206718)
\psline[linewidth=0.08cm](31.6,11.793282)(31.6,-12.206718)
\end{pspicture}
}